\documentclass[12pt]{article}
\usepackage[utf8]{inputenc}

\usepackage{amsmath}
\usepackage{amsthm}
\usepackage{amsfonts}
\usepackage{amssymb, dsfont}
\usepackage{txfonts}
\usepackage{mathtools}
\usepackage[shortcuts]{extdash}
\usepackage[dvipsnames]{xcolor}

\usepackage{hyperref}
\usepackage{imakeidx}
\usepackage{amscd}
\usepackage[all,cmtip]{xy}
\usepackage{pb-diagram}
\usepackage{tikz-cd}

\usepackage{enumitem}

\newtheorem{theorem}{Theorem}
\newtheorem{lemma}{Lemma}
\newtheorem{proposition}{Proposition}

\newtheorem{corollary}{Corollary}
\newtheorem{problem}{Problem}
\newtheorem{remark}{Remark}
\newtheorem{notation}{Notation}
\newtheorem{definition}{Definition}

\theoremstyle{definition}
\newtheorem{example}{Example}

\def\Z{{\mathds Z}}
\def\N{{\mathds N}}
\def\Q{{\mathds Q}}
\def\MA{{\mathbb A}}
\def\MB{{\mathbb B}}
\def\MC{{\mathbb C}}

\def\MR{{\mathds R}}

\def\PLS{{\mathcal{PLS}}}
\def\PDS{{\mathcal{PDS}}}
\def\Aint{{\mathcal{O}}}

\def\K{{\mathcal{K}}}

\def\AC{\mathcal {AC}}

\def\Codes{\mathrm{Int}}
\def\Id{{\mathbf{Id}}}

\newcommand{\nsZ}{\widetilde{\mathds Z}}
\newcommand{\nsA}{\widetilde{\mathbb A}}
\newcommand{\nsB}{\widetilde{\mathbb B}}

\newcommand{\nsuA}{\widetilde{A}}

\newcommand{\Aut}{\mathrm{Aut}}
\newcommand{\Orb}{\mathrm{Orb}}
\newcommand{\Th}{\mathrm{Th}}
\newcommand{\Mod}{\mathrm{Mod}}
\newcommand{\GL}{\mathrm{GL}}
\newcommand{\SL}{\mathrm{SL}}
\newcommand{\UT}{\mathrm{UT}}

\newcommand{\inj}{\mathrm{inj}}

\newcommand{\param}{\mathrm{par}}
\newcommand{\BS}{\mathrm{BS}}
\newcommand{\reg}{\mathrm{reg}}

\newcommand{\abs}{\mathrm{abs}}
\newcommand{\self}{\mathrm{self}}
\newcommand{\CG}{\mathrm{C}}
\newcommand{\ZG}{\mathrm{Z}}
\newcommand{\ring}{\mathrm{ring}}
\newcommand{\group}{\mathrm{group}}
\newcommand{\pre}{\mathrm{pre}}
\newcommand{\tp}{\mathrm{tp}}

{}
{}
{}
{}

\title{Theory of Interpretations I. Foundations\footnote{The research was supported in accordance with the state task of the IM SB RAS, project FWNF-2026-0033.}}
\date{}

\author{Evelina Danyarova\footnote{Sobolev Institute of Mathematics} \quad and \quad Alexei Myasnikov\footnote{Stevens Institute of Technology}} 

\begin{document}

\maketitle

\begin{flushright}
In memory of B.\,I.\,Plotkin
\end{flushright}

\begin{abstract} 
This is the first paper in a series in which we lay down the foundations of the theory of interpretations. We systematically study different types of interpretations and their properties. Some of these interpretations are known, while others are new. Each of them serves a different purpose. In the last section, we describe applications of interpretations to Diophantine problems, first-order classification, isotypeness, definability of structures by types, elimination of imaginaries, richness, logical categories, and bi-interpretations with $\Z$ or $\N$. Additionally, throughout the text, we pose some open questions that naturally arise in this context and provide the most typical examples, usually from algebra.
The current literature is plagued by discrepancies and inconsistencies in definitions, concepts, and fundamental applications of interpretations. To address this, we thoroughly examine various principal notions, definitions, and arguments, bringing order to the existing theory. Simultaneously, we develop several key concepts, such as regular interpretations, regular bi-interpretations, and invertible interpretations, and outline their main applications.

\end{abstract}

\tableofcontents

\section{Introduction}

In this paper, we lay down the foundations of the theory of interpretations. We systematically study different types of interpretations and their properties. Some of these interpretations are known, while others are new. Each of them serves a different purpose. Our main focus is twofold: on the one hand, we study particular properties of interpretations that are required for specific applications to various problems from algebra, geometry, topology, model theory, and logic; on the other hand, we introduce general notions, techniques, and methods that are useful for developing the theory of interpretations itself.  In fact, we aim to outline a framework for the emerging theory that encompasses its own notions and results, techniques, research methods, applications, and general viewpoints on various phenomena that occur in this area.  

This is the first paper in a series; therefore, we discuss in detail various principal notions, provide precise definitions, prove some basic results, and offer the most typical examples, typically from algebra. 

Until recently, the prevalent viewpoint in model theory was that it deals with properties that are transferred via interpretation; hence, the interpretation itself is rarely an object of study, only the fact that it exists. However, it has become apparent that different types of interpretations convey different properties. Interpretations with parameters are particularly effective for undecidability and stability, yet they fall short when it comes to elementary equivalence, elimination of imaginaries, or richness. The distinctions between these various types of interpretations are challenging to articulate using commonly used definitions. Therefore, it is necessary to develop specific notation and terminology to formulate precise definitions, and more importantly, to facilitate the construction of proofs. Standard interpretations, whether absolute or with parameters, are well understood and require no further explanation. However, the role of regular interpretations, also known as interpretations with definable parameters, remains insufficiently clarified. The situation becomes particularly challenging with bi-interpretations. They entail stronger interactions between algebraic structures, their definitions are more involved and complex, and they provide more variety, describing different properties of structures. The best introduction to bi-interpretations is in Hodges' book~\cite{Hodges}, which describes absolute interpretations in detail. It also notes that by adding constants to the language, one can obtain bi-interpretations with parameters. It is important to note that absolute interpretations are very rare. Therefore, one typically deals with various types of bi-interpretations with parameters, among which regular bi-interpretations are the most advantageous. It is challenging to work with such bi-interpretations, particularly when keeping the parameters as constants in the language; therefore, there is a need for precise terminology. To illustrate this point, the notion of bi-interpretability in Hodges' book and the paper~\cite{KhMS} is referred to in this paper as weak bi-interpretability, whereas the bi-interpretability used in the influential paper~\cite{AKNS} is different; we term it here as the strong one. They describe different relations between structures and transfer different properties, with the strong bi-interpretability transferring more properties than the weak one. Sometimes, it is not even clear which type of bi-interpretability is used in the literature.  

One of the main aims of this paper is to thoroughly describe various types of interpretability, particularly bi-interpretability, and to introduce the necessary terminology. While the exposition may appear pedantic at times, this seems unavoidable. 

We also describe some properties that these interpretations transfer and introduce general techniques designed for applications of interpretations to particular problems. For instance, the classical Diophantine problem for a structure $\MA$ asks if there exists an algorithm that, given a finite system of equations with coefficients in $\MA$, decides whether or not this system has a solution in $\MA$. Typically, undecidability of the Diophantine problem in $\MA$ is obtained by providing a computable reduction of the undecidable Diophantine problem of some other structure $\MB$ to the Diophantine problem in $\MA$. In most cases, the structure $\MB$ is the ring of integers $\Z$, or its variations, whose Diophantine problem is undecidable~--- the famous solution to the Tenth Hilbert Problem. But the reductions themselves are very different; they heavily depend on the structure $\MA$. Nevertheless, analysis of these reductions shows that they almost always come from a very particular type of interpretation, so-called interpretations by equations or, more generally, Diophantine interpretations. This means that there is a general uniform method to prove the undecidability of the Diophantine problems, which greatly streamlines and simplifies the proofs and does not depend on the particular structure $\MA$. We describe this and some other such methods in the paper, with more detailed discussions to be presented in future papers in the series.

At present, it becomes evident that interpretations serve purposes beyond merely transferring properties. We give two examples. The first one concerns the fundamental problem of first-order classification of structures due to Tarski. This problem requires describing the algebraic structure of all models of the first-order theory $\Th(\MA)$ of a given structure $\MA$. 
It turns out that if $\MA$ is bi-interpretable with $\Z$, then a structure $\nsA$ is elementarily equivalent to $\MA$ if and only if $\nsA \simeq \Gamma(\nsZ)$, where $\nsZ \equiv \Z$ and $\Gamma$ is the interpretation arising from the bi-interpretability of $\MA$ in $\Z$. Thus, $\Gamma$ provides the algebraic structure of all models of $\Th(\MA)$. For instance, the group $\SL_n(\Z)$, where $n \geqslant 3$, is bi-interpretable with $\Z$~\cite{MS}. Hence, a group $H$ is first-order equivalent to $\SL_n(\Z)$ if and only if $H \simeq \SL_n(\nsZ)$ for some ring $\nsZ$ such that $\nsZ \equiv \Z$. This method works in a much more general context; one does not even need bi-interpretability here; it suffices to have an invertible interpretation of $\MA$ in $\Z$. We discuss this in Subsections~\ref{subsec:ns} and~\ref{subsec:Z}.

The second one is related to classical results like fundamental theorems of geometry (projective, affine, symplectic, etc.), rigidity theorems, theorems on abstract isomorphisms in group theory and topology, and on categorical equivalence of algebraic structures. Quite often, interpretations hidden somewhere behind these results play a key role in the arguments. Furthermore, they offer a more precise description of the interactions between the structures involved, thus enhancing the results themselves.
 See~\cite{Berger} for the fundamental theorem of projective geometry, \cite{ALM} for first-order rigidity of irreducible non-uniform higher-rank lattices, \cite{BuninaGv} for abstract isomorphisms of Chevalley groups, and \cite{ABM} for Mal'cev correspondence between nilpotent $k$\=/groups and nilpotent Lie algebras over a field $k$ of characteristic zero.

In the theory of interpretations, fundamental constructions and basic theorems sometimes require a categorical approach. This view treats interpretations as functors, particularly when examining properties like compositions of interpretations, their homotopies, invertible interpretations, etc.   Following pioneering ideas of Boris Plotkin, one can consider various ``logical categories'' related to the structure $\MA$~\cite{Plotkin6}. Plotkin viewed these logical categories as natural generalizations of the categories of algebraic sets over $\MA$ from algebraic geometry, where the algebraic sets are replaced by the logical sets that are definable in $\MA$ by sets of formulas in the language of $\MA$. One can go even further and consider projective logical geometries over $\MA$, where projective logical sets are the quotients of logical sets modulo equivalence relations definable in $\MA$. It turns out that these logical categories, especially the categories of projective logical geometry, provide useful instruments for studying categorical properties of interpretations. This lies outside the scope of this paper; we only briefly mention it in Subsection~\ref{subsec:cat} in the context of bi-interpretations and will address it in detail in the next article of the series.

Now we say a few words on the principal types of interpretations and how they come about. The notion of interpretability was introduced by Tarski; initially, it was used mainly as a tool to prove the undecidability of theories. This approach is based on the following result: if a structure $\MA$ is interpretable with parameters in a structure $\MB$ then for every sentence $\psi$ in the language of $\MA$ there is a sentence $\psi^\ast$ of the language of $\MB$ with parameters from $\MB$ such that $\MA {\models} \psi \Longleftrightarrow \MB {\models }\psi^\ast$. Moreover, if the signature of $\MA$ is finite (or, more generally, the interpretation is effective, see below), then the formula $\psi^\ast$ can be effectively constructed from $\psi$; also, if the interpretation is absolute (no parameters), then $\psi^\ast$ does not contain parameters as well. Note that the G\"odel's first incompleteness theorem implies that the first-order theory of the arithmetic $\N$ is undecidable. Hence, if $\N$ is absolutely interpretable in $\MB$ then the map $\psi \to \psi^\ast$ gives a computable reduction of $\Th(\N)$ to $\Th(\MB)$, so $\Th(\MB)$ is undecidable. The only inconvenience here is that interpretations often come with parameters. To this end, Tarski, Robinson, and Mostowski proved in~\cite{TRM} that the theory of $\N$ not only is undecidable but it is hereditarily undecidable (see Subsection~\ref{subsec:undecidability} below), which leads to the following much stronger result: if $\N$ (or $\Z$) is interpretable with parameters in $\MB$ then $\Th(\MB)$ is undecidable. The book~\cite{TRM} together with a survey~\cite{Ershov-Lavrov} was very influential in the development of the theory of interpretations at the beginning, where the focus was on the decidability questions. Over the years, this research has produced a lot of interesting interpretability results in algebra. 

The next principal step in the application of interpretability comes from Mal'cev, who demonstrated how interpretations can describe groups that are elementarily equivalent to each other. The starting point here is clear: if $\MA$ is absolutely interpretable in $\MB$ by a set of formulas $\Gamma$, so $\MA \simeq \Gamma(\MB)$, and if $\nsB \equiv \MB$, then the structure $\nsA \simeq \Gamma(\nsB)$, which is interpretable in $\nsB$ by $\Gamma$, satisfies $\nsA \equiv \MA$. However, if the interpretation $\Gamma$ contains parameters, then $\nsA$ and $\MA$ may not be elementarily equivalent. For example, for the field of real numbers $\MR$ and elements $a, b \in \MR$ with $a \neq b$ one has $(\MR,a) \not \equiv (\MR,b)$ (here for $c \in \MR$ $(\MR,c)$ denotes the field  $\MR$ with an extra constant $c$ in the language). In his seminal work~\cite{Malcev-matrices} on the characterization of elementarily equivalent classical matrix groups, Mal'cev used interpretability with parameters, which are taken from a definable set. In this case, similar to the case of absolute interpretability, if $\MA$ is interpretable in $\MB$ with a definable set of parameters, one can still construct a map $\psi \to \psi^\ast$ which gives a reduction of the theory $\Th(\MA)$ to $\Th(\MB)$. Following~\cite{Myas-x, Miasn, Miasn1}, we refer to this type of interpretability as {\em regular interpretability} (Hodges, in his book~\cite{Hodges}, uses the term {\em interpretability with definable parameters}, but we feel that regular interpretability is shorter and avoids overuse of the term ``definable''). Thus, the notion of regular interpretability is somewhat between absolute interpretability and interpretability with parameters. The former fits extremely well for questions related to first-order equivalence but occurs rather rarely. Interpretability with parameters is more common, but the parameters are usually not in the language of $\MB$, which presents an obstacle for elementary equivalence. Regular interpretability combines the useful properties of both: it can be used in first-order classification problems in a way similar to absolute interpretability and occurs almost as often as interpretability with parameters. On the downside, to get regular interpretability, one needs to go deeper into model-theoretic properties of the structures. Nowadays, regular interpretability gives a powerful tool in studying elementarily equivalent structures.

In general, many properties are preserved under interpretations: $\lambda$\=/stability, $\lambda$\=/saturation, $\omega$\=/categoricity, independence property, and finite Morley rank, to mention a few (see below). Nevertheless, even if two structures $\MA$ and $\MB$ are absolutely interpretable in each other, one can't expect that they are ``equivalent from the model-theoretic viewpoint'' (or have the same model-theoretic properties). To achieve this, one needs a much stronger notion called bi-interpretation. Absolute bi-interpretations were introduced by Ahlbrandt and Ziegler in~\cite{AhZ}. In the book~\cite{Hodges}, Hodges laid out the fundamentals on bi-interpretations and showed interesting applications. Later, Kh\'elif  and Nies  linked the QFA property of a structure $\MA$ with bi-interpretability of $\MA$ and $\N$ or $\Z$  (see~\cite{Khelif, Nies1,Nies2}). Recall that a finitely generated structure $\MA$ is QFA (quasi-finitely axiomatizable) if there is a sentence $\psi$ in the language of $\MA$ such that for any finitely generated structure $\MB$ in the language of $\MA$, if $\MB \models \psi$ then $\MB$ and  $\MA$ are isomorphic. They also proved a similar result for the properties of primality and homogeneity. 

For us here, an important milestone was the paper~\cite{KhMS} by Kharlampovich, Myasnikov, and Sohrabi on rich structures, where different types of bi-interpretability, and especially regular bi-interpretability, were introduced and studied. Intuitively, a structure $\MA$ is \emph{rich} if every formula of the weak second-order logic with variables for elements in $\MA$ is equivalent in $\MA$ to some first-order formula on the same variables. The initial result here is that if a structure $\MA$ is absolutely bi-interpretable with a rich structure $\MB$, then $\MA$ is also rich. But again, absolute bi-interpretability of structures $\MA$ and $\MB$ is very rare (for example, it implies that the automorphism groups of $\MA$ and $\MB$ have to be isomorphic). Fortunately, the regular bi-interpretability (more precisely, its variation, which is called {\em regular half-absolute bi-interpretability}) strikes again and gives the required results. In particular, structures regularly bi-interpretable with $\N$ are rich~\cite{KhMS}. 

In this paper, following~\cite{KhMS}, we develop further various types of bi-interpretability. We believe that the regular strong bi-interpretations are the most useful and applicable; we study them in detail. Invertible interpretations are another new type of interpretation that we introduce here; in many cases, they are as useful as bi-interpretations, but more common.

In the last section, we discuss applications of interpretations to Diophantine problems, the first-order classification, isotypeness, and definability of structures by types, elimination of imaginaries, richness, logical categories, and bi-interpretations with $\Z$ or $\N$. To illustrate the concepts and techniques, we provide typical examples and fundamental results. Additionally, throughout the text, we pose some open questions that naturally arise in this context.

{\bf Acknowledgments.} The authors express their sincere gratitude to the organizers and participants of the International Workshop ``Algebraic Geometry and Model Theory of Groups II'', held at Ivane Javakhishvili Tbilisi State University, Tbilisi, Georgia, in August--September 2024. In particular, discussions with colleagues at this conference~--- especially Pavel Gvozdevsky~--- were invaluable for the preparation of this article.

\section{Interpretations and their codes}

In this section, we discuss various types of interpretations of structures. Some of them are well-known, others are new, and they each serve different purposes. 

Interpretations provide very precise descriptions of some intrinsic relationships between algebraic structures. These descriptions are given by formulas of first-order logic in a specific manner, so they are very precise and unambiguous. When possible, we follow the Hodges'~\cite[Section~5.3]{Hodges} and Marker's books~\cite[Section~1.3]{Marker}.   

\subsection{Languages and algebraic structures}

Let $L$ be a language (also called a signature) with a set of functional (operational) symbols $\{f, \ldots\}$ along with their arities $n_f \in \N$, a set of constant symbols $\{c, \ldots\}$ ($n_c=0$ for every constant symbol) and a set of relation (or predicate) symbols $\{R, \ldots \}$ along with their arities $n_R \in \N$. We write $f(x_1, \ldots,x_n)$ or $R(x_1,\ldots,x_n)$ to show that $n_f = n$ or $n_R=n$. The standard language of groups $L_\group=\{\,\cdot\,,\,^{-1},e\}$ includes the symbol $\cdot$ for the binary operation of multiplication, the symbol $^{-1}$ for the unary operation of inversion, and the symbol $e$ for the group identity; the standard language of unitary rings is $L_\ring=\{+,\,\times\,,0,1\}$. An interpretation of a constant symbol $c$ in a set $A$ is an element $c^A\in A$, of a functional symbol $f$ is a function $f^A\colon A^{n_f}\to A$, and of a predicate symbol $R$ is a set $R^A\subseteq A^{n_R}$.

A countable first-order language $L = \{f_1, \ldots, R_1, \ldots, c_1, \ldots\}$, where the sets of function, predicate, and constant symbols may be finite or even empty, is termed as {\em computable} if the functions $i \to n_{f_i}$ and $i \to n_{R_i}$, defining the arities of $f_i$ and $R_i$ respectively, are computable. For a computable language $L$, there is a computable G\"odel numbering for all first-order formulas in $L$, using variables from a countable set $X$.

An algebraic structure in the language $L$ (a $L$\=/structure) with the underlying set $A$ is denoted by $\MA = \langle A; L\rangle$, or by $\MA = \langle A; f, \ldots, R,\ldots, c, \ldots \rangle$, or by $\MA = \langle A; f^A, \ldots, R^A,\ldots, c^A, \ldots \rangle$. For a structure $\MA$ we denote by $L(\MA)$ the language $L$ of $\MA$.

We usually denote variables by small letters $x,y,z, a,b, u,v, \ldots$, while the same symbols with bars $\bar x, \bar y,  \ldots$ denote tuples of the corresponding distinct variables, say $\bar x = (x_1, \ldots,x_n)$, and $\bar{\bar x}$ is a tuple of tuples $\bar{\bar x}=(\bar x_1,\ldots,\bar x_m)$, where $|\bar x_i|=n$. Also  $\forall\, \bar x$ denotes $\forall\, x_1 \ldots \forall \,x_n$ and $\exists\, \bar x$ denotes $\exists \,x_1 \ldots \exists \,x_n$,  provided $\bar x = (x_1, \ldots,x_n)$. For any map $h\colon X\to Y$ and any tuple $\bar a=(a_1,\ldots,a_n)$ from $X^n$ (or similarly, tuple $\bar x$ of variables) we denote by $h(\bar a)$ the tuple $(h(a_1),\ldots,h(a_n))$. For an equivalence relation $\sim$ on $X^n$ by $X^n/{\sim}$ we denote the set of all equivalence classes of $X^n$ (the quotient of $X^n$ by $\sim$). If  $\bar{\bar a}$ is a tuple $(\bar a_1,\ldots,\bar a_m)$ of tuples $\bar a_i\in X^n$, then $\bar{\bar a}/{\sim}$ is the tuple of equivalence classes  $(\bar a_1/{\sim},\ldots,\bar a_m/{\sim})$.

Given a structure $\MA$ by a complete type in free variables $x_1, \ldots,x_n$  relative to the theory $\Th(\MA)$, we understand a maximal set of formulas in variables $x_1, \ldots,x_n$ that is locally satisfiable in $\MA$. We denote such types by $\tp(x_1, \ldots,x_n)$.  For any tuple of elements $\bar a=(a_1,\ldots,a_n)\in \MA$ by $\tp(\bar a)$ we denote the complete type of $\bar a$, i.\,e., the set of all formulas $\phi(x_1, \ldots,x_n)$ in the language $L(\MA)$ such that $\MA \models \phi(a_1, \ldots,a_n)$. By $\Orb(\bar a)$ we denote the automorphic orbit of $\bar a$ in $\MA$. Clearly, for any tuples $\bar a,\bar b\in A^n$ if $\bar b\in \Orb(\bar a)$ then $\tp(\bar a)=\tp(\bar b)$.  

\subsection{Definable sets}\label{subsec:def_sets}

Let $\MB = \langle B; L(\MB)\rangle$ be an algebraic structure. A subset $X \subseteq B^n$ is called {\em definable} in $\MB$ if there is a first-order formula $\psi(x_1,\ldots,x_n,y_1,\ldots,y_k)$ in $L(\MB)$ and a tuple $\bar{p}=(p_1,\ldots,p_k)$ of elements from $B$ such that 
$$
X = \{(b_1,\ldots,b_n) \in B^n \mid \MB \models \psi(b_1, \ldots,b_n,p_1,\ldots,p_k)\}.
$$ 
The set $X$ is also called {\em $P$\=/definable}, where $\{p_1,\ldots,p_k\}\subseteq P\subseteq B$, and denoted by $\psi(\MB,\bar p)$ or $\psi(\MB_P)$, or $\psi(B^n,\bar p)$. Here, the algebraic structure $\MB_P=\langle B; L(\MB)\cup P\rangle$ is obtained from $\MB$ by adding new constant symbols from $P$ into the language.  If $P=\emptyset$ then $X$ is called {\em $0$\=/definable} or {\em absolutely definable}. A set $X\subseteq B^n$ is definable in $\MB$ if and only if there exists a finite set of parameters $P\subseteq B$ such that $X$ is absolutely definable in $\MB_P$. 

Recall that a formula $\phi(x_1,\ldots,x_n)$ in $L(\MB)$ is called {\em primitive positive} or {\em Diophantine} if 
$$
\phi(x_1,\ldots,x_n) = \exists\, y_1\ldots \exists \, y_k\; \psi(x_1,\ldots,x_n,y_1,\ldots,y_k),
$$
where $\psi$ is a conjunction of atomic formulas. Sets definable by Diophantine formulas  are called {\em Diophantine}.

\begin{definition}\label{def:definable_map}
  Let $X\subseteq B^n$, $Y\subseteq B^s$ be sets and $\sim_X$, $\sim_Y$ be equivalence relations on $X$, $Y$, respectively. A map $g\colon (X/{\sim_X})^m\to Y/{\sim_Y}$ is termed {\em definable} in $\MB$, if the preimage in $\MB$ of the graph of $g$ is definable in $\MB$, i.\,e., the set 
\begin{multline*}
\{(b^1_1,\ldots,b^1_n,\:\ldots,\: b^m_1,\ldots,b^m_n,\: c_1,\ldots, c_s) \in B^{nm+s}\mid  \\ 
g((b^1_1,\ldots,b^1_n)/{\sim_X},\:\ldots,\: (b^m_1,\ldots,b^m_n)/{\sim_X})=(c_1,\ldots,c_s)/{
\sim_Y},\\
(b^1_1,\ldots,b^1_n),\ldots, (b^m_1,\ldots,b^m_n)\in X, \: (c_1,\ldots,c_s)\in Y\}
\end{multline*}
is definable in $\MB$.  Note that here $g$ may also be a function of arity $0$ (i.\,e., a constant or a point $\bar y/{\sim_Y}\in Y/{\sim_Y}$).  Similarly, a relation $Q$ of arity $m$ on the quotient set $X/{\sim_X}$ is {\em definable} in $\MB$, if the {\em preimage in $\MB$ of its graph}, i.\,e., the full preimage of the set $Q$ in $B^{nm}$, is definable in $\MB$. Correspondingly, $g$ and $Q$ are $0$\=/{\em definable} in $\MB$, if their preimages  are $0$\=/definable in $\MB$.
\end{definition}

It is obvious that the preimage in $\MB$ of the graph of $g$ coincides with the preimage in $\MB$ of the graph of the map $\bar g\colon X^m\to Y/{\sim_Y}$, where $\bar g(\bar {\bar x})= g(\bar{\bar x}/{\sim_X})$, $\bar {\bar x}\in X^m$.

Usually, when definable maps $g$ or predicates $Q$ occur in this paper, the relations $\sim_X$ and $\sim_Y$ are also definable, but we do not require this in the definitions.

By default, we consider the sets $\N$, $\Z$, $\Q$, $\MR$ and $\MC$ as algebraic structures in the language $\{+,\times,0,1\}$. Note that definable sets in arithmetic $\N$, the so-called arithmetic sets, are well-studied. In particular, it is known that every computable and computably enumerable set in $\N$ is definable. The same holds for the ring of integers $\Z$. Furthermore, every element of $\N$ (or $\Z$, or $\Q$) is absolutely definable in $\N$ ($\Z$, $\Q$), so every definable set in $\N$ ($\Z$, $\Q$) is absolutely definable.

Recall that an algebraic structure $\MB$ is called {\em rigid} if the group of automorphisms $\Aut(\MB)$ of $\MB$ is trivial. For example, $\N,\Z, \Q, \MR$ are rigid,  as well as, $\MB_B$ for any $\MB=\langle B;L(\MB)\rangle$.

\subsection{Interpretability with parameters}\label{subsec:def_int}

Intuitively, {\em interpretability} is a particular way of describing by formulas an algebraic structure $\MA$ in an algebraic structure $\MB$. 

\begin{definition}  \label{de:interpretation}
An algebraic  structure $\MA = \langle A;f,\ldots,R,\ldots,c,\ldots\rangle$ is {\em interpretable} with parameters in an algebraic structure $\MB=\langle B;L(\MB)\rangle$ if the following conditions hold:
\begin{enumerate}[label=\arabic*)]
\item there is a finite subset $P\subseteq B$,
\item for some integer $n$ there is a subset $A^\ast  \subseteq B^n$ $P$\=/definable in $\MB$, 
\item there is an equivalence relation $\sim$ on $A^\ast$ $P$\=/definable in $\MB$, 
\item there are interpretations $f^A, \ldots$, $R^A,\ldots$, $c^A, \ldots$ of the symbols $f, \ldots, R,\ldots, c, \ldots$ on the quotient set $A^\ast /{\sim}$, all $P$\=/definable in $\MB$ (see Definition~\ref{def:definable_map}),\item the structure $\MA^\ast=\langle A^\ast /{\sim}; f^A, \ldots,R^A,\ldots, c^A, \ldots \rangle$ is $L(\MA)$\=/isomorphic to $\MA$.
\end{enumerate}
\end{definition}

 In Definition~\ref{de:interpretation}  we allow $P=\emptyset$, in which case the interpretation  of $\MA$ in $\MB$ is called {\em absolute} or {\em $0$\=/interpretation}, or {\em interpretation without parameters}. Thus, we may assume that absolute interpretations are a particular case of interpretations with parameters. In all other similar situations where defining a notion $C$ that could be absolute, we may also refer to it as $0$\=/$C$, or simply as $C$ without parameters. Note that we have already done it in  Definition~\ref{def:definable_map}. Furthermore, following Marker's book~\cite{Marker}, we sometimes simply write ``interpretation'' instead of ``interpretation with parameters'', so by default our interpretation is interpretation with parameters. We write $\MA\rightsquigarrow\MB$ if $\MA$ is interpretable in $\MB$ (with or without parameters).

If the equivalence relation $\sim$ in Definition~\ref{de:interpretation} is the identity relation on $A^\ast$, we follow Hodges' book~\cite{Hodges} in saying that $\MA$ is \emph{injectively} interpretable in $\MB$ (with or without parameters).  Injective interpretability $\MA\rightsquigarrow\MB$ is sometimes called {\em definability} or {\em definable interpretability} of $\MA$ in $\MB$. Any interpretation $\MA\rightsquigarrow\MA$ is called {\em self-interpretation}.

\begin{example}
If $N$ is a normal $0$\=/definable subgroup of a group $G$ (in the language $\{\cdot,^{-1}, e\}$), then the equivalence relation $x \sim y$ on $G$ given by $xN = yN$ is $0$\=/definable in $G$. It is easy to see that the  preimages in $G$ of the group operations $\cdot,^{-1}, e$ on $G/N$, are $0$\=/definable in $G$. This shows that the quotient group $G/N$ is $0$\=/interpretable in $G$.  Similarly, if $N$ is a normal subgroup of $G$ definable with parameters in $G$, then the quotient group $G/N$ is interpretable with parameters in $G$.  
\end{example}

An embedding $h\colon\MA\to\MA^\prime$ of $L$\=/structures gives an (absolute) injective interpretability of $\MA$ in $\MA^\prime$ if and only if the set $h(A)$ is (absolutely) definable in $\MA^\prime$. In particular, every isomorphism $h\colon\MA\to\MA^\prime$ gives an absolute injective interpretability. 

\begin{example}\label{ex:NZ}
The semiring of natural numbers $\N$ is injectively $0$\=/interpretable in the ring of integers $\Z$. Indeed, by Lagrange’s Four Squares Theorem, every nonnegative integer is a sum of four squares, so the formula 
$$
\exists\, x_1\:\exists\, x_2\:\exists\, x_3 \:\exists\, x_4\;(y = x_1^2 +x_2^2 +x_3^2 +x_4^2)
$$
defines $\N$ in $\Z$. It is also easy to see that the ring $\Z$ is injectively $0$\=/interpretable in the semiring $\N$.
\end{example}

\begin{example}
The ring of integers $\Z$ is injectively $0$\=/interpretable in the ring of rationals $\Q$. Due to J.\,Robinson~\cite{J-Robinson}, the formula 
$$
\varphi (x)= \forall\:y\,\forall\:z\; ([\Phi(y,z,0)\:\wedge\:(\forall\: w\;(\Phi(y,z,w)\longrightarrow\Phi(y,z,w+1)))]\longrightarrow\Phi(y,z,x)),
$$
where $\Phi(y,z,t)=\exists\:a\,\exists\:b\,\exists\:c\;(yzt^2+2=a^2+yb^2-zc^2)$,
defines the set $\Z$ in $\Q$.
\end{example}

\begin{example}
The following is a simple $0$\=/interpretation of  $\Q$ in $\Z$. Here we associate a fraction $m/n$ with a pair of numbers $(m,n)$, $n\ne 0$, and define the equivalence relation by the rule: $(m_1,n_1)\sim (m_2,n_2)$ if and only if $m_1n_2=m_2n_1$.  The ring operations on pairs $(m,n)$, $n\ne 0$, are defined by the standard formulas on fractions. We can also interpret $\Q$ in $\Z$ by irreducible fractions; it gives an absolute injective interpretability of $\Q$ in $\Z$.
\end{example}

\begin{example}\label{ex:RC}
The field of complex numbers $\MC$ (in the ring language) is absolutely injectively interpretable in the field of reals $\MR$, where a complex number $a+ib$ is represented by a pair $(a,b)$. At the same time, $\MR$ is not interpretable in $\MC$~\cite[Corollary~1.3.6]{Marker}. 
\end{example}

\begin{example}\label{ex:O1}
The ring $\Z$ is injectively $0$\=/interpretable in any ring of algebraic integers $\Aint$; the ring of algebraic integers $\Aint$ is injectively $0$\=/interpretable in its field of fractions; the ring $\Z$ is injectively $0$\=/interpretable in any algebraic number field $K$~\cite{J-Robinson}; the ring $\Z$ is interpretable in any finitely generated infinite field $F$ if $\mathrm{char} \, F\ne 2$~\cite{DP, Scanlon}; the ring $\Z$ is interpretable in any infinite finitely generated integral domain~\cite{AKNS}.
\end{example}

\begin{example}\label{ex:O2}
Any finitely generated ring of algebraic integers $\Aint$ is injectively absolutely interpretable in the ring $\Z$.
\end{example}

\begin{example}\label{ex:Noskov}
G.\,A.\,Noskov in~\cite{Noskov} proved that the ring $\Z$ is interpretable with parameters in any finitely generated non-virtually abelian soluble group $G$.
\end{example}

\begin{example}
Let $G$ be a countable group with an arithmetic multiplication table. Then $G$ is absolutely interpretable in $\N$. Indeed, by our assumptions, the group $G$ can be represented on the set $\N$ with operations of multiplication and inversion such that their graphs are definable in $\N$ by formulas of arithmetic; hence the result. 
\end{example}

\begin{example}\label{ex:BS1}
The metabelian non-abelian Baumslag\,--\,Solitar group $\BS(1,k)$, $k>1$, is absolutely interpretable in $\Z$~\cite{BS, Khelif}. Indeed, the group $\BS(1,k)$ is isomorphic to the semi-direct product $\Z[1/k]\rtimes \Z$ of the abelian groups $\Z[1/k] =\{zk^i\mid  z,i\in\Z\}$ and an infinite cyclic $\Z$; where the action of $\Z$ on $\Z[1/k]$ is defined by $(m,x)\to xk^{-m}$, $m\in \Z$, $x\in \Z[1/k]$. Every element from $\Z[1/k]\rtimes \Z$ has a form $(zk^i,m)$, $z,i,k\in \Z$;  and for any $z,i,m,z_j,i_j,m_j$, $j=1,2$, one has
\begin{enumerate}[label=\alph*)]
\item $(z_1k^{i_1},m_1)\cdot(z_2k^{i_2},m_2)=(z_1k^{i_1}+z_2k^{i_2-m_1},m_1+m_2)$,
\item $(zk^{i},m)^{-1} = (-zk^{i+m}, -m)$,
\item $e=(0,0)$.
\end{enumerate}
So an element $(zk^i,m)$ is interpretable by a triple $(z,i,m)\in \Z$. Further, $(z_1,i_1,m_1)\sim(z_2,i_2,m_2)$ if and only if $z_1k^{i_1}=z_2k^{i_2}$ and $m_1=m_2$. Note that the equivalence relation $\sim$ and the group operations in the interpretation are given by computable predicates, hence this interpretation can be given by Diophantine formulas~\cite{Matijasevich}. It is easy to ``refine'' this interpretation into an injective one.
\end{example}

\begin{example}\label{ex:BS2}
The ring $\Z$ is injectively interpretable in the group $\BS(1,k)$, $k>1$, with parameter $b=(0,1)$ on the set $\{b^n, n\in \Z\}=\CG_{\BS(1,k)}(b)$. Here an integer $m$ is interpretable by an element $b^m$ and for any $m_1,m_2\in \Z$ one has $b^{m_1+m_2}=b^{m_1}\cdot b^{m_2}$. Besides, there exists a formula $\otimes(x_1,x_2,x_3)$ in $L_{\group}\cup\{b\}$, such that $\BS(1,k)\models \otimes(x_1,x_2,x_3)$ if and only if $x_1=b^{m_1}$, $x_2=b^{m_2}$ and $x_3=b^{m_1m_2}$ for some $m_1,m_2\in \Z$~\cite[Lemma~5]{BS}.
\end{example}

\begin{example}\label{ex:GL}
For any associative unitary ring $R$ the classical matrix groups $\GL_n(R)$, $\SL_n(R)$ (in this case $R$ is commutative), $\mathrm{T}_n(R)$, and $\UT_n(R)$ are absolutely injectively interpretable in $R$ (a matrix is interpretable by a tuple of its elements listed in a fixed order); while $R$ is interpretable in these groups, for $n \geqslant 3$, with parameters (see~\cite[Lemma~5.1]{MS}). 
\end{example}

\begin{example}\label{ex:UT1}
Let $R$ be an arbitrary unitary ring (not necessarily associative). Consider the group of upper unitriangular $3\times3$\=/matrices $\UT_3(R)$ over $R$ (also called the Heisenberg group over $R$). As was mentioned in Example~\ref{ex:GL}, there exists the standard absolute injective interpretation $\Gamma\colon \UT_3(R)\rightsquigarrow R$, where a matrix $A =(a_{ij}) \in \UT_3(R)$ is represented by a triple  $(a_{12},a_{23},a_{13})$ from $R^3$, so that
\begin{enumerate}[label=\roman*)]
\item $(a_1,b_1,c_1)\cdot (a_2,b_2,c_3)=(a_1+a_2,\: b_1+b_2,\: c_1+c_2+a_1b_2)$, \; $a_i,b_i,c_i\in R$, \; $i=1,2$,
\item $(a,b,c)^{-1}=(-a,\: -b,\: -c+ab)$, \; $a,b,c\in R$,
\item $e=(0,0,0)$.
\end{enumerate}
Now we describe another absolute injective interpretation $\Delta\colon\UT_3(R)\rightsquigarrow R$, which comes from nilpotent groups or groups with bounded generation. Recall that every element $g\in\UT_3(R)$ can be decomposed uniquely as a product $g=t_{12}(\alpha)\cdot t_{23}(\beta)\cdot t_{13}(\gamma)$ for some $\alpha,\beta,\gamma\in R$, where $t_{ij}(\alpha)=e+\alpha e_{ij}$ are transvections, $\alpha\in R$. Furthermore, for every $\alpha, \beta, \gamma, \alpha_i, \beta_i, \gamma_i\in R$, $i=1,2$, one has the following equalities that define group operations on triples $(\alpha,\beta,\gamma)$:
\begin{enumerate}[label=\arabic*)]
    \item $(t_{12}(\alpha_1)\cdot t_{23}(\beta_1)\cdot t_{13}(\gamma_1))\cdot( t_{12} (\alpha_2) \cdot t_{23}(\beta_2)\cdot t_{13}(\gamma_2))=t_{12}(\alpha_1+\alpha_2)\cdot t_{23}(\beta_1+\beta_2)\cdot t_{13}(\gamma_1+\gamma_2-\alpha_2\beta_1)$,
    \item $(t_{12}(\alpha)\cdot t_{23}(\beta)\cdot t_{13}(\gamma))^{-1}=t_{12}(-\alpha)\cdot t_{23}(-\beta)\cdot t_{13}(-\gamma-\alpha\beta)$,
    \item $e=t_{12}(0)\cdot t_{23}(0)\cdot t_{13}(0)$,
    \item\label{item4} $[t_{12}(\alpha_1)\cdot t_{23}(\beta_1)\cdot t_{13}(\gamma_1), t_{12}(\alpha_2)\cdot t_{23}(\beta_2)\cdot t_{13}(\gamma_2)]=t_{13}(\alpha_1\beta_2-\alpha_2\beta_1)$.
\end{enumerate} 
This gives the interpretation $\Delta\colon\UT_3(R)\rightsquigarrow R$. Note that the last equality~\ref{item4} is not needed for this interpretation, but we will need it in the sequel.  
\end{example}

The interpretations $\Gamma$ and $\Delta$ in Example~\ref{ex:UT1} differ in their codes. Thus, we proceed to discuss the concept of code.

\subsection{Interpretation codes}\label{subsec:code}

Here we discuss the main technical tool, the interpretation code, which formalizes the notion of interpretation.  

The structure $\MA^\ast$ from Definition~\ref{de:interpretation} is completely described by the tuple of parameters $\bar p$ and the following set of formulas in the language $L(\MB)$:
\begin{equation*} 
\Gamma =  \{U_\Gamma(\bar x,\bar y), E_\Gamma(\bar x, \bar x^\prime,\bar y), Q_\Gamma(\bar x_1, \ldots,\bar x_{t_Q},\bar y) \mid Q \in L(\MA)\},
\end{equation*}
where $\bar x$, $\bar x^\prime$ and $\bar x_i$ are $n$\=/tuples of variables and $\bar y = (y_1, \ldots,y_k)$ is a tuple of extra variables for parameters $\bar p$. Namely,  $U_\Gamma$  defines in $\MB$ a set $A_\Gamma  = U_\Gamma(B^n,\bar p)  \subseteq B^n$ (the set $A^\ast$ in Definition~\ref{de:interpretation}), $E_\Gamma$ defines an equivalence relation $\sim_\Gamma$ on $A_\Gamma$ (the equivalence relation $\sim$ in Definition~\ref{de:interpretation}), sometimes we denote it by $\sim_{\Gamma,\bar p}$, and the formulas $Q_\Gamma$ define preimages in $\MB$ of graphs of constants, functions and predicates $Q\in L(\MA)$ on the quotient set $A_\Gamma/{\sim_\Gamma}$ in such a way that the structure $\Gamma(\MB,\bar p) = \langle A_\Gamma/{\sim_\Gamma}; L(\MA) \rangle $ is isomorphic to $\MA$. Here $t_c=1$ for a constant $c\in L(\MA)$, $t_f=n_f+1$ for a function $f\in L(\MA)$ and $t_R=n_R$ for a predicate $R\in L(\MA)$. Note that we interpret a constant $c \in L(\MA)$ in the structure $\Gamma(\MB,\bar p)$ by the $\sim_\Gamma$\=/equivalence class of some tuple $\bar b_c \in A_\Gamma$ defined in $\MB$ by the formula $c_\Gamma(\bar x,\bar p)$. Sometimes we write formulas from $\Gamma$ in a generalized form $\psi(\bar{\bar x},\bar y)$, where $\bar{\bar x}$ either equals $\bar x$ or $(\bar x,\bar x^\prime)$, or $(\bar x_1,\ldots, \bar x_{t_Q})$, which is always clear from the context. 

We term  $\Gamma$ the {\em interpretation code from $L(\MA)$ to $L(\MB)$} (and sometimes write as $\Gamma\colon L(\MA) \to L(\MB)$). The notation $U_\Gamma$, $E_\Gamma$, and $Q_\Gamma$ is consistently used throughout the entire paper. The number $n$ is the {\em dimension} of $\Gamma$, while $k$ is the {\em parameter dimension} of $\Gamma$, we denote them by $n = \dim \Gamma$ and $k=\dim_{\param}\Gamma$, respectively. 

If $\Gamma$ interprets $\MA$ in $\MB$ then we refer to $\Gamma$  as the {\em code} of this interpretation of $\MA$ in $\MB$, and the interpretation itself is completely defined by a pair $(\Gamma, \bar p)$ (assuming the structure $\MB$ is clear from context).  Note that there may be several different interpretations of  $\MA$ in $\MB$ defined by different codes and different parameters. As we explained above, the writing $\MA\rightsquigarrow\MB$ means that $\MA$ is interpretable in $\MB$ by some code $\Gamma$ and some tuple of parameters $\bar p$. If we want to specify the code $\Gamma$, we may write 
$\MA\stackrel{\Gamma}{\rightsquigarrow}\MB$ or $\Gamma\colon\MA\rightsquigarrow\MB$, or $\MA\simeq\Gamma(\MB)$; in the case when we want to specify the code $\Gamma$ and the parameters $\bar p$, we write $\MA \stackrel{\Gamma,\bar p}{\rightsquigarrow}\MB$ or $(\Gamma,\bar p)\colon \MA\rightsquigarrow\MB$, or $\MA \simeq \Gamma(\MB,\bar p)$. The code $\Gamma$ (and an interpretation $\MA\stackrel{\Gamma}{\rightsquigarrow}\MB$) is {\em absolute} if $\bar y=\emptyset$; in this case, we may write $\MA \stackrel{\Gamma,\emptyset}{\rightsquigarrow}\MB$, $(\Gamma,\emptyset)\colon \MA\rightsquigarrow\MB$ and $\MA\simeq\Gamma(\MB,\emptyset)$. 

A surjective map $\mu_\Gamma=\mu_{\Gamma,\bar p}\colon A_\Gamma \to A$ that induces an isomorphism $\bar \mu_\Gamma=\bar\mu_{\Gamma,\bar p}\colon \Gamma(\MB,\bar p) \to \MA$ is termed the {\em coordinate map} of the interpretation $(\Gamma,\bar p)\colon \MA\rightsquigarrow\MB$. Note that such $\mu_\Gamma$ may not be unique; it is uniquely defined up to an automorphism of $\MA$; So, if $\MA$ is rigid, then $\mu_\Gamma$ is uniquely defined. If we want to specify a particular coordinate map $\mu_\Gamma$ for an interpretation $(\Gamma, \bar p)$, we write $(\Gamma,\bar p, \mu_\Gamma)$. Sometimes, we refer to such a  triple $(\Gamma,\bar p, \mu_\Gamma)$ as {\em a coordinatization of the interpretation $\MA\stackrel{\Gamma, \bar p}{\rightsquigarrow}\MB$}. Some authors define interpretations exclusively as triples $(\Gamma,\bar p, \mu_\Gamma)$ (see, for instance,~\cite{AKNS}), so in this case each interpretation $\MA\stackrel{\Gamma,\bar p}{\rightsquigarrow}\MB$ inherently comes equipped with a fixed coordinate map. We want to point out that in~\cite{AKNS}, the direction of the arrow $\rightsquigarrow$ is opposite to the one used in our notation. Note that if $\MA\stackrel{\Gamma,\bar p}{\rightsquigarrow}\MB$ is an injective interpretation, then the surjection $\mu_\Gamma$ is also injective. The map $\mu_\Gamma\colon A_\Gamma \to A$ gives rise to a map $\mu_\Gamma^m\colon A^m_\Gamma\to A^m$ on the corresponding Cartesian powers, which we often denote again by  $\mu_\Gamma$. And we write $U^{\wedge}_\Gamma(\bar{\bar x},\bar y)$ instead of $\bigwedge_{i=1}^m U_\Gamma(\bar x_i,\bar y)$, where $\bar{\bar x}=(\bar x_1,\ldots,\bar x_m)$, $|\bar x_i|=\dim\Gamma$. In these notations, for example, the entry $U^\wedge_\Gamma(\bar x,\bar u,\bar v,\bar y)$ denotes $U_\Gamma(\bar x,\bar y)\wedge U_\Gamma(\bar u,\bar y)\wedge U_\Gamma(\bar v,\bar y)$.  

It is important to note that there can be really different interpretations of $\MA$ in $\MB$, as illustrated in Example~\ref{ex:UT1}, while others are essentially the same, differing mainly in the way we write formulas. We will define the latter formally below. Before that, note that for an $L(\MB)$\=/structure $\nsB$ and a tuple $\bar q \in \nsB$, the algebraic structure $\Gamma(\nsB, \bar q)$ described above might not exist. The criteria for its existence are discussed in Subsection~\ref{subsec:AC}.

\begin{definition}\label{def:equiv_codes}
    Two codes $\Gamma\colon L(\MA)\to L(\MB)$ and $\Gamma^\prime\colon L(\MA)\to L(\MB)$ are called {\em equivalent} (correspondingly, {\em $\phi$\=/equivalent}, for a formula $\phi(y_1,\ldots,y_k)$ in $L(\MB)$) if
    \begin{enumerate}[label=(\arabic*)]
        \item $\dim \Gamma=\dim \Gamma^\prime$, $\dim_\param \Gamma=\dim_\param \Gamma^\prime=k$,
        \item for any $L(\MB)$\=/structure $\nsB$ and any tuple $\bar q\in \nsB$ (correspondingly, any tuple $\bar q\in \phi(\nsB)$) the pair $(\Gamma,\bar q)$ defines an $L(\MA)$\=/structure $\Gamma(\nsB,\bar q)$ in $\nsB$ if and only if $(\Gamma^\prime,\bar q)$ defines an $L(\MA)$\=/structure $\Gamma^\prime(\nsB,\bar q)$ in $\nsB$, moreover, in this case $\Gamma(\nsB,\bar q)=\Gamma^\prime(\nsB,\bar q)$.
    \end{enumerate}
We write $\Gamma\approx\Gamma^\prime$ ($\Gamma\approx_\phi\Gamma^\prime$) if $\Gamma$ and $\Gamma^\prime$ are equivalent ($\phi$\=/equivalent).
\end{definition}

\begin{remark}\label{rem:Id1}
We can always assume that for any tuple $(\bar b,\bar b\,^\prime,\bar q)\in E_\Gamma(\MB)$ (similarly, $(\bar b_1,\ldots,\bar b_{t_Q},\bar q)\in Q_\Gamma(\MB)$) one has $\bar b,\bar b\,^\prime \in U_\Gamma(\MB,\bar q)$ (correspondingly, $\bar b_1,\ldots,\bar b_{t_Q} \in U_\Gamma(\MB,\bar q)$). Otherwise, one can replace the formula $E_\Gamma$ with a new formula $E_\Gamma^\prime(\bar x, \bar x^\prime, \bar y) = E_\Gamma(\bar x, \bar x^\prime, \bar y) \wedge U^\wedge_\Gamma(\bar x, \bar x^\prime,\bar y)$. And the same we can do with all the formulas $Q_\Gamma$, $Q\in L(\MA)$. In particular, it follows that we can also assume $U_\Gamma(\bar x, \bar y) = E_\Gamma(\bar x, \bar x, \bar y)$. The new code $\Gamma^\prime$ we get this way is equivalent to the old one $\Gamma$.
\end{remark}

\begin{remark}\label{rem:Id3}
Sometimes it is necessary to adjust formulas from a code $\Gamma$, such that  $\sim_\Gamma$ become an equivalence relation on the whole set $B^n$ and $U_\Gamma(\MB,\bar p)$ is closed under $\sim_\Gamma$, i.\,e., if $\bar b_1,\bar b_2\in B^n$ such that $\bar b_1\sim_\Gamma \bar b_2$ and $\bar b_1\in U_\Gamma(\MB,\bar p)$, then $\bar b_2\in U_\Gamma(\MB,\bar p)$. Indeed, it is enough to  replace the formula $E_\Gamma$ by
$$
E^\prime_\Gamma(\bar x_1, \bar x_2,\bar y)= (U^\wedge_\Gamma(\bar x_1,\bar x_2,\bar y) \longrightarrow E_\Gamma(\bar x_1,\bar x_2, \bar y)) \,  \wedge\,(\neg (\bar x_1=\bar x_2) \longrightarrow U^\wedge_\Gamma(\bar x_1,\bar x_2,\bar y)). 
$$
In this way, we will again obtain a code equivalent to the original one.
\end{remark}

\begin{definition}\label{E_inj}
A code $\Gamma$ is {\em injective} if the formula $E_\Gamma(\bar x,\bar x',\bar y)$ has the form $E_{\inj}(\bar x,\bar x^\prime,\bar y) = \bigwedge_{i = 1}^n (x_i = x_i^\prime)$, where $\bar x=(x_1,\ldots,x_n)$, $\bar x^\prime=(x_1^\prime,\ldots,x_n^\prime)$.
\end{definition}

The following remark shows that every injective interpretation can be replaced with an interpretation given by an injective code.

\begin{remark}\label{rem:Id4}
    For a code $\Gamma$ denote by $\Gamma_{\inj}$ the code obtained from $\Gamma$ by replacing the formula $E_\Gamma(\bar x,\bar x^\prime,\bar y)$ with the formula $E_{\inj}(\bar x,\bar x^\prime,\bar y)$.  Then if $\MA \simeq \Gamma(\MB,\bar p)$ is an injective interpretation of $\MA$ in $\MB$ then $\MA \simeq \Gamma_{\inj}(\MB,\bar p)$.  
\end{remark}

\begin{definition}
Let $L(\MA)$ and $L(\MB)$ be computable languages. We term a code $\Gamma\colon L(\MA)\to L(\MB)$, $\Gamma=\{U_\Gamma,E_\Gamma,Q_\Gamma\mid Q\in L(\MA)\}$, {\em computable}, if for every symbol $Q$ in the signature of $L(\MA)$ one can effectively construct the formula $Q_\Gamma$ from the code $\Gamma$. To be more precise, one can use G\"odel numbering of formulas in the languages $L(\MA)$ and $L(\MB)$, but we leave it to the reader.  Furthermore, an interpretation $\MA\rightsquigarrow\MB$ is {\em computable}, if it is defined by a computable code.
\end{definition}

In particular, if languages $L(\MA)$ and $L(\MB)$ are finite, then every code $\Gamma\colon L(\MA)\to L(\MB)$ is computable. All examples of interpretations in Subsection~\ref{subsec:code} are computable.

\begin{definition}
A code $\Gamma$ is {\em Diophantine} (correspondingly, $\Sigma_n$\=/{\em code}, $\Pi_n$\=/{\em code}) if all its formulas are Diophantine (correspondingly, have $\Sigma_n$\=/form, $\Pi_n$\=/form). Furthermore, an interpretation $\MA\rightsquigarrow\MB$ is {\em Diophantine} (correspondingly, $\Sigma_n$\=/{\em interpretation}, $\Pi_n$\=/{\em interpretation}), if it is defined by a Diophantine code (correspondingly, by a $\Sigma_n$\=/{\em code}, $\Pi_n$\=/{\em code}) $\Gamma$. 
\end{definition}

\begin{remark}\label{rem:cor_code}
In practical scenarios, when we design a code $\Gamma$, we often neglect to ensure that the formulas $Q_\Gamma$ represent full preimages of the graphs of $Q \in L(\MA)$. For example, while creating an interpretation $\Q \rightsquigarrow \Z$ on pairs $(x_1,x_2) \in \Z^2$ with $x_2 \neq 0$, we may define addition  $(x_1,x_2) + (x_3,x_4) = (x_5,x_6)$ using the formula:
$$
+(x_1,x_2,x_3,x_4,x_5,x_6)\:=\:(x_5=x_1\cdot x_4 + x_3\cdot x_2)\,\wedge\,(x_6=x_2\cdot x_4).
$$
Thus, addition is defined only for certain representatives of equivalence classes. However, the definition can be refined by replacing the formulas $Q_\Gamma$ with new ones: 
$$
Q^\prime_\Gamma(\bar x_1, \ldots, \bar x_{t_Q},\bar y)=\exists\bar x_1^\prime\:\ldots \exists\bar x_{t_Q}^\prime \; (Q_\Gamma(\bar x_1^\prime,\ldots, \bar x_{t_Q}^\prime,\bar y) \wedge \bigwedge\limits_{i=1}^{t_Q}E_\Gamma(\bar x_i,\bar x_i^\prime,\bar y)\,)
$$
and obtaining the full preimages of the graphs. Besides, if here $Q_\Gamma$ and $E_\Gamma$ are Diophantine, then $Q^\prime_\Gamma$ is Diophantine as well.
\end{remark}

The classical interpretations $\N\rightsquigarrow\Z$, $\MC\rightsquigarrow\MR$, $\Aint\rightsquigarrow\Z$, $\BS(1,k)\rightsquigarrow\Z$, $\GL_n(R)\rightsquigarrow R$, $\SL_n(R)\rightsquigarrow R$, $\UT_3(R)\rightsquigarrow R$, mentioned in Subsection~\ref{subsec:code}, are Diophantine, while the interpretation $\Z\rightsquigarrow\Q$ is $\Pi_3$.

Diophantine interpretations were introduced in~\cite{GMO1, GMO2} to address decidability of the Diophantine problems in groups and rings. Sometimes, they are referred to as {\em e\=/interpretations}, or {\em interpretations by equations}.

For a language $L$, the following absolute injective Diophantine code of dimension $1$
\begin{multline*}
\Id_{L}=\{U_{\Id_{L}}(x)=(x=x),\: E_{\Id_{L}}(x,x^\prime)=(x=x^\prime),\\ c_{\Id_{L}}(x)=(c=x), \: f_{\Id_{L}}(x_1,\ldots,x_{n_f},x_0)=(f(x_1,\ldots,x_{n_f})=x_0),\\ R_{\Id_{L}}(x_1,\ldots,x_{n_R})=R(x_1,\ldots,x_{n_R})\mid c,f,R\in L\}
\end{multline*}
gives an interpretation $\MA\stackrel{\Id_{L}}{\rightsquigarrow}\MA$ for any algebraic structure $\MA$ with $L(\MA) = L$, termed the {\em identical interpretation}.

For an interpretation $(\Gamma,\bar p, \mu_\Gamma)$ of $\MA$ in $\MB$ and an element $a\in A$ we denote by $\bar\mu^{-1}_\Gamma(a)$ some fixed tuple $\bar b$ from $\mu_\Gamma^{-1}(a)$. Further, for a subset $X\subseteq A$ we denote by $\Gamma_X\colon L(\MA)\cup X \to L(\MB)\cup B$ the code with additional constants from $X$,
$$
\Gamma_X=\Gamma\cup\{c_{\Gamma_X}(\bar x) \mid c\in X\}, \ \mbox{ where } c_{\Gamma_X}(\bar x)=E_\Gamma(\bar x,\bar\mu_\Gamma^{-1}(c),\bar p). 
$$

\begin{remark}\label{lem:AB1}
If $\MA=\langle A; L(\MA)\rangle$ is (injectively) interpretable in $\MB=\langle B; L(\MB)\rangle$ by a code $\Gamma$, then $\MA_A=\langle A; L(\MA)\cup A\rangle$ is absolutely (injectively) interpretable in $\MB_B=\langle B; L(\MB)\cup B\rangle$ by the code $\Gamma_A$. If the language $L(\MA)$ is finite, then the inverse statement also holds: if $\MA_A$ is (injectively) interpretable in $\MB_B$, then $\MA$ is (injectively) interpretable in $\MB$.
\end{remark}

We continue with Example~\ref{ex:UT1} and describe the interpretation $R\rightsquigarrow \UT_3(R)$ due to A.\,I.\,Mal'cev~\cite{Malcev}.

\begin{example}\label{ex:UT2}
Let $R$ be a unitary ring and $G=\UT_3(R)$, $\ZG(G)$ be the center of $G$, $[G,G]$ be its commutator subgroup, and $\CG_G(g)$ be the centralizer of an element $g\in G$. Recall that $\ZG(G)=\{t_{13}(\alpha), \alpha\in R\}=[G,G]=\{[t_{12}(\alpha), t_{23}(1)] \mid \alpha\in R\}$,  $\CG_G(t_{12}(1))=\{t_{12}(\alpha)\cdot t_{13}(\gamma) \mid \alpha,\gamma\in R\}$, and $\CG_G(t_{23}(1))=\{t_{23}(\beta)\cdot t_{13}(\gamma) \mid \beta,\gamma\in R\}$. The following code $\Gamma\colon L_\ring\to L_\group$ and parameters $(y_1,y_2)=(t_{12}(1),t_{23}(1))$ give an interpretation $R\rightsquigarrow \UT_3(R)$:
\begin{gather*}
U_\Gamma(x,y_1,y_2)=\exists\,x_1\:\exists\,x_2\:(x=[x_1,x_2]),  \quad
E_\Gamma(x_1,x_2,y_1,y_2)=(x_1=x_2),\\
\oplus(x_1,x_2,x_3,y_1,y_2)=(x_1\cdot x_2=x_3),\\   \otimes(x_1,x_2,x_3,y_1,y_2)=\exists\,z_1\,\exists\,z_2\,([y_1,z_1]=[y_2,z_2]=e\wedge x_1=[z_1,y_2]\wedge x_2=[y_1,z_2]\wedge x_3=[z_1,z_2]),\\ \mathbf{0}(x,y_1,y_2)=(x=e), \quad
   \mathbf{1}(x,y_1,y_2)=(x=[y_1,y_2]).
\end{gather*}
Note that the code $\Gamma$ is injective and Diophantine. Here $U_\Gamma(G)=\ZG(G)$. In the sequel  we will write $x_1\oplus x_2=x_3$, $x_1\otimes x_2=x_3$, $\mathbf{0}$, $\mathbf{1}$, instead of exact formulas from the code $\Gamma$, when considering this example. So, $t_{13}(\alpha)\oplus t_{13}(\beta)=t_{13}(\alpha+\beta)$, $t_{13}(\alpha)\otimes t_{13}(\beta)=t_{13}(\alpha\beta)$ for all $\alpha,\beta\in R$.
\end{example}

\subsection{Regular interpretability}\label{subsec:reg_int}

An absolute interpretation is a specific type of interpretation with parameters. While it is more convenient, it appears rarely. Conversely, parameters can limit the interpretation's applicability in certain cases. This issue can be resolved by aligning the parameter choice with a specific formula-descriptor.

\begin{definition}\label{def:regular_int}
We say that $\MA$ is {\em regularly interpretable} in $\MB$ if there exists a code $\Gamma\colon L(\MA)\to L(\MB)$ and an $L(\MB)$\=/formula $\phi(y_1,\ldots,y_k)$, $k=\dim_\param\Gamma$ (called a {\em parameter descriptor}, or just a {\em descriptor} of $\Gamma$), such that $\phi(\MB)\ne\emptyset$ and for each $\bar p=(p_1,\ldots,p_k)\in \phi(\MB)$ the pair $(\Gamma,\bar p)$ gives an interpretation $\MA\simeq\Gamma(\MB,\bar p)$ of $\MA$ in $\MB$.
\end{definition}

We denote by $(\Gamma,\phi)$  the regular interpretation with the code $\Gamma$ and the descriptor $\phi$. We write $\MA\simeq \Gamma(\MB,\phi)$ or $\MA\stackrel{\Gamma, \phi}{\rightsquigarrow}\MB$, or $(\Gamma,\phi)\colon\MA\rightsquigarrow \MB$,  if $(\Gamma,\phi)$ regularly interprets $\MA$ in $\MB$. 

\begin{remark}\label{rem:Id5}
When dealing with a regular interpretation $\MA\simeq\Gamma(\MB,\phi)$ it is sometimes convenient to redefine each formula $\psi(\bar{\bar x},\bar y)$ from the code $\Gamma$ by $\psi^\prime(\bar{\bar x},\bar y)=\phi(\bar y)\wedge \psi(\bar {\bar x},\bar y)$. The new code will be $\phi$\=/equivalent to $\Gamma$.
\end{remark}

If for all $\bar p\in \phi(\MB)$ the interpretation $(\Gamma,\bar p)$ is injective, then the regular interpretation $(\Gamma,\phi)$ is called {\em injective}. Note that if $\MA\simeq \Gamma(\MB,\phi)$ is a regular injective interpretation, then $\MA\simeq \Gamma_{\inj}(\MB,\phi)$ is a regular (injective) interpretation  of $\MA$ in $\MB$.

We say that descriptor $\phi$ of a regular interpretation $(\Gamma,\phi)\colon\MA\rightsquigarrow\MB$ is {\em removable}, if $\Gamma(\MB,\bar p_1)=\Gamma(\MB,\bar p_2)$ for any tuples $\bar p_1,\bar p_2\in \phi(\MB)$. If $(\Gamma,\phi)\colon\MA\rightsquigarrow\MB$ is a regular interpretation with a removable descriptor, then $(\Gamma^\abs,\emptyset)\colon\MA\rightsquigarrow\MB$ is an absolute interpretation, where the code $\Gamma^\abs$ is obtained from the code $\Gamma$ by replacing every formula $\psi(\bar{\bar x},\bar y)$ from $\Gamma$   by the formula $\psi^\abs(\bar{\bar x})=\exists\,\bar y\:(\phi(\bar y)\wedge\psi(\bar{\bar x},\bar y))$ (here $|\bar y|=\dim_\param\Gamma$).

Thus, we intend to view absolute interpretations as a particular case of regular ones. Informally,  whatever is true for regular interpretations also holds for the absolute ones, without changing much of the arguments. More formally, any absolute code $\Gamma$ can be converted into a pair  $(\Gamma^\reg,\reg(y))$, termed as the {\em simple parameter extension} of $\Gamma$, by introducing a fictitious variable $y$, for parameter, into all formulas in $\Gamma$, namely, formally denoting every formula $\psi(\bar{\bar x})$ from $\Gamma$ by $\psi^\reg(\bar{\bar x},y)=\psi(\bar{\bar x})$; and defining $\reg(y)=(y=y)$. If $\MA \simeq \Gamma(\MB,\emptyset)$, then $\MA \simeq \Gamma^\reg(\MB,\reg)$ is a regular interpretation with removable descriptor. Similarly, $(\Gamma,\emptyset)$ is the {\em simple absolutization} of $(\Gamma^\reg,\reg)$. 
The simple parameter extension $(\Gamma,\emptyset)\to(\Gamma^\reg,\reg)$ and simple absolutization $(\Gamma^\reg,\reg)\to(\Gamma,\emptyset)$ are particular cases of notions of parameter extension and absolutization; the former is introduced in this subsection, and the latter will be addressed in one of the upcoming papers in this series.

Note that if there exists an absolute interpretation $\MA\rightsquigarrow\MB$, then there exists a regular interpretation $\MA\rightsquigarrow\MB$, and if there exists a regular interpretation $\MA\rightsquigarrow\MB$, then there exists an interpretation $\MA\rightsquigarrow\MB$ with parameters. Example~\ref{ex:G1} below shows two structures $\MA$ and $\MB$ such that $\MA$ is regularly interpretable in $\MB$, but not absolutely; while Example~\ref{ex1} presents two structures $\MA$ and $\MB$ such that $\MA$ is interpretable in $\MB$ with parameters, but not regularly. The general intuition is that absolute interpretations are rarer or harder to establish than those with parameters, and regular interpretations offer a compromise between them. Regular interpretations allow for obtaining a number of consequences that cannot be derived from interpretations with parameters.

By the way, in Hodges’ book~\cite{Hodges} the settings are slightly different: there, interpretations by default are absolute, and the regular ones are called {\em interpretations with definable parameters}. Here, the basic interpretations are with parameters. Of course, one can always put parameters in the language, as in Remark~\ref{lem:AB1}, and obtain an absolute interpretation. Sometimes this simplifies things, but sometimes not. This makes it especially difficult to deal with regular interpretations $(\Gamma,\phi)$ of $\MA$ in $\MB$, for example, if the set $\phi(\MB)$ is infinite the language $L(\MB) \cup \phi(\MB)$ becomes infinite, even though the initial language $L(\MB)$ is finite. Also, adding constants to the language makes it difficult to study elementary equivalence of structures, in particular, the first-order classification problems, or to study decidability problems: a decidable elementary theory $\Th(\MB)$  might become undecidable after adding some constants, etc. 

In an upcoming paper in this series, we show that, under certain favorable conditions, an interpretation with parameters can be enhanced to a regular interpretation, and occasionally to an absolute interpretation. As an illustration, in Lemma~\ref{ex:UT3} below we construct a regular interpretation $R\rightsquigarrow \UT_3(R)$ from an interpretation with parameters. Furthermore, from this regular interpretation, one can build up an absolute one. However, this is not always the case, as the following example shows two structures $\MA$ and $\MB$ such that $\MA$ is regularly interpretable in $\MB$, but it can't be absolutely interpretable in $\MB$. 

\begin{example}[P.\,Gvozdevsky]\label{ex:G1}
Let $L(\MA)$ be the empty language and $A=\Z$, while $L(\MB)$ be the language with a unique predicate $R$ of arity $2$ and $B=\MR$, where $\MB\models R(x,y)$ if and only if $x-y\in \Z$. Then $\MA$ is regularly interpretable in $\MB$ by means of $(\Gamma, \phi)$, where $\Gamma=\{U_\Gamma(x,y)=R(x,y), E_\Gamma(x_1,x_2,y)=(x_1=x_2)\}$ and $\phi(y)=(y=y)$. At the same time, $\MA$ is not absolutely interpretable in $\MB$, since for any $0$\=/definable non-empty set $X$ in $\MB$ and any $0$\=/definable equivalence relation $\sim$ on $X$ the quotient $X/{\sim}$ is finite or uncountable. Let us show that this fact is true. Suppose that a set $X\subseteq B^n$ and a relation $\sim$ are fixed. Firstly, note that for any infinite set $W$, a finite subset $U\subset W$ and an injective map $f\colon U\to W$, there exists a bijection $\bar f\colon W\to W$, such that $f(u)=\bar f(u)$ for all $u\in U$. Secondly, for a real number $r\in \MR$ let us denote by $[r]$ the set $\{r+z\in \MR \mid z\in \Z\}$, i.\,e., the equivalence class of $r$ with respect to relation $R$. For any tuple $\bar b=(b_1,\ldots,b_n)\in \MR^n$ denote by $R_{\bar b}$ the set $\{[b_{i_1}],\ldots,[b_{i_s}]\}$ of all pairwise distinct elements from  the set $\{[b_1],\ldots,[b_n]\}$. Then for any pairwise distinct sets $[r_{i_1}],\ldots, [r_{i_s}]$, $r_{i_j}\in \MR$, there exists a tuple $\bar r=(r_1,\ldots,r_n)\in \MR^n$, such that $\bar r\in \Orb(\bar b)$. Thirdly, the automorphic orbit $\Orb(\bar b)$ is an absolutely definable set; it is defined by the following formula
$$
\varphi_{\bar b}(x_1,\ldots,x_n)=\bigwedge_{b_i=b_j}(x_i=x_j)\wedge \bigwedge_{b_i\ne b_j}\neg(x_i=x_j)\wedge\bigwedge_{b_i-b_j\in\Z} R(x_i,x_j)\wedge \bigwedge_{b_i-b_j\not\in\Z} \neg R(x_i,x_j).
$$
Thus, one has $\Orb(\bar b)=\varphi_{\bar b}(\MB)$ and for any tuple $\bar c\in B^n$ one has $\Orb(\bar b)=\Orb(\bar c)$ if and only if $\bar c\in \varphi_{\bar b}(\MB)$. As we see, the automorphic orbit $\Orb(\bar b)$ is an uncountable set. At the same time, for the fixed integer $n\geqslant 1$, the set of all automorphic orbits $\Orb(\bar b)$ of tuples $\bar b\in B^n$ is finite, as soon as there exists just a finite set of formulas $\varphi_{\bar b}$. Fourthly, note that $\Orb(\bar b)\subseteq X$ for any tuple $\bar b\in X$, since $X$ is $0$\=/definable, and suppose that there exist tuples $\bar b_0,\bar c_0\in X$, such that $\Orb(\bar b_0)=\Orb(\bar c_0)=Y$, but $\bar b_0\not\sim\bar c_0$ (otherwise, the set $X/{\sim}$ is finite, and there is nothing to prove). In this case one can find tuples $\bar b,\bar c\in Y$, such that $\bar b\not\sim\bar c$ and $R_{\bar b}\cap R_{\bar c}=\emptyset$, i.\,e., $\MB\models \psi (\bar b,\bar c)$, where $\psi( x_1,\ldots x_n, y_1,\ldots,y_n)=\bigwedge\limits_{i,j=1}^n\neg R(x_i,y_j)$. Indeed, if for all $\bar b_1,\bar c_1\in Y$ with $\MB\models \psi (\bar b_1,\bar c_1)$ one has $\bar b_1\sim \bar c_1$, then it is easy to find a tuple $\bar a\in Y$ such that $\MB\models \psi (\bar a,\bar b_0)\wedge \psi (\bar a,\bar c_0)$, so that $\bar b_0\sim\bar a\sim \bar c_0$, but it contradicts to $\bar b_0\not\sim\bar c_0$. Fifthly, there exists an uncountable set $Y_0\subset \Orb(\bar b)=Y$, such that $\bar c\in Y_0$ and for every non-equal tuples $\bar r_1,\bar r_2\in Y_0$ one has $R_{\bar r_1}\cap R_{\bar r_2}=\emptyset$, i.\,e., $\MB\models \psi(\bar r_1,\bar r_2)$. Therefore, for any $\bar r_1\ne\bar r_2\in Y_0$ one has $(\bar r_1,\bar r_2)\in \phi_{(\bar b,\bar c)}(\MB)$, i.\,e.,  
 $\Orb((\bar b,\bar c))=\Orb((\bar r_1,\bar r_2))$. It gives that $\bar r_1\not\sim \bar r_2$, so the set $X/{\sim}$ is uncountable. 
\end{example}

We don't know yet of any examples of this kind that are more natural, because it can be difficult to prove the absence of an absolute interpretation.

\begin{problem}\label{problem1}
Find an example of algebraic structures $\MA$ and $\MB$ in the class of groups or rings such that there exists a regular interpretation $\MA\rightsquigarrow\MB$, but $\MA$ is not absolutely interpretable in $\MB$. 
\end{problem}

In the next result, we use the ideas due to A.\,I.\,Mal'cev~\cite{Malcev} and Examples~\ref{ex:UT1} and~\ref{ex:UT2} to prove a regular interpretability $R\rightsquigarrow \UT_3(R)$. 

\begin{lemma}\label{ex:UT3}
Suppose that $R$ is a commutative associative unitary ring and $G=\UT_3(R)$. Then there exists a regularization of the interpretation $R\rightsquigarrow \UT_3(R)$ from Example~\ref{ex:UT2} with descriptor $\phi(y_1,y_2)$, which describes the following statements:
\begin{enumerate}[label=\arabic*)]
    \item $[y_1,y_2]\ne e$,
    \item $\CG_G(y_1)\cap \CG_G(y_2)=\ZG(G)$,
    \item for any element $z\in \ZG(G)$ there exist unique modulo $\ZG(G)$ elements $z_1\in \CG_G(y_1)$ and $z_2\in \CG_G(y_2)$, such that $z=[y_1,z_2]=[z_1,y_2]$.
\end{enumerate}
\end{lemma}

\begin{proof}
It is easy to see that $(t_{12}(1),t_{23}(1))\in \phi(G)$. Take an arbitrary pair of elements $\bar a=(a_1,a_2)\in \phi(G)$. Then the $L_\group$\=/structure $\Gamma(G,\bar a)$ is well-defined. We need to show that $R\simeq \Gamma(G,\bar a)$. 

Let $a_i=t_{12}(\alpha_i)\cdot t_{23}(\beta_i)\cdot t_{13}(\gamma_i)$, $\alpha_i,\beta_i,\gamma_i\in R$, $i=1,2$. Then $[a_1,a_2]=t_{13}(\alpha_1\beta_2-\alpha_2\beta_1)$; put  $\gamma=\gamma_{\bar a}=\alpha_1\beta_2-\alpha_2\beta_1$. For  $z=t_{13}(1)$ there exist  $z_1\in \CG_G(a_1)$, $z_2\in \CG_G(a_2)$, such that $z=[a_1,z_2]=[z_1,a_2]$, say, $z_i=t_{12}(\xi_i)\cdot t_{23}(\tau_i)\cdot t_{13}(\kappa_i)$, $\xi_i,\tau_i,\kappa_i\in R$, $i=1,2$. Then $[z_1,z_2]=t_{13}(\xi_1\tau_2-\xi_2\tau_1)$; denote $\delta=\delta_{\bar a}=\xi_1\tau_2-\xi_2\tau_1$. 

In this notation, $\gamma\delta=1$, so $\gamma,\delta$ are units of the ring $R$. Indeed, since $z_i\in \CG_G(a_i)$, one has $\alpha_i\tau_i=\xi_i\beta_i$, $i=1,2$, and since $z=[a_1,z_2]=[z_1,a_2]$, one has $1=\alpha_1\tau_2-\xi_2\beta_1=\xi_1\beta_2-\alpha_2\tau_1$. Therefore, $\gamma\delta=(\xi_1\tau_2-\xi_2\tau_1)(\alpha_1\beta_2-\alpha_2\beta_1)=\xi_1\beta_2\tau_2\alpha_1+\xi_2\beta_1\tau_1\alpha_2-\xi_1\beta_1\tau_2\alpha_2-\xi_2\beta_2\tau_1\alpha_1=\xi_1\beta_2(1+\xi_2\beta_1)+(\tau_2\alpha_1-1)\tau_1\alpha_2-\tau_1\alpha_1\tau_2\alpha_2-\xi_2\beta_2\xi_1\beta_1=\xi_1\beta_2-\tau_1\alpha_2=1$. Further, for any $\rho\in R$ and $z_i(\rho)=t_{12}(\rho\alpha_i\delta)\cdot t_{23}(\rho\beta_i\delta)$ one has $z_i(\rho)\in \CG_G(a_i)$, $i=1,2$, and since $\alpha_1\rho\beta_2\delta-\rho\alpha_2\delta\beta_1=\rho\gamma\delta=\rho$, then $t_{13}(\rho)=[z_1(\rho),a_2]=[a_1,z_2(\rho)]$. Recall, that the formula $\phi$ guarantees that $z_1(\rho),z_2(\rho)$ are unique modulo $\ZG(G)$ with such properties. Therefore, for any $\alpha,\beta\in R$ one has $t_{13}(\alpha)\otimes t_{13}(\beta)=[z_1(\alpha),z_2(\beta)]=[t_{12}(\alpha\alpha_1\delta)\cdot t_{23}(\alpha\beta_1\delta),t_{12}(\beta\alpha_2\delta)\cdot t_{23}(\beta\beta_2\delta)]=t_{13}(\alpha\beta(\alpha_1\beta_2-\alpha_2\beta_1)\delta^2)=t_{13}(\alpha\beta\gamma\delta^2)=
t_{13}(\alpha\beta\delta)$.

The isomorphism $\bar\mu_\Gamma=\bar\mu_{\Gamma,\bar a}\colon\Gamma(\UT_3(R),\bar a)\to R$ is set by the rule: 
$$
\bar\mu_\Gamma \colon \;t_{13}(\alpha)\in \ZG(G), \:\alpha\in R \ \longrightarrow \ \alpha\delta\in R.
$$
In particular, $\bar\mu_\Gamma(\mathbf{0})=\bar\mu_\Gamma(e)=\bar\mu_\Gamma(t_{13}(0))=0$, $\bar\mu_\Gamma(\mathbf{1})=\bar\mu_\Gamma([a_1,a_2])=\bar\mu_\Gamma(t_{13}(\gamma))=\gamma\delta=1$,
$\bar\mu_\Gamma(t_{13}(\alpha)\oplus t_{13}(\beta))=\bar\mu_\Gamma(t_{13}(\alpha)\cdot t_{13}(\beta))=\bar\mu_\Gamma(t_{13}(\alpha+\beta))=(\alpha+\beta)\delta=\bar\mu_\Gamma(t_{13}(\alpha))+ \bar\mu_\Gamma(t_{13}(\beta))$, and finally, $\bar\mu_\Gamma(t_{13}(\alpha)\otimes t_{13}(\beta))=\bar\mu_\Gamma(t_{13}(\alpha\beta\delta))=\alpha\beta\delta^2=\bar\mu_\Gamma(t_{13}(\alpha)) \bar\mu_\Gamma(t_{13}(\beta))$. Thus, $\bar\mu_\Gamma$ is an $L_\ring$\=/homomorphism. Since $\delta$ is a unite of the ring $R$, $\bar\mu_\Gamma$ is surjective, and $\alpha\delta=\beta\delta$ implies that $\alpha=\beta$, $\alpha,\beta\in R$, i.\,e., $\bar\mu_\Gamma$ is injective, so $\bar\mu_\Gamma$ is bijective and it is an $L_\ring$\=/isomorphism and $(\Gamma,\phi)$ is a regular interpretation $R\rightsquigarrow G$, as required. 
\end{proof}

We denote by $\Gamma_\param$ any code, which is obtained from the code $\Gamma$ (even if it is not absolute) after adding a tuple $\bar w$ of new fictitious variables (maybe empty) for parameters. We say that $\Gamma_\param$ is a {\em parameter extension} of the code $\Gamma$ by the tuple $\bar w$. In this case $\dim\Gamma_\param=\dim\Gamma$ and $\dim_\param\Gamma_\param\geqslant\dim_\param\Gamma$ and an interpretation $\MA \simeq \Gamma(\MB,\bar p)$ implies that  $\Gamma(\MB)=\Gamma_\param(\MB,\bar p,\bar q)$ for any $\bar q\in \MB$ and therefore $\MA \simeq \Gamma_\param(\MB,\bar p,\bar q)$.

\begin{definition}\label{def:param_ext}
    A regular interpretation $(\Gamma_\param,\phi\wedge\delta)\colon\MA\rightsquigarrow\MB$ is called a {\em parameter extension} of a regular interpretation $\MA\simeq\Gamma(\MB,\phi(\bar y))$, if $\Gamma_\param$ is a parameter extension of the code $\Gamma$ by a tuple  $\bar w$, and $\delta(\bar y,\bar w)$ is an $L(\MB)$\=/formula. In the case, when all free variables of $\delta$ are in $\bar w$, we name the parameters extension $(\Gamma_\param,\phi\wedge\delta)$ {\em split}. For an absolute interpretation $\MA\simeq\Gamma(\MB,\emptyset)$ {\em parameter extension} is a regular interpretation $(\Gamma_\param,\delta)$; it always splits.
\end{definition}

It is clear, that for any regular interpretation $\MA\simeq\Gamma(\MB,\phi(\bar y))$ and any parameter extension $\Gamma_\param$ of the code $\Gamma$ by a tuple $\bar w$, and any $L(\MB)$\=/formula $\delta(\bar y,\bar w)$ the pair $(\Gamma_\param,\phi\wedge\delta)$ is a parameter extension of $\MA\simeq\Gamma(\MB,\phi(\bar y))$ if and only if $\phi\wedge\delta(\MB)\ne \emptyset$. The simple parameter extension above $(\Gamma^\reg,\reg)$ is a particular case of parameter extension.

\subsection{The admissibility conditions for interpretations}\label{subsec:AC}

The book~\cite{Hodges} presented admissibility conditions for (absolute) interpretations as sets of formulas and discussed their significance when approaching interpretations as functors. The specific formulas weren't detailed as they weren't crucial for understanding. However, to understand compositions of interpretations and bi-interpretations, it's necessary to explore admissible conditions further, which we do in this section for arbitrary interpretations (with parameters or not). 

Consider a code 
$$
\Gamma =  \{U_\Gamma(\bar x,\bar y), E_\Gamma(\bar x_1, \bar x_2,\bar y), Q_\Gamma(\bar x_1, \ldots,\bar x_{t_Q},\bar y) \mid Q \in L(\MA)\}
$$ 
for a language $L(\MA)$ into a language $L(\MB)$ with parameters $\bar y=(y_1, \ldots,y_k)$, $k=\dim_{\param}\Gamma$.

Intuitively, the {\em admissibility conditions} $\AC_\Gamma(\bar y)$ for $\Gamma$ is a set of formulas in the language $L(\MB)$ with free variables $\bar y$ such that for an arbitrary $L(\MB)$\=/structure $\nsB$ and arbitrary tuple $\bar q$ in $\nsB$ there exists an $L(\MA)$\=/structure $\Gamma(\nsB,\bar q)$ in $\nsB$ if and only if $\nsB\models \AC_\Gamma(\bar q)$. 
Now we describe formulas in $\AC_\Gamma$ and mention what each one of them ``means''. Here for brevity, we use the notation $U^\wedge_\Gamma$ from Subsection~\ref{subsec:code}.
\begin{enumerate}[label=(\roman*)] 
\item The condition ``$U_\Gamma(\nsB,\bar q)\ne\emptyset$'' is described by the formula
$$
AC_0(\bar y)\:=\:\exists \, \bar x \; U_\Gamma(\bar x, \bar y);
$$
\item The condition ``$E_{\Gamma}(\bar x_1, \bar x_2,\bar q)$ defines an equivalence relation $\sim_\Gamma$ on $U_\Gamma(\nsB,\bar q)$'' is described by formulas 
\begin{gather*}
AC_1(\bar y)\:=\:\forall \, \bar x\: (U_\Gamma(\bar x, \bar y) \,\longrightarrow\, E_\Gamma (\bar x, \bar x, \bar y)\,),\\
AC_2(\bar y)\:=\:\forall\, \bar x_1 \,\forall\, \bar x_2 \: (U^\wedge_\Gamma(\bar x_1,\bar x_2,\bar y)  \wedge E_\Gamma(\bar x_1,\bar x_2, \bar y)\, \longrightarrow \, E_\Gamma (\bar x_2, \bar x_1, \bar y)\,),\\
AC_3(\bar y)\:=\:\forall \, \bar x_1 \,\forall \,\bar x_2 \,\forall \, \bar x_3 \: (U^\wedge_\Gamma(\bar x_1,\bar x_2,\bar x_3,\bar y)\,
\wedge\, E_\Gamma(\bar x_1,\bar x_2,\bar y)\wedge E_\Gamma (\bar x_2, \bar x_3,\bar y) \, \longrightarrow \, E_\Gamma(\bar x_1,\bar x_3,\bar y)\,); 
\end{gather*}
\item ``The interpretation of a symbol $Q\in L(\MA)$ on $U_\Gamma(\nsB,\bar q)/{\sim_\Gamma}$ is well-defined'' is described by 
$$
AC_{Q}(\bar y)\:=\:\forall \, \bar{\bar x}^1 \, \forall\, \bar {\bar x}^2\: (U^\wedge_\Gamma(\bar{\bar x}^1,\bar{\bar x}^2,\bar y)\wedge \,Q_\Gamma (\bar {\bar x}^1, \bar y) \,\wedge 
\bigwedge\limits_{i=1}^{t_Q} E_\Gamma (\bar x_i^1,\bar x_i^2,\bar y) \, \longrightarrow\, Q_\Gamma (\bar {\bar x}^2, \bar y)\,),$$
where $\bar{\bar x}^j=(\bar x_1^j,\ldots,\bar x_{t_Q}^j)$, $j=1,2$;
\item ``Every constant symbol $c\in L(\MA)$ interprets an element from $U_\Gamma(\nsB,\bar q)/{\sim_\Gamma}$'':
\begin{gather*}
 AC_{c,1}(\bar y)\:=\:\exists \, \bar x\; (U_\Gamma (\bar x,\bar y)\wedge c_\Gamma (\bar x, \bar y)\,),\\ 
 AC_{c,2}(\bar y)\:=\:\forall \, \bar x_1 \forall \, \bar x_2 \; (U^\wedge_\Gamma(\bar x_1,\bar x_2,\bar y) \,\wedge \,c_\Gamma (\bar x_1,\bar y) \,\wedge c_\Gamma (\bar x_2, \bar y)\longrightarrow E_\Gamma(\bar x_1, \bar x_2, \bar y)\,);
\end{gather*}
\item ``Every functional symbol $f\in L(\MA)$ interprets a function on $U_\Gamma(\nsB,\bar q)/{\sim_\Gamma}$'':
\begin{gather*}
AC_{f,1}(\bar y)\:=\:\forall \, \bar{\bar x} \; (U^\wedge_\Gamma (\bar {\bar x},\bar y) \,\longrightarrow\, \exists \, \bar u \;(U_\Gamma(\bar u,\bar y) \wedge f_\Gamma (\bar {\bar x},\bar u, \bar y)\,)\,),\\ 
AC_{f,2}(\bar y)\:=\:\forall \, \bar {\bar x}\:\forall \, \bar u\: \forall \, \bar u^\prime\; (U^\wedge_\Gamma (\bar{\bar x},\bar u,\bar u^\prime,\bar y) \,\wedge \,f_\Gamma (\bar{\bar x},\bar u, \bar y) \,\wedge
 \, f_\Gamma (\bar {\bar x},\bar u^\prime, \bar y)\,\longrightarrow\, E_\Gamma(\bar u, \bar u^\prime,\bar y)\,),
\end{gather*}
where $\bar{\bar x}=(\bar x_1,\ldots,\bar x_{n_f})$.
\end{enumerate}

\begin{definition}
For a code $\Gamma\colon L(\MA)\to L(\MB)$ the set of $L(\MB)$\=/formulas $\AC_\Gamma=\{AC_0, AC_1, AC_2, AC_3\} \cup \{AC_c,  AC_{c,1}, AC_{c,2},AC_f, AC_{f,1}, AC_{f,2}, AC_R \mid c,f, R\in L(\MA)\}$ is called the {\em admissibility conditions} of the code $\Gamma$.
\end{definition}

To deal with regular interpretations, we introduce the following set of formulas.

\begin{definition}
Let $\Gamma\colon L(\MA)\to L(\MB)$ be an interpretation code and $\phi(\bar y)$ be an $L(\MB)$\=/formula, $|\bar y|=\dim_{\param}\Gamma$. We define the admissibility conditions $\AC_{\Gamma,\phi}$ as follows: 
$$
\AC_{\Gamma, \phi} = \{AC_{\phi} \mid AC \in \AC_\Gamma\} \cup\{\exists \,\bar y \:\phi(\bar y)\}, \ \mbox{where} \  AC_{\phi} = \forall \,\bar y\:( \phi(\bar y) \,\longrightarrow\, AC(\bar y)).
$$
Thus, all formulas in $\AC_{\Gamma, \phi}$ are sentences in $L(\MB)$.
\end{definition} 

\begin{definition} \label{def:well-defined}
Let $\Gamma\colon L(\MA)\to L(\MB)$ be an interpretation code and $k=\dim_{\param}\Gamma$. For an $L(\MB)$\=/structure $\nsB$ and a $k$\=/tuple $\bar q$ in $\nsB$ we say that $\Gamma(\nsB,\bar q)$ is {\em well-defined}, if $\Gamma(\nsB,\bar q)$ defined in Subsection~\ref{subsec:code} is, indeed, some $L(\MA)$\=/structure. 
\end{definition} 

From the construction of $\AC_\Gamma$, we get the following fact.

\begin{proposition} \label{le:adm-cond}
Let $\Gamma\colon L(\MA)\to L(\MB)$ be an interpretation code and $k=\dim_{\param}\Gamma$. Then for an $L(\MB)$\=/structure $\nsB$ and a $k$\=/tuple $\bar q$ in $\nsB$ the structure $\Gamma(\nsB,\bar q)$ is  well-defined if and only if $\nsB \models \AC_\Gamma(\bar q)$, in which case $(\Gamma,\bar q)$ gives an interpretation $\Gamma(\nsB,\bar q)\stackrel{\Gamma,\bar q}{\rightsquigarrow}\nsB$ with a coordinate map $\mu_\Gamma\colon U_\Gamma(\nsB,\bar q)\to U_\Gamma(\nsB,\bar q)/{\sim_\Gamma}$, that maps a tuple $\bar b\in U_\Gamma(\nsB,\bar q)$ to the equivalence class $\bar b/{\sim_\Gamma}$. In particular, if $(\Gamma,\bar p)$ interprets $\MA$ in $\MB$ then $\MB \models \AC_\Gamma(\bar p)$.
\end{proposition}

Let $\MA \simeq \Gamma(\MB,\bar p)$. If $\MB\models\AC_\Gamma(\bar q)$  for some tuple $\bar q $ in $\MB$ then  $\Gamma(\MB,\bar q)$ is well-defined, but it may not be isomorphic to $\MA$. 

\begin{corollary} 
Let  $\MA\stackrel{\Gamma, \phi}{\rightsquigarrow}\MB$ be  a regular interpretation.  Then $\MB \models  \AC_{\Gamma, \phi}$ and if $\nsB \models \AC_{\Gamma, \phi}$ for an $L(\MB)$\=/structure $\nsB$, then $\phi(\nsB) \neq \emptyset$ and for every tuple $\bar q \in \phi(\nsB)$ the $L(\MA)$\=/structure $\Gamma(\nsB,\bar q)$ is well-defined. 
\end{corollary}

\begin{remark}\label{single_AC}
    If the language $L(\MA)$ is finite, then the set $\AC_\Gamma$ is also finite. In this case one can take a single formula $AC_\Gamma(\bar y)=\bigwedge \{AC(\bar y) \mid AC(\bar y) \in \AC_\Gamma\}$, such that for any $L(\MB)$\=/structure $\nsB$ and $\bar q\in \nsB$ one has $\nsB\models AC_\Gamma(\bar q)$ if and only if $\Gamma(\nsB,\bar q)$ is well-defined. In this case, the set of sentences $\AC_{\Gamma, \phi}$ is also finite and it can be described by a single sentence $AC_{\Gamma,\phi}$~--- the conjunction of the sentences in $\AC_{\Gamma, \phi}$. 
\end{remark}

\subsection{Translation for formulas}\label{subsec:tran}

Let $\Gamma =  \{U_\Gamma(\bar x,\bar y), E_\Gamma(\bar x_1, \bar x_2,\bar y), Q_\Gamma(\bar x_1, \ldots,\bar x_{t_Q},\bar y) \mid Q \in L(\MA)\}$ be an interpretation code from $L(\MA)$ to $L(\MB)$. In this subsection, we describe three types of translations, or reductions,  of formulas from $L(\MA)$ to $L(\MB)$ (perhaps, with extra constants in the language): $\Gamma$\=/{\em translation} $\psi \to \psi_\Gamma$, $(\Gamma,\phi)$\=/{\em translations} $\psi \to \psi_{\Gamma,\forall\phi}$, $\psi \to \psi_{\Gamma,\exists\,\phi}$, and $(\Gamma,\bar p,\mu_\Gamma)$\=/{\em translation} $\psi \to \psi_{\Gamma,\bar p,\mu_\Gamma}$. In Section~\ref{subsec:reduction} we prove the Reduction Theorem and its corollaries that reveal the meaning and purpose of these translations.

We start with the translation $\psi \to \psi_\Gamma$. 
Take a formula $\psi(x_1,\ldots,x_m)$ in $L(\MA)$. We may assume that  $\psi$ does not have variables from $\bar y$ and that it is written in the prenex normal form where the matrix is a disjunction of conjunctions of unnested atomic formulas and their negations; otherwise, we replace $\psi$ with an equivalent formula $\psi^\pre$ in the required form.  Here by an unnested atomic formula, we understand a formula in one of the following forms $R(x_1, \ldots,x_{n_R})$, $f(x_1, \ldots,x_{n_f}) = x_0$, $x_i = x_j$, or $x_i=c$, where $R$ is a predicate symbol from $L(\MA)$, $f$ is a functional symbol from $L(\MA)$, $c$ is a constant from $L(\MA)$, and $x_i$\=/s are variables (see~\cite[Section~2.6]{Hodges}). If the language $L(\MA)$ is computable, then there is an algorithm that, when given $\psi$, computes such a formula $\psi^\pre$ and the length of $\psi^\pre$ is bounded by a linear function of the length of $\psi$. 

Following~\cite{Hodges}, we first define an intermediate map $\psi \to \psi^{\ast}$ as follows. We start with replacing every variable $x_i$ by an $n$\=/tuple of distinct variables $\bar x_i$, where $n = \dim\Gamma$. Then we define the map $\psi \to \psi^{\ast}$ on unnested atomic formulas as 
$$
\left.
\begin{array}{rcl}
x_i=x_j & \longrightarrow & E_\Gamma(\bar x_i,\bar x_j,\bar y),\\
x_i=c & \longrightarrow & c_\Gamma(\bar x_i, \bar y),\\
f(x_1, \ldots,x_{n_f}) = x_0 & \longrightarrow & f_\Gamma (\bar x_1, \ldots, \bar x_{n_f},\bar x_0,\bar y),\\
R(x_1, \ldots,x_{n_R}) & \longrightarrow & R_\Gamma(\bar x_1, \ldots, \bar x_{n_R},\bar y),
\end{array}
\right.
$$
where $|\bar y|=\dim_{\param}\Gamma$ and  $U_\Gamma, E_\Gamma, c_\Gamma, f_\Gamma, R_\Gamma$ are formulas from the code $\Gamma$. Furthermore, we put 
$$
\left.
\begin{array}{rcl}
(\psi_1 \vee \psi_2)^{\ast} & = & \psi^{\ast}_1 \vee \psi^{\ast}_2,\\
(\psi_1 \wedge \psi_2)^{\ast} & = & \psi^{\ast}_1 \wedge \psi^{\ast}_2,\\
(\neg \psi_1)^{\ast} &=& \neg \psi^{\ast}_1,\\
(\exists\, x \: \psi_0)^\ast  & = & \exists\, \bar x\: (U_\Gamma(\bar x, \bar y) \wedge \psi^\ast_0), \\
 (\forall \,x \: \psi_0)^\ast & = & \forall\, \bar x \:(U_\Gamma(\bar x,\bar y) \longrightarrow \psi^\ast_0).
\end{array}
\right.
$$
And finally, we define the $\Gamma$\=/translation $\psi(x_1,\ldots,x_m) \to \psi_\Gamma(\bar x_1,\ldots, \bar x_m, \bar y)$ by
$$
\psi_\Gamma(\bar x_1,\ldots, \bar x_m, \bar y)= \psi^\ast(\bar x_1,\ldots, \bar x_m, \bar y) \wedge U^\wedge_\Gamma(\bar x_1,\ldots,\bar x_m, \bar y).
$$
Here we use the notation $U^\wedge_\Gamma$ from Subsection~\ref{subsec:code} for brevity.

In the book~\cite{Hodges}, the map $\psi \to \psi_\Gamma$ is defined differently with $\psi^*$ denoted as $\psi_\Gamma$. Here, we've slightly modified the definition of $\psi_\Gamma$ to maintain its utility, while facilitating our proofs.

Note that if $\Gamma$ is absolute, i.\,e., $\bar y=\emptyset$, then for every formula $\psi$ in $L(\MA)$ the formula $\psi_\Gamma$ has no variables corresponding to the parameters $\bar y$.

\begin{remark} \label{rem:comp-transl}
If the languages $L(\MA), L(\MB)$ and the interpretation $\Gamma$ are computable, then the map $\psi \to \psi_\Gamma$ is computable. And if the language $L(\MA)$ is finite, then for an $L(\MA)$\=/formula $\psi$ the length of the $L(\MB)$\=/formula $\psi_\Gamma$ is bounded by a linear function of the length of $\psi$.
\end{remark}

\begin{remark}\label{rem:Diophantine}
    If a formula $\psi$ and a code $\Gamma$ are Diophantine, then the formula $\psi_\Gamma$ is Diophantine too.
\end{remark}

\begin{remark}
    If $\psi$ is in (prenex) $\Sigma_n$\=/form ($\Pi_n$\=/form) and $\Gamma$ is $\Sigma_m$\=/code or $\Pi_m$\=/code, then $\psi_\Gamma$ is equivalent to a formula in (prenex) $\Sigma_{n+m+1}$\=/form ($\Pi_{n+m+1}$\=/form). 
\end{remark}

\begin{remark}\label{rem:Id2}
If $\Gamma=\Id_{L(\MA)}$ then  $\psi$ is equivalent to $\psi_\Gamma$ in first-order logic.   
\end{remark}

\begin{remark}\label{remark_Id_AC}
For an interpretation $\Gamma$ the $\Gamma$\=/translation $(\AC_{\Id_{L(\MA)}})_\Gamma = \{\psi_\Gamma \mid \psi \in \AC_{\Id_{L(\MA)}} \}$ of  formulas from the admissibility conditions  $\AC_{\Id_{L(\MA)}}$ could play the role of the admissibility condition set $\AC_\Gamma(\bar y)$. Namely, $(\AC_{\Id_{L(\MA)}})_\Gamma$ satisfies Proposition~\ref{le:adm-cond}, where $\AC_\Gamma(\bar y)$ is replaced by $(\AC_{\Id_{L(\MA)}})_\Gamma$. 
\end{remark}

\begin{definition}\label{def:reg_transl}
    Let $\Gamma\colon L(\MA) \to L(\MB)$ be an interpretation  code and $L(\MB)$\=/formula $\phi(\bar y)$, where $|\bar y|=\dim_{\param}\Gamma$. Then for a formula $\psi$ in $L(\MA)$  we define    $(\Gamma,\phi)$\=/translations $\psi\to\psi_{\Gamma,\forall\phi}$ and  $\psi\to\psi_{\Gamma,\exists\,\phi}$ as
$$
\psi_{\Gamma,\forall\phi}= \forall \,\bar y \:(\phi(\bar y) \longrightarrow  \psi_\Gamma(\bar y)\,),\quad \psi_{\Gamma,\exists\,\phi}= \exists \,\bar y \:(\phi(\bar y) \wedge  \psi_\Gamma(\bar y)\,).
$$
If $\Gamma$ is an absolute code, then we assume that $\psi_{\Gamma,\forall\phi}$ and $\psi_{\Gamma,\exists\,\phi}$ are just $\psi_\Gamma$.
\end{definition}

In particular, if $\psi$ is a sentence, then $\psi_{\Gamma,\forall\phi}$, $\psi_{\Gamma,\exists\,\phi}$ are sentences too. 

\begin{remark} \label{rem:comp-transl-2}
If the languages $L(\MA),L(\MB)$ and the code $\Gamma$ are computable, then the maps $\psi \to \psi_{\Gamma,\forall\phi}$, $\psi \to \psi_{\Gamma,\exists\,\phi}$ are computable too.  
\end{remark}

\begin{remark} 
If the code $\Gamma$ and formulas $\phi,\psi$ are Diophantine, the formula $\psi_{\Gamma,\exists\,\phi}$ is Diophantine too.  
\end{remark}

Let us note that the $\Gamma$\=/translation and the $(\Gamma,\phi)$\=/translations of formulas depend solely on a code $\Gamma$ and a formula $\phi$, and they have no connection to a specific interpretation $\MA\rightsquigarrow \MB$. The third translation, $\psi\to\psi_{\Gamma,\bar p,\mu_\Gamma}$, which we will consider, translates formulas from $L(\MA)\cup A$ into formulas in $L(\MB)\cup B$. This translation depends on the interpretation $(\Gamma,\bar p, \mu_\Gamma)$ of the algebraic structure $\MA=\langle A;L(\MA)\rangle$ into the algebraic structure $\MB=\langle B;L(\MB)\rangle$. As before, for every element $a\in A$ we denote by $\bar\mu^{-1}_\Gamma(a)$ some fixed tuple $\bar b$ from $\mu_\Gamma^{-1}(a)$.

\begin{definition}
    For any formula $\psi(x_1,\ldots,x_m, a_1,\ldots,a_s)$ in $L(\MA)\cup A$ with free variables $x_1,\ldots,x_m$ and parameters $a_1,\ldots,a_s\in A$ we put 
$$
\psi_{\Gamma,\bar p, \mu_\Gamma}(\bar x_1,\ldots,\bar x_m)=\psi_\Gamma(\bar x_1,\ldots,\bar x_m,\bar\mu_\Gamma^{-1}(a_1),\ldots,\bar\mu_\Gamma^{-1}(a_s), \bar p),
$$
so, $\psi_{\Gamma,\bar p, \mu_\Gamma}$ is a formula in $L(\MB)\cup B$. 
\end{definition}

Note that the formula $\psi_{\Gamma,\bar p, \mu_\Gamma}$ depends on the choice of fixed tuples $\bar b$ from preimages $\mu_\Gamma^{-1}(a)$, $a\in A$, but, as we will see later, the set $\psi_{\Gamma,\bar p, \mu_\Gamma}(\MB_B)\subseteq B^{nm}$ does not depend on this choice (see Lemma~\ref{cor2}), and this is enough for us to work with the $(\Gamma,\bar p,\mu_\Gamma)$\=/translation.

If the code $\Gamma$ is absolute, then for any formula $\psi$ in $L(\MA)$ one has $\psi_{\Gamma,\emptyset, \mu_\Gamma}=\psi_\Gamma$.

\subsection{Reduction Theorem}\label{subsec:reduction}

In Section~\ref{subsec:tran}, translations of formulas via a code $\Gamma\colon L(\MA)\to L(\MB)$ were defined. In this section, we prove the Reduction Theorem, which, together with its implications, reveals the power of these translations. This theorem is a fundamental result in interpretability, appearing in various forms in the literature, such as in~\cite[Theorem~5.3.2]{Hodges}. 

We start with the following notations.

\begin{notation}
Assume that $\Gamma\colon L(\MA)\to L(\MB)$ is a code, $\psi(x_1,\ldots,x_m)$ is a formula in $L(\MA)$ and $\psi^+(\bar x_1,\ldots,\bar x_m,\bar y)$ is a formula in $L(\MB)$, where $|\bar y|=\dim_\param\Gamma$, $|\bar x_i|=\dim\Gamma$, $i=1,\ldots,m$. Then we write
$$
\psi \Longleftrightarrow_\Gamma \psi^+
$$
if and only if for any interpretation $(\Gamma,\bar q,\mu_\Gamma)$ of an $L(\MA)$\=/structure $\nsA$ in an $L(\MB)$\=/structure $\nsB$ one has 
$$
\mu_\Gamma^{-1}(\psi(\nsA))=\psi^+(\nsB,\bar q),
$$
or, equivalently, if and only if for any $L(\MA)$\=/structure $\nsA$, any $L(\MB)$\=/structure $\nsB$ and any tuple $\bar q\in \nsB$ with interpretation $\nsA\simeq \Gamma(\nsB,\bar q)$
\begin{enumerate}[label=\arabic*)]
\item if $\MA\models \psi(\bar a)$, $\bar a=(a_1,\ldots,a_m)\in \nsA$, then for any coordinate map $\mu_\Gamma$ of the interpretation $\nsA\simeq \Gamma(\nsB,\bar q)$ one has $\nsB\models \psi^+(\mu_\Gamma^{-1}(a_1),\ldots, \mu_\Gamma^{-1}(a_m),\bar q)$, i.\,e., $\MB\models \psi^+(\bar b_1,\ldots, \bar b_m,\bar q)$ for any elements $\bar b_i\in \mu_\Gamma^{-1}(a_i)$,
\item and conversely, if $\MB\models \psi^+(\bar b_1,\ldots, \bar b_m,\bar q)$, $\bar b_1,\ldots,\bar b_m\in \nsB$, then $\bar b_1,\ldots,\bar b_m\in U_\Gamma(\nsB,\bar q)$ and for any coordinate map $\mu_\Gamma$ of the interpretation $\nsA\simeq \Gamma(\nsB,\bar q)$ one has $\nsA\models \psi(\mu_\Gamma(\bar b_1),\ldots,\mu_\Gamma(\bar b_m))$.
\end{enumerate}
\end{notation}

\begin{theorem}[The Reduction Theorem]\label{le:interpr_corol} 
For a code $\Gamma\colon L(\MA)\to L(\MB)$ and a formula $\psi(x_1,\ldots,x_m)$ in $L(\MA)$ one has
$$
\psi\Longleftrightarrow_{\Gamma} \psi_\Gamma.
$$
\end{theorem}

\begin{proof}
By induction on the construction of $\psi$ and the definition of the interpretation, it is easy to see that for any interpretation $\nsA\simeq\Gamma(\nsB,\bar q)$, condition $\nsA\models \psi(\bar a)$, $\bar a\in \nsA$, implies $\nsB\models \psi^\ast(\bar{\bar b},\bar q)$ for any $\mu_\Gamma$ and $\bar{\bar b}\in \mu_\Gamma^{-1}(\bar a)$; therefore, $\nsB\models \psi_\Gamma(\bar{\bar b},\bar q)$. And inversely, for any $\bar{\bar b}\in \nsB$ condition $\nsB\models \psi_\Gamma(\bar{\bar b},\bar q)$ means that $\nsB\models \psi^\ast(\bar{\bar b},\bar q)$ and $\bar{\bar b}\in U^\wedge_\Gamma(\nsB,\bar q)$, therefore, for any $\mu_\Gamma$ one has $\nsA\models \psi (\mu_\Gamma(\bar{\bar b}))$, hence the required.
\end{proof}

Sometimes it is convenient to have the Reduction Theorem~\ref{le:interpr_corol} in the following slightly different formulation.

\begin{corollary}\label{le:interpr_corol_ns}
Let $\Gamma\colon L(\MA)\to L(\MB)$ be a code, $\nsB$ be an $L(\MB)$\=/structure and $\bar q\in \nsB$ be a tuple of elements, such that $\nsB\models\AC_{\Gamma}(\bar q)$, i.\,e., the $L(\MA)$\=/structure $\Gamma(\nsB,\bar q)$ is well-defined. Then for any formula $\psi(x_1,\ldots,x_m)$ in $L(\MA)$ and any tuples $\bar b_1,\ldots,\bar b_m$ of elements from $\nsB$ of length $n=\dim\Gamma$ one has
$$
\Gamma(\nsB,\bar q)\models \psi(\bar b_1/{\sim_\Gamma},\ldots,\bar b_m/{\sim_\Gamma})\iff \nsB\models \psi_\Gamma(\bar b_1,\ldots,\bar b_m,\bar q).
$$
\end{corollary}

Formally, the Reduction Theorem~\ref{le:interpr_corol}  is proven only for formulas $\psi$ in the prenex normal form. However, we can eliminate this restriction and, at the same time, show that the $\Gamma$\=/translation does not depend on the way the normal form is written. The following notation will be useful.

\begin{notation}
Assume that $\Gamma\colon L(\MA)\to L(\MB)$ is a code and $\varphi_1(\bar {\bar x},\bar y)$, $\varphi_2(\bar {\bar x},\bar y)$ are formulas in $L(\MB)$, where $\bar{\bar x}=(\bar x_1,\ldots,\bar x_m)$, $|\bar x_i|=\dim \Gamma$, $i=1,\ldots,m$, $|\bar y|=\dim_{\param}\Gamma$. Then we write 
$$
\varphi_1\longleftrightarrow_\Gamma \varphi_2
$$ 
if and only if for any interpretation $(\Gamma,\bar q)$ of an $L(\MA)$\=/structure $\nsA$ in an $L(\MB)$\=/structure $\nsB$ one has 
$$
\varphi_1(\nsB,\bar q)\cap U^\wedge_\Gamma(\nsB,\bar q)=\varphi_2(\nsB,\bar q)\cap U^\wedge_\Gamma(\nsB,\bar q),
$$
or, equivalently, 
$$
\AC_{\Gamma}(\bar y)\cup \{U_\Gamma(\bar x_1,\bar y),\ldots,U_\Gamma(\bar x_m,\bar y)\}\vdash \varphi_1(\bar {\bar x},\bar y) \longleftrightarrow \varphi_2(\bar {\bar x},\bar y).
$$
It is obvious that $\longleftrightarrow_\Gamma$ is an equivalence relation on the set of all $L(\MB)$\=/formulas $\varphi(\bar{\bar x},\bar y)$.
\end{notation}

\begin{corollary}\label{cor_form}
Let $\Gamma\colon L(\MA)\to L(\MB)$ be a code, $\psi(x_1,\ldots,x_m)$ be a formula in $L(\MA)$ and $\psi^1,\psi^2$ be $L(\MA)$\=/formulas in prenex normal form with matrix as disjunction of conjunctions of unnested atomic formulas and their negations, both equivalent to $\psi$. Then  one has 
$\psi^1_\Gamma\longleftrightarrow_\Gamma \psi^2_\Gamma$.
\end{corollary}

\begin{proof}
For any $L(\MB)$\=/structure $\nsB$ and any tuple $\bar q\in \nsB$, such that $\nsB\models\AC_\Gamma(\bar q)$,
an  $L(\MA)$\=/structure $\Gamma(\nsB,\bar q)$ is well-defined and interpretable in $\nsB$ by the code $\Gamma$ and parameters $\bar q$. Since $\Gamma(\nsB,\bar q)\models\forall\,{\bar x} \:(\psi^1({\bar x})\longleftrightarrow \psi^2({\bar x}))$, we obtain the required by Corollary~\ref{le:interpr_corol_ns}.
\end{proof}

\begin{corollary}\label{cor1}
Let $\Gamma\colon L(\MA)\to L(\MB)$ be a code and $\psi,\psi^1,\psi^2$ are formulas in $L(\MA)$. Then one has 
$$
\left.
\begin{array}{rcl} 
(\neg\psi)_\Gamma & \longleftrightarrow_\Gamma & \neg \psi_\Gamma,\\
(\psi^1\wedge\psi^2)_\Gamma &\longleftrightarrow_\Gamma &\psi^1_\Gamma\wedge\psi^2_\Gamma, \\
(\psi^1\vee\psi^2)_\Gamma  &\longleftrightarrow_\Gamma &\psi^1_\Gamma\vee\psi^2_\Gamma.
\end{array}
\right.
$$
\end{corollary}

\begin{proof}
Indeed, take any $L(\MB)$\=/structure $\nsB$ and any tuples $\bar{\bar b},\bar q\in \nsB$, such that $\nsB\models\AC_\Gamma(\bar q)$ and $\bar{\bar b}\in U_\Gamma^\wedge(\nsB,\bar q)$. If, say, $\nsB\models \psi^1_\Gamma(\bar{\bar b},\bar q)\wedge\psi^2_\Gamma (\bar{\bar b},\bar q)$, then $\Gamma(\nsB,\bar q)\models \psi^1(\bar {\bar b}/{\sim_\Gamma})$ and $\Gamma(\nsB,\bar q)\models \psi^2(\bar {\bar b}/{\sim_\Gamma})$, i.\,e., $\Gamma(\nsB,\bar q)\models \psi^1\wedge\psi^2(\bar {\bar b}/{\sim_\Gamma})$, therefore, $\nsB\models (\psi^1\wedge\psi^2)_\Gamma(\bar{\bar b},\bar q)$. And so on.
\end{proof}

One can use $\Gamma$\=/translation and the Reduction Theorem to establish decidability/undecidability of the Diophantine problems (see Subsection~\ref{subsec:Diophantine}). However, to prove
decidability/undecidability of elementary theories one can  use $(\Gamma,\phi)$\=/translations and the lemma below (see Subsection~\ref{subsec:undecidability}).

\begin{lemma}  \label{le:interpr_corol_2} 
Let $\MA \simeq  \Gamma(\MB,\phi)$ be a regular interpretation of $\MA$ in $\MB$. Then for every sentence  $\psi$ of $L(\MA)$ the following holds 
$$
\MA\models \psi \iff \MB\models \psi_{\Gamma,\forall\phi}\iff \MB\models\psi_{\Gamma,\exists\,\phi}.
$$
\end{lemma}

\begin{proof}
    This immediately follows from the Reduction Theorem~\ref{le:interpr_corol}.
\end{proof}

Now we state some corollaries of the  Reduction Theorem for $(\Gamma,\bar p,\mu_\Gamma)$\=/translations.

\begin{lemma}\label{cor2}
Let $(\Gamma,\bar p,\mu_\Gamma)$ be an interpretation of an algebraic structure $\MA=\langle A; L(\MA)\rangle$ into an algebraic structure $\MB=\langle B; L(\MB)\rangle$, $n=\dim\Gamma$. Then if $X\subseteq A^m$ is a definable subset in $\MA$, then the set $\mu_\Gamma^{-1}(X)\subseteq B^{m\cdot n}$ is definable in $\MB$. Namely, if $X\subseteq A^m$ is defined by a formula \ $\psi(x_1, \ldots, x_m)$ in $L(\MA)\cup A$, then $\mu_\Gamma^{-1}(X)$ is defined by the formula $\psi_{\Gamma,\bar p, \mu_\Gamma}(\bar x_1, \ldots, \bar x_m)$ in $L(\MB)\cup B$. In particular, the set $\psi_{\Gamma,\bar p, \mu_\Gamma}(\MB_B)$ does not depend on the choice of tuples from preimages $\mu_\Gamma^{-1}(a)$, $a\in A$, despite the fact that the formula $\psi_{\Gamma,\bar p, \mu_\Gamma}$ itself depends on them.
\end{lemma}

\begin{proof}
Immediately follows the Reduction Theorem~\ref{le:interpr_corol}.
\end{proof}

The result above shows that every subset of $U^m_\Gamma(\MB,\bar p)$ definable in $\Gamma(\MB,\bar p)$ is definable in $\MB$. However, the converse is not true in general, i.\,e., a subset of $U^m_\Gamma(\MB,\bar p)$ definable in $\MB$ may not be definable in $\Gamma(\MB,\bar p)$. In the case, when every subset of $U^m_\Gamma(\MB,\bar p)$, $m\in \N$, definable in $\MB$ is definable in $\Gamma(\MB,\bar p)$, the interpretation $\Gamma(\MB,\bar p)\rightsquigarrow\MB$ is called {\em pure}~\cite{Marker}.

\begin{corollary}\label{cor2.0}
Let $(\Gamma, \emptyset)\colon\MA\rightsquigarrow\MB$ be an absolute interpretation  and $X\subseteq A^m$ a $0$\=/definable subset in $\MA$. Then for any coordinate map $\mu_\Gamma$ of interpretation $\Gamma$ the set $\mu_\Gamma^{-1}(X)\subseteq B^{m\cdot n}$ is $0$\=/definable in $\MB$; moreover, the set $\mu_\Gamma^{-1}(X)$ does not depend on $\mu_\Gamma$, i.\,e., for any other coordinate map $\mu_{\Gamma 0}$ of $\Gamma$ one has $\mu_\Gamma^{-1}(X)=\mu_{\Gamma0}^{-1}(X)$.
\end{corollary}

\begin{remark}\label{rem:WSOL-def}
If $\MA\simeq\Gamma(\MB)$ is a $0$\=/interpretation, then $0$\=/definable subsets of $A^m$ are not the only sets whose preimages under $\mu_\Gamma$ do not depend on the coordinate map $\mu_\Gamma$. Moreover, any unions, intersections of such sets, and any sets absolutely definable by formulas in weak second-order logic also satisfy this property~\cite[Theorem~4.7]{KhMS}.
\end{remark}

The following corollary of the Reduction Theorem will be useful for us later when constructing compositions of interpretations.

\begin{lemma}\label{psi+}
Let $\Gamma\colon L(\MA)\to L(\MB)$ be a code, $n=\dim\Gamma$. Then for any formula $\psi(x_1,\ldots,x_m)$ in $L(\MA)$ of the type 
\begin{equation}\label{AC}
\forall \, \bar t \; (\bigwedge_{i=1}^r\varphi^i(\bar t,\bar x) \:\longrightarrow \: \exists \, \bar z \:  (\bigwedge_{j=1}^h\vartheta^j(\bar z,\bar t,\bar x)\,) \,), \ \bar x=(x_1,\ldots,x_m), 
\end{equation}
where $\varphi^i,\vartheta^j$ are any formulas in $L(\MA)$, one has 
$$
\psi\Longleftrightarrow_{\Gamma} \psi^+,
$$
where $\psi^+(\bar x_1,\ldots,\bar x_m,\bar y)$ is an $L(\MB)$\=/formula defined by
$$
(\forall \, \bar{\bar t}\; ( \bigwedge_{i=1}^r\varphi^i_\Gamma(\bar{\bar t},\bar {\bar x}, \bar y) \:\longrightarrow\: \exists \,\bar {\bar z}\: (\bigwedge_{j=1}^h\vartheta^j_\Gamma(\bar{\bar z},\bar{\bar t},\bar{\bar x}, \bar y)) \,)\,)\,\wedge\, U^\wedge_\Gamma(\bar {\bar x}, \bar y), \quad \bar{\bar x}=(\bar x_1,\ldots,\bar x_m). 
$$
\end{lemma}

\begin{proof}
Due to Corollary~\ref{cor1} the conjunctions $\bigwedge_{i=1}^r\varphi^i$ and $\bigwedge_{j=1}^h\vartheta^j$ can be replaced by single formulas $\varphi$ and $\vartheta$ without loss of generality. Fix an interpretation $(\Gamma,\bar q,\mu_\Gamma)$ of an $L(\MA)$\=/structure $\nsA$ into an $L(\MB)$\=/structure $\nsB$. 

Assume that $\bar a\in \nsA$ and $\nsA\models\psi(\bar a)$. Take any tuples $\bar {\bar b}\in \mu_\Gamma^{-1} (\bar a)$ and $\bar t_1,\ldots,\bar t_s\in \nsB$. If $\nsB\models  \varphi_\Gamma(\bar t_1,\ldots,\bar t_s,\bar {\bar b}, \bar q)$ then $\nsA\models  \varphi(\mu_\Gamma(\bar t_1),\ldots,\mu_\Gamma(\bar t_s),\bar a)$, therefore, $\nsA\models \vartheta(\bar c,\mu_\Gamma(\bar t_1),\ldots,\mu_\Gamma(\bar t_s), \bar a)$ for some $\bar c\in \nsA$. Thus, $\nsB\models \vartheta_\Gamma(\bar {\bar d},\bar t_1,\ldots,\bar t_s, \bar {\bar b},\bar q)$ with $\bar {\bar d}\in\mu_\Gamma^{-1}(\bar c)$, so $\nsB\models\psi^+(\bar {\bar b},\bar q)$.

Conversely, let $\nsB\models\psi^+(\bar {\bar b},\bar q)$ for some tuple $\bar {\bar b}\in \nsB$. Then $\bar {\bar b}\in U^\wedge_\Gamma(\nsB, \bar q)$. Suppose that $\nsA\models  \varphi(\bar t,\mu_\Gamma(\bar {\bar b}))$ for some tuple $\bar t\in \nsA$. In this case $\nsB\models  \varphi_\Gamma(\bar {\bar u},\bar {\bar b}, \bar q)$ with $\bar {\bar u}\in \mu_\Gamma^{-1}(\bar t)$. Therefore, $\nsB\models \vartheta_\Gamma(\bar {\bar d},\bar {\bar u}, \bar {\bar b},\bar q)$ for some tuple $\bar {\bar d}\in U^\wedge_\Gamma(\nsB,\bar q)$. Thus $\nsA\models \vartheta(\mu_\Gamma(\bar {\bar d}),\bar t, \mu_\Gamma(\bar {\bar b}))$, so $\nsA\models \psi(\mu_\Gamma(\bar {\bar b}))$, as required. 
\end{proof}

\begin{remark}
Note that every formula from the admissibility conditions $\AC_\Gamma$ is equivalent to a conjunction of formulas of type~\eqref{AC}.
\end{remark}

\section{Homotopic interpretations}

In this and the next sections, we will focus on interactions between several interpretations. Homotopic interpretations form ``parallel connections'', while invertible and bi-interpretations form ``sequential connections'' of interpretations.

\subsection{Homotopy with parameters}\label{subsec:homotopy}

Let 
\begin{gather*}
\Gamma_1 =  \{U_{\Gamma_1}(\bar x_0,\bar y), E_{\Gamma_1}(\bar x_1, \bar x_2,\bar y), Q_{\Gamma_1}(\bar x_1, \ldots,\bar x_{t_Q},\bar y) \mid Q \in L(\MA)\},\\
\Gamma_2 =  \{U_{\Gamma_2}(\bar v_0,\bar z), E_{\Gamma_2}(\bar v_1, \bar v_2,\bar z), Q_{\Gamma_2}(\bar v_1, \ldots,\bar v_{t_Q},\bar z) \mid Q \in L(\MA)\}
\end{gather*}
be codes, consisting of $L(\MB)$\=/formulas, where $|\bar x_i|=\dim\Gamma_1$, $|\bar v_i|=\dim\Gamma_2$, $|\bar y|=\dim_{\param}\Gamma_1$, $|\bar z|=\dim_{\param}\Gamma_2$.

\begin{definition}
Two interpretations $(\Gamma_1,\bar p_1)$ and $(\Gamma_2,\bar p_2)$ of an algebraic structure $\MA$ into an algebraic structure $\MB=\langle B;L(\MB)\rangle$ are called {\em homotopic} with parameters, if there is an $L(\MA)$\=/isomorphism $\lambda\colon\Gamma_1(\MB,\bar p_1)\to\Gamma_2(\MB,\bar p_2)$, definable with parameters in $\MB$ by a formula $\theta(\bar x_1,\bar x_2)$ in the language $L(\MB)\cup B$ (see Definition~\ref{def:definable_map}). If $\Gamma_1$, $\Gamma_2$ are absolute and $\lambda$ is $0$\=/definable, then $\Gamma_1$ and $\Gamma_2$ are called {\em absolutely homotopic}. If $\mu_{\Gamma_1}$ and $\mu_{\Gamma_2}$ are fixed coordinate maps such that $\lambda=\bar\mu^{-1}_{\Gamma_2}\circ\bar\mu_{\Gamma_1}$, i.\,e., the diagram 
\begin{equation*}
\begin{tikzcd}[column sep=large, row sep=large]
& \Gamma_1(\MB,\bar p_1) \ar[dl, "\bar{\mu}_{\Gamma_1}"'] 
\ar[dd, shift right, "\lambda"']\\
\MA  & \\
& \Gamma_2(\MB,\bar p_2) 
\ar[ul, "\bar{\mu}_{\Gamma_2}"] 
\ar[uu, shift right, "\lambda^{-1}"']
\end{tikzcd} 
\end{equation*} 
is commutative, then we call interpretations (coordinatizations)  $(\Gamma_1,\bar p_1,\mu_{\Gamma_1})$ and $(\Gamma_2,\bar p_2, \mu_{\Gamma_2})$ {\em strongly homotopic}.

We refer to the isomorphism $\lambda$ above as a {\em homotopy isomorphism}, and to formula $\theta$ as a  {\em homotopy}, or a {\em connector}, or a {\em connecting formula} between $(\Gamma_1,\bar p_1)$ and $(\Gamma_2,\bar p_2)$, and write $(\theta)\colon (\Gamma_1,\bar p_1)\to (\Gamma_2,\bar p_2)$ or  $(\theta)\colon \Gamma_1(\MB,\bar p_1)\to \Gamma_2(\MB,\bar p_2)$. In the case of a strong homotopy we write $(\theta)\colon (\Gamma_1,\bar p_1,\mu_{\Gamma_1})\to (\Gamma_2,\bar p_2,\mu_{\Gamma_2})$.  We denote by $\theta^{-1}(x_1,x_2)$ the formula $\theta(x_2,x_1)$, which defines the homotopy $(\theta^{-1})\colon \Gamma_2(\MB,\bar p_2)\to \Gamma_1(\MB,\bar p_1)$.
\end{definition}

The following is a useful observation. 

\begin{remark}
If codes $\Gamma_1$ and $\Gamma_2$ are equivalent (see Definition~\ref{def:equiv_codes}) and there exist interpretations $\MA\simeq\Gamma_1(\MB,\bar p)$ and $\MA\simeq\Gamma_2(\MB,\bar p)$ then they are homotopic with the connector $\theta(\bar x,\bar v,\bar p)=U_{\Gamma_1}(\bar x,\bar p)\wedge U_{\Gamma_2}(\bar v,\bar p)\wedge E_{\Gamma_1}(\bar x,\bar v,\bar p)$.
\end{remark}

\begin{example}
    Two absolute interpretations $\Gamma$ and $\Delta$ in Example~\ref{ex:UT1} that interpret $\UT_3(R)$ in $R$ are absolutely homotopic by means of a homotopy $(\theta)\colon\Gamma\to\Delta$, where $\theta(a,b,c,\alpha,\beta,\gamma)=(\alpha=a)\wedge(\beta=b)\wedge (\gamma=c-ab)$.
\end{example}

There are many known examples of homotopic interpretations, and only a  few of non-homotopic ones. For the latter, see Example~\ref{ex:G2} below. 

\begin{remark}
Let $(\theta) \colon \Gamma_1(\MB,\bar p_1)\to \Gamma_2(\MB,\bar p_2)$ be a homotopy, where $\theta = \theta(\bar x,\bar v, \bar p)$, $|\bar x|=\dim\Gamma_1$, $|\bar v|=\dim\Gamma_2$, and $\bar p$~--- parameters from $\MB$. It is possible to assume that $\bar p=(\bar p_1,\bar p_2)$ by introducing additional fictitious variables for parameters in $\theta$ or within the formulas of the codes $\Gamma_1$ and $\Gamma_2$. Since it might not always be preferable to modify the codes $\Gamma_1$ and $\Gamma_2$, we may choose to adjust only $\theta$. Therefore, the parameters in $\theta$  are assumed to take the form $(\bar p_1, \bar p_2, \bar r)$.
\end{remark}

\begin{remark}\label{rem:strong}
Let  $(\theta)\colon (\Gamma_1,\bar p_1)\to (\Gamma_2,\bar p_2)$ be a homotopy with a homotopy isomorphism $\lambda$. Then for any coordinate maps  $\mu_{\Gamma_1}$ and $\mu_{\Gamma_2}$ one can either replace $\mu_{\Gamma_1}$ by the composition $\mu_{\Gamma_1}^\prime = \bar \mu_{\Gamma_2} \circ \lambda\circ \bar{\mu}_{\Gamma_1}^{-1}\circ\mu_{\Gamma_1}$ or replace  $\mu_{\Gamma_2}$ by $\mu_{\Gamma_2}^\prime = \bar \mu_{\Gamma_1} \circ \lambda^{-1}\circ \bar{\mu}_{\Gamma_2}^{-1}\circ\mu_{\Gamma_2}$ such that $(\theta)\colon (\Gamma_1,\bar p_1,\mu_{\Gamma_1}^\prime)\to (\Gamma_2,\bar p_2,\mu_{\Gamma_2})$ and $(\theta)\colon (\Gamma_1,\bar p_1,\mu_{\Gamma_1})\to (\Gamma_2,\bar p_2,\mu_{\Gamma_2}^\prime)$ are strong homotopies. 
\end{remark}

\begin{remark}
    If an algebraic structure $\MA$ is rigid, then any interpretation $(\Gamma,\bar p)\colon \MA\rightsquigarrow\MB$ has a unique coordinate map $\mu_\Gamma$; therefore, for any homotopy  $(\theta)\colon (\Gamma_1,\bar p_1) \to (\Gamma_1,\bar p_1)$  of any interpretations $\MA \stackrel{\Gamma_1,\bar p_1}{\rightsquigarrow} \MB$, $\MA \stackrel{\Gamma_2,\bar p_2}{\rightsquigarrow} \MB$  with coordinate maps $\mu_{\Gamma_1}$ and $\mu_{\Gamma_2}$  one has that $(\theta)\colon (\Gamma_1,\bar p_1,\mu_{\Gamma_1}) \to (\Gamma_2,\bar p_2,\mu_{\Gamma_2})$ is a strong homotopy.
\end{remark}

\begin{remark}
There exists a homotopy $(\theta)$ between a self-interpretation $(\Gamma,\bar p)\colon \MA\rightsquigarrow\MA$ and the identical interpretation $\Id_{L(\MA)}\colon \MA\rightsquigarrow\MA$ if and only if there is a definable coordinate map $\mu_\Gamma\colon U_\Gamma(\MA,\bar p)\to A$. In this case the homotopy $(\theta)\colon (\Gamma,\bar p,\mu_\Gamma)\to(\Id_{L(\MA)},\emptyset,\mathrm{id}_\MA)$ is strong.
\end{remark}

\begin{definition}
We name an algebraic structure $\MA=\langle A; L(\MA)\rangle$ {\em i\=/rigid}, if every  self-interpretation $\MA\rightsquigarrow\MA$ is  homotopic to $\Id_{L(\MA)}$. 
\end{definition}

For example, the ring of integers $\Z$ is i\=/rigid~\cite[Lemma~2.7]{AKNS}.

Denote by 
\begin{itemize}
    \item $\Codes(\MA,\MB)$ the set of all interpretations $(\Gamma,\bar p)$ of $\MA$ in $\MB$;
    \item $\Codes^+(\MA,\MB)$ the set of all coordinatizations $(\Gamma,\bar p,\mu_\Gamma)$ of all interpretations  of $\MA$ in $\MB$;
    \item $\Codes_\abs(\MA,\MB)$ the set of all absolute interpretations $(\Gamma,\emptyset)$ of $\MA$ in $\MB$;
    \item $\Codes_\abs^+(\MA,\MB)$ the set of all coordinatizations $(\Gamma,\emptyset,\mu_\Gamma)$ of all  absolute interpretations  of $\MA$ in $\MB$;
    \item $\Codes_\reg(\MA,\MB)$ the set of all regular interpretations $(\Gamma,\phi)$ of $\MA$ in $\MB$.
\end{itemize}

\begin{lemma}\label{lemma:hom_par}
Let $\MA$ and $\MB$ be algebraic structures. Then the following holds:
\begin{enumerate}[label=\arabic*)]
    \item homotopy is an equivalence relation on the set $\Codes(\MA,\MB)$;
    \item strong homotopy is an equivalence relation on the set $\Codes^+(\MA,\MB)$;
    \item $0$\=/homotopy is an equivalence relation on the set $\Codes_{\abs}(\MA,\MB)$;
    \item strong $0$\=/homotopy is an equivalence relation on the set $\Codes_{\abs}^+(\MA,\MB)$. 
\end{enumerate}
\end{lemma}

\begin{proof} 
Straightforward.
\end{proof}

The following fact is a direct continuation of Remark~\ref{lem:AB1}.

\begin{remark}\label{lem:AB2}
Let $\MA$ and $\MB$ be algebraic structures. If $(\theta)\colon(\Gamma,\bar p,\mu_{\Gamma}) \to (\Delta, \bar q,\mu_{\Delta})$ is a  strong homotopy of interpretations of $\MA$ into $\MB$, then $(\theta)\colon (\Gamma_{A}, \emptyset,\mu_{\Gamma}) \to (\Delta_{A},\emptyset,\mu_{\Delta})$ is  a strong $0$\=/homotopy of interpretations of $\MA_A$ into $\MB_B$. Furthermore, if the language $L(\MA)$ is finite, then the inverse statement also holds. However, if a homotopy $(\theta)\colon (\Gamma,\bar p) \to (\Delta, \bar q)$ does not give a strong homotopy of the coordinatizations $(\theta)\colon (\Gamma,\bar p,\mu_{\Gamma}) \to (\Delta, \bar q,\mu_{\Delta})$ then  we cannot claim that $\theta$ defines a homotopy  $(\theta)\colon (\Gamma_{A}, \emptyset,\mu_{\Gamma}) \to (\Delta_{A},\emptyset,\mu_{\Delta})$.
\end{remark}

\begin{lemma}\label{lem:cor4}
Let $h\in \Aut(\MB)$ be a definable automorphism of $\MB$ and $\theta_\Aut(v_1,v_2)$ be formula in $L(\MB)\cup B$, which defines $h$. Then any interpretations $(\Gamma,\bar p_1)\colon \MA\rightsquigarrow \MB$ and $(\Gamma,\bar p_2)\colon \MA\rightsquigarrow \MB$ with common code $\Gamma$ and $h(\bar p_1)=\bar p_2$ are homotopic with connector
$$
\theta(\bar x_1,\bar x_2)=\exists\,\bar v^1\:\exists\,\bar v^2\;(\bigwedge_{i=1}^2(
U^\wedge_\Gamma(\bar x_i,\bar v^i,\bar p_i)\,\wedge\, E_\Gamma(\bar x_i,\bar v^i,\bar p_i))\wedge \bigwedge_{j=1}^n\theta_{\Aut}(v^1_j,v^2_j)\,), 
$$
where $\bar v^i=(v^i_1,\ldots,v^i_n)$,  $n=\dim\Gamma$, $i=1,2$.
\end{lemma}

\begin{proof}
Straightforward.
\end{proof}

In a later paper in this series, we will provide specific conditions, referred to as homotopy conditions, that assert the formula $\theta$ defines an isomorphism.

\subsection{Regular homotopy}\label{subsec:homotopy_reg}

A regular analog of homotopy is the following.

\begin{definition}\label{def:reg_hom}
Two regular interpretations $(\Gamma_1,\phi_1)$ and $(\Gamma_2,\phi_2)$ of an algebraic structure $\MA$ into an algebraic structure $\MB$ are called {\em regularly homotopic}, if there is an $L(\MB)$\=/formula $\theta(\bar x_1,\bar x_2, \bar y_1,\bar y_2,\bar w)$ (a {\em connector} or a {\em connecting formula}) and an $L(\MB)$\=/formula $\delta(\bar y_1,\bar y_2,\bar w)$ (a {\em descriptor} or, more formally, {\em homotopy descriptor}), $|\bar x_i|=\dim\Gamma_i$, $|\bar y_i|=\dim_{\param}\Gamma_i$, $i=1,2$, such that for any tuples of parameters $\bar p_1\in \phi_1(\MB)$, $\bar p_2\in \phi_2(\MB)$ there exists a tuple $\bar r\in\delta(\bar p_1,\bar p_2,\MB)$ (i.\,e., $\MB \models \delta(\bar p_1,\bar p_2,\bar r)$), moreover, for any such $\bar r$ the formula $\theta(\bar x_1,\bar x_2,\bar p_1,\bar p_2,\bar r)$ defines an 
$L(\MA)$\=/isomorphism $\lambda_{\bar p_1,\bar p_2,\bar r}\colon\Gamma_1(\MB,\bar p_1)\to\Gamma_2(\MB,\bar p_2)$. In this case, we call the pair $(\theta,\delta)$ {\em regular homotopy} between regular interpretations $(\Gamma_1,\phi_1)$ and $(\Gamma_2,\phi_2)$ and write $(\theta,\delta)\colon (\Gamma_1,\phi_1)\to(\Gamma_2,\phi_2)$, or $(\theta,\delta)\colon \Gamma_1(\MB,\phi_1)\to\Gamma_2(\MB,\phi_2)$. 
\end{definition}

When discussing a regular homotopy between regular interpretations, the prefix ``regular'' is often omitted when the context makes it clear. 

There are several particular but important cases of Definition~\ref{def:reg_hom} that we would like to discuss separately. 

\begin{definition}
[particular cases of Definition~\ref{def:reg_hom}]\label{def:reg_hom_without} 
{\quad}
\begin{enumerate}[label=\arabic*)]
\item {\bf Regular homotopy between absolute interpretations.} If one (or both) of $\Gamma_1$ and $\Gamma_2$ is absolute, then we define a regular homotopy between them by replacing each absolute interpretation $\Gamma_i$ by its simple parameter extension $\Gamma_i^\reg$ (see Subsection~\ref{subsec:reg_int}) and then proceeding as in Definition~\ref{def:reg_hom}. It is easy to see that the definition above is equivalent to the following: if $\Gamma_1$ and $\Gamma_2$ are absolute then they are regularly homotopic if and only if there are  $L(\MB)$\=/formulas $\theta(\bar x_1,\bar x_2, \bar w)$ and $\delta(\bar w)$, $|\bar x_i|=\dim\Gamma_i$,  $i=1,2$, such that $\delta(\MB)\ne \emptyset$ and for any tuple $\bar r\in\delta(\MB)$ the formula $\theta(\bar x_1,\bar x_2,\bar r)$ defines an isomorphism $\lambda_{\bar r}\colon\Gamma_1(\MB)\to\Gamma_2(\MB)$.
\item {\bf Regular homotopy without a descriptor.} We say that  $(\theta,\emptyset)\colon (\Gamma_1,\phi_1)\to(\Gamma_2,\phi_2)$ is a regular homotopy  {\em without a descriptor} or {\em descriptor-free}, if there is an $L(\MB)$\=/formula $\theta(\bar x_1,\bar x_2, \bar y_1,\bar y_2)$, $|\bar x_i|=\dim\Gamma_i$, $|\bar y_i|=\dim_{\param}\Gamma_i$, $i=1,2$, such that for any tuples of parameters $\bar p_1\in \phi_1(\MB)$, $\bar p_2\in \phi_2(\MB)$ the formula $\theta(\bar x_1,\bar x_2,\bar p_1,\bar p_2)$ defines an 
$L(\MA)$\=/isomorphism $\lambda_{\bar p_1,\bar p_2}\colon\Gamma_1(\MB,\bar p_1)\to\Gamma_2(\MB,\bar p_2)$. In this case, we also term the connector $\theta$ {\em descriptor-free}.
\item   {\bf Regular homotopy  with a removable descriptor.} We say that $(\Gamma_1,\phi_1)$ and $(\Gamma_2,\phi_2)$ are regularly homotopic {\em with a removable descriptor}, if they are regularly homotopic with a connector $\theta$ and a descriptor $\delta$, such that for any tuples $\bar p_1\in \phi_1(\MB)$, $\bar p_2\in \phi_2(\MB)$ and $\bar r_1,\bar r_2\in \delta(\bar p_1,\bar p_2,\MB)$  the isomorphisms $\lambda_{\bar p_1,\bar p_2,\bar r_1}$ and $\lambda_{\bar p_1,\bar p_2,\bar r_2}$ from Definition~\ref{def:reg_hom} are equal.
\item   {\bf Regular homotopy  with a split descriptor.} We say that a homotopy $(\theta,\delta)\colon (\Gamma_1,\phi_1)\to(\Gamma_2,\phi_2)$ has a split descriptor $\delta$ if $\delta$ does not have variables $\bar y_1$ and $\bar y_2$, i.\,e., $\delta $ is in the form $\delta(\bar w)$. 
\end{enumerate}
\end{definition}

\begin{remark}\label{rigid1}
Let $\MA$ be rigid, i.\,e., $\Aut(\MA) =e$, and $(\theta,\delta)\colon (\Gamma_1,\phi_1)\to(\Gamma_2,\phi_2)$ be a homotopy of regular interpretations $(\Gamma_1,\phi_1),(\Gamma_2,\phi_2)\colon \MA\rightsquigarrow\MB$. Then the descriptor $\delta$ is removable. Indeed, for any fixed tuples $\bar p_1\in \phi_1(\MB)$, $\bar p_2\in \phi_2(\MB)$ and $\bar r_1,\bar r_2\in \delta(\bar p_1,\bar p_2,\MB)$ the  isomorphisms $\lambda_{\bar p_1,\bar p_2,\bar r_1}$ and $\lambda_{\bar p_1,\bar p_2,\bar r_2}$ differ by an automorphism of $\Gamma_2(\MB,\bar p_2)$, which is isomorphic to $\MA$; since the group of automorphisms $\Aut(\MA)$ is trivial, one has that $\lambda_{\bar p_1,\bar p_2,\bar r_1}=\lambda_{\bar p_1,\bar p_2,\bar r_2}$, so descriptor $\delta$ is removable.
\end{remark}

\begin{remark} \label{rem:without-descriptor}
Two regular interpretations $(\Gamma_1,\phi_1)$ and $(\Gamma_2,\phi_2)$ of $\MA$ into $\MB$ are regularly homotopic without a descriptor if and only if they are regularly homotopic with a removable descriptor. Indeed, if $(\theta,\emptyset)\colon (\Gamma_1,\phi_1) \to (\Gamma_2,\phi_2)$ is a regular descriptor-free homotopy, then   $(\theta^{\reg},\reg)\colon (\Gamma_1,\phi_1) \to (\Gamma_2,\phi_2)$, where $\theta^{\reg}(\bar x_1,\bar x_2,\bar y_1,\bar y_2, w)=\theta(\bar x_1,\bar x_2,\bar y_1,\bar y_2)$ and  $\reg(w)=(w=w)$, is  a regular homotopy with a removable descriptor. Conversely, if $(\theta,\delta)\colon (\Gamma_1,\phi_1) \to (\Gamma_2,\phi_2)$ is a regular homotopy with a removable descriptor, then the formula $\theta^\abs(\bar x_1,\bar x_2,\bar y_1,\bar y_2)=\forall\,\bar w\:(\delta(\bar y_1,\bar y_2\bar w)\longrightarrow \theta(\bar x_1,\bar x_2,\bar y_1,\bar y_2,\bar w))$ provides a regular descriptor-free homotopy $(\theta^\abs,\emptyset)\colon (\Gamma_1,\phi_1) \to (\Gamma_2,\phi_2)$. 
\end{remark}

We refer to $(\theta^{\reg},\reg)$ as the {\em simple parameter extension} of $(\theta,\emptyset)$, and $(\theta^\abs,\emptyset)$ as the {\em simple absolutization} of the homotopy $(\theta,\delta)$, aligning with the terminology from Subsection~\ref{subsec:reg_int}.

\begin{remark}\label{rem:descriptor}
Let $(\theta,\delta)\colon (\Gamma_1,\phi_1) \to (\Gamma_2,\phi_2)$  be a regular homotopy of regular interpretations with a split descriptor $\delta(\bar w)$ and $(\Gamma_{i,\param},\phi_i\wedge\delta)$ be a split parameter extension of $(\Gamma_i,\phi_i)$, $i=1,2$ (see Definition~\ref{def:param_ext}). Then $(\theta,\emptyset)$ is a regular descriptor-free homotopy between  $(\Gamma_{1,\param},\phi_1\wedge\delta)$ and $(\Gamma_2,\phi_2)$, as well as between $(\Gamma_1,\phi_1)$ and $(\Gamma_{2,\param},\phi_2\wedge\delta)$. Similarly, if $\delta$ has the form $\delta(\bar y_1,\bar w)$ then $(\theta,\emptyset)$ is a regular descriptor-free homotopy between  $(\Gamma_{1,\param},\phi_1\wedge\delta)$ and $(\Gamma_2,\phi_2)$; if $\delta$ is of the form $\delta(\bar y_2,\bar w)$ then $(\theta,\emptyset)$ is a regular descriptor-free homotopy between    $(\Gamma_1,\phi_1)$ and $(\Gamma_{2,\param},\phi_2\wedge\delta)$. Note that the cases above occur when either  $(\Gamma_1,\phi_1)$ or $(\Gamma_2,\phi_2)$ is absolute.
\end{remark}

Unlike Remark~\ref{rem:without-descriptor}, in Remark~\ref{rem:descriptor}, we alter the interpretation $(\Gamma_i, \phi_i)$ to $(\Gamma_{i, \param}, \phi_i \wedge \delta)$, which can sometimes be inconvenient.

\begin{definition}
We name an algebraic structure $\MA=\langle A; L(\MA)\rangle$ {\em regularly i\=/rigid}, if every regular self-interpretation $\MA\rightsquigarrow\MA$ is  regularly homotopic to $\Id_{L(\MA)}$. 
\end{definition}

It can be shown that the ring of integers $\Z$ is regularly i\=/rigid (see Corollary~\ref{cor:reg_i-rigid}).

Unfortunately, regular homotopy is not an equivalence relation on the set $\Codes_{\reg}(\MA,\MB)$ of regular interpretations from $\MA$ to $\MB$, as it lacks reflexivity. However, it is symmetric and transitive, as demonstrated by the result below.

\begin{lemma}\label{le:codes_reg}
Let $\Codes_\reg(\MA,\MB)\ne \emptyset$, then regular homotopy is a symmetric and transitive relation on the set $\Codes_{\reg}(\MA,\MB)$. 
\end{lemma}

\begin{proof} 
The symmetry of regular homotopy is easy. Suppose that regular interpretations $(\Gamma_1,\phi_1)$ and $(\Gamma_2,\phi_2)$ are regularly homotopic with connector $\theta_{12}(\bar x_1,\bar x_2,\bar y_1,\bar y_2,\bar w)$ and descriptor $\delta_{12}(\bar y_1,\bar y_2,\bar w)$, as well as $(\Gamma_2,\phi_2)$ and $(\Gamma_3,\phi_3)$ are regularly homotopic with connector $\theta_{23}(\bar x_2,\bar x_3,\bar y_2,\bar y_3,\bar v)$ and descriptor $\delta_{23}(\bar y_2,\bar y_3, \bar v)$. Then $(\Gamma_1,\phi_1)$ and $(\Gamma_3,\phi_3)$ are regularly homotopic with connector $\theta_{13}(\bar x_1,\bar x_3,\bar y_1,\bar y_3,\bar w,\bar v,\bar y_2)=\exists \,\bar x_2\:(\theta_{12}(\bar x_1,\bar x_2,\bar y_1,\bar y_2,\bar w)\wedge \theta_{23}(\bar x_2,\bar x_3,\bar y_2,\bar y_3,\bar v)\,)$  and descriptor $\delta_{13}(\bar y_1,\bar y_3,\bar w,\bar v,\bar y_2)=\delta_{12}(\bar y_1,\bar y_2,\bar w)\wedge\delta_{23}(\bar y_2,\bar y_3,\bar v)\wedge\phi_2(\bar y_2)$.  
\end{proof}

So, the following definitions arise naturally. 

\begin{definition}
A regular interpretation $(\Gamma,\phi)$ of $\MA$ into $\MB$  is
\begin{enumerate}[label=\arabic*)]
    \item {\em self-homotopic}, if there is a regular homotopy $(\theta,\delta)\colon (\Gamma,\phi)\to (\Gamma,\phi)$, called a {\em self-homotopy};  
    \item {\em isolated}, if there is no regular interpretation $(\Gamma_0,\phi_0)$ of $\MA$ into $\MB$, which is regularly homotopic to $(\Gamma,\phi)$.
\end{enumerate} 
\end{definition}

\begin{remark}\label{rem:abs_self-hom1}
  It is obvious that any absolute interpretation $\MA\simeq\Delta(\MB)$ is self-homotopic with connector $\theta(\bar x_1,\bar x_2)=U_\Delta(\bar x_1)\wedge U_\Delta(\bar x_2)\wedge E_\Delta(\bar x_1,\bar x_2)$ and without a descriptor.  
\end{remark}

Denote  by $\Codes_{\self}(\MA,\MB)$ the subset of all regular self-homotopic interpretations from $\Codes_{\reg}(\MA,\MB)$.

Lemma~\ref{le:codes_reg} has the following corollaries.

\begin{corollary}\label{cor:self_iso}
The following is true:
\begin{itemize}
    \item Every regular interpretation is either self-homotopic or isolated.
    \item Regular homotopy is an equivalence relation on $\Codes_{\self}(\MA,\MB)$. 
\end{itemize}
\end{corollary}

Most natural examples of regular interpretations from algebra are self-homotopic. Below, we provide examples of regular isolated interpretations.

\begin{example}\label{ex:G2}
The regular interpretation $(\Gamma,\phi)$ of $\MA=\langle \Z;\emptyset\rangle$ into $\MB=\langle \MR;R\rangle$ from Example~\ref{ex:G1} is isolated. Indeed, if $\theta(x_1,x_2,z_1,\ldots,z_s)$ is a formula in $L(\MB)$, then for any tuple $(b_1,b_2,c_1,\ldots,c_s)\in \theta(\MB)$, if $b_1\not\in\{b_2,c_1,\ldots,c_s\}$, then the set $\{(b,b_2,c_1,\ldots,c_s)\in \Orb(b_1,b_2,c_1,\ldots,c_s), b\in \MB\}$ is infinite and, therefore,  
the set $\{b\in \MR\mid \MB\models\theta(b,b_2,c_1,\ldots,c_s)\}$ is not a singleton; it is infinite too. Anyway, $\theta$ cannot be a connector between $\Gamma(\MB,p_1)$ and $\Gamma(\MB,p_2)$ when parameters $p_1,p_2\in \MB$ are unequal.
\end{example}

\begin{example}[P.\,Gvozdevsky]\label{ex:G3}
Let $\MA = \langle \Q;< \rangle$ be a set of rational numbers $\Q$ with natural ordering $<$. Then there exists an isolated regular injective self-interpretation $(\Gamma,\phi)\colon \MA\rightsquigarrow\MA$. Here $\dim \Gamma=1$, 
$\dim_\param\Gamma=2$, and $U_\Gamma(x,y_1,y_2)=(y_1<x<y_2)$, $<_\Gamma(x_1,x_2)=(x_1<x_2)$, and $\phi(y_1,y_2)=(y_1<y_2)$. Since $\Th(\MA)$ has quantifier elimination, every formula $\theta(x_1,\ldots,x_m)$ in the language $\{<\}$ is equivalent to a disjunction of formulas of the type $\bigwedge_{i=1}^m(z_i=z_i^\prime)\wedge\bigwedge_{j=1}^l(v_j<v^\prime_j)$, where $z_i,z_i^\prime,v_j,v^\prime_j$ are in $\{x_1,\ldots,x_m\}$. Therefore, there is no formula $\theta$ in $\{<\}$ with parameters in $\Q$, which gives an isomorphism between $\langle (p_1,p_2),<\rangle$ and $\langle (q_1,q_2),<\rangle$, $p_1,p_2,q_1,q_2\in\Q$, if $(p_1,p_2)\ne(q_1,q_2)$.
\end{example}

Note that isolated regular interpretations $(\Gamma,\phi)\colon \MA\rightsquigarrow\MB$ in Examples~\ref{ex:G2} and~\ref{ex:G3} have the following additional property: for any unequal tuples $\bar p_1,\bar p_2\in \phi(\MB)$ interpretations $(\Gamma,\bar p_1)$ and $(\Gamma,\bar p_2)$ are not homotopic with parameters. So we write a new problem.

\begin{problem}
   Find an example of an isolated regular interpretation $(\Gamma,\phi)\colon \MA\rightsquigarrow\MB$, such that for any tuples $\bar p_1,\bar p_2\in \phi(\MB)$ interpretations $(\Gamma,\bar p_1)$ and $(\Gamma,\bar p_2)$ are homotopic with parameters.
\end{problem}

\begin{problem}
    Find an example of algebraic structures $\MA$ and $\MB$, such that
    there exists a regular interpretation $\MA\rightsquigarrow\MB$, but all such interpretations are isolated.
\end{problem}

The following facts are almost obvious, but we want to mention them.

\begin{lemma}
Let $(\Gamma,\phi)\colon\MA\rightsquigarrow\MB$ be a regular interpretation and $(\Gamma_\param,\phi\wedge\delta)$ be a parameter extension of $(\Gamma,\phi)$. Then the following holds:
\begin{enumerate}[label=(\arabic*)]
\item\label{item1:par_ext} if $(\Gamma,\phi)$ is self-homotopic, then $(\Gamma_\param,\phi\wedge\delta)$ is self-homotopic and it is regularly homotopic to $(\Gamma,\phi)$;
\item\label{item2:par_ext} if $(\Gamma_\param,\phi\wedge\delta)$ is self-homotopic and it is a split parameter extension of $(\Gamma,\phi)$, then $(\Gamma,\phi)$ is self-homotopic;
\item\label{item3:par_ext} if  $(\Gamma_\param,\phi\wedge\delta)$ is self-homotopic, then $(\Gamma,\phi^\prime)$ is self-homotopic, where $\phi^\prime(\bar y)=\exists\,\bar w\:(\phi(\bar y)\wedge\delta(\bar y,\bar w))$.
\end{enumerate}
\end{lemma}

There is more to say about regular self-homotopic interpretations; we leave it to subsequent papers.

\section{Invertibility and bi-interpretability}

If a structure $\MA$ is interpreted in a structure $\MB$, then $\MA$ inherits some properties from $\MB$. If they are interpreted in each other ({\em mutually interpreted}), then they have many properties in common, in particular, their elementary theories are reducible to each other via the translation maps (see Section~\ref{subsec:undecidability}), provided these interpretations are absolute or regular. Thus, mutual interpretability is a useful equivalence relation on structures. 
However, we can't say that mutually interpretable structures are ``really the same'' from a model theory viewpoint. Fortunately, there is a natural relation between algebraic structures, called bi-interpretability,  which is much stronger than mutual interpretability and, in the big picture, can be viewed as really similar in the model-theoretic framework.

Invertible and bi-interpretations are constructed based on the successive application of interpretations and their compositions. By ``invertible interpretation'', we refer to an interpretation $\MA\stackrel{\Gamma}{\rightsquigarrow}\MB$ that has an inverse interpretation $\MB\stackrel{\Delta}{\rightsquigarrow}\MA$, such that the composition of interpretations $\Gamma\circ\Delta$ is homotopic to the identity interpretation. More precisely, we obtain two concepts of invertibility (right and left) depending on which composition the additional conditions are imposed on: $\Gamma\circ\Delta$ or $\Delta\circ\Gamma$. If the conditions are fulfilled simultaneously for different inverses, $\Delta_1$ and $\Delta_2$, this is known as two-sided invertibility. If they are fulfilled for the same inverse, $\Delta$, then it is called bi-interpretability. These definitions typically have three levels: absolute, regular, and involving parameters.
 
Bi-interpretability, particularly regular bi-interpretability, poses certain challenges. Invertible interpretations are simpler to handle and can sometimes suffice for the problem at hand. In a forthcoming paper in our series, we'll show that a regular right-invertible interpretation of $\MA$ in $\MB$ suffices to algebraically characterize all structures $\nsA$ with $\nsA\equiv \MA$, based on the models of the theory $\Th(\MB)$. However, regular bi-interpretability helps in offering a clearer understanding of the overall picture.

We begin with compositions of interpretations.

\subsection{Compositions of interpretations}

It is known that the relation $\MA\rightsquigarrow\MB$ is transitive on algebraic structures (see, for example,~\cite{Hodges}). The proof of this fact is based on the notion of $\Gamma$\=/translation and composition of codes. 

\begin{definition} \label{def:code-composition}
Let $\Gamma\colon L(\MA)\to L(\MB)$ and $\Delta\colon L(\MB)\to L(\MC)$ be codes (either with parameters or absolute), $\Gamma=\{U_\Gamma,E_\Gamma,Q_\Gamma\mid Q\in L(\MA)\}$. Then the {\em composition} $\Gamma\circ\Delta\colon L(\MA)\to L(\MC)$ of the codes $\Gamma$ and $\Delta$ is the following code:
$$
\Gamma\circ \Delta = \{U_{\Gamma\circ\Delta}, E_{\Gamma\circ\Delta}, Q_{\Gamma\circ\Delta} \mid Q \in L(\MA)\}=\{(U_\Gamma)_\Delta, (E_\Gamma)_\Delta, (Q_\Gamma)_\Delta \mid Q \in L(\MA)\},
$$
where $\dim \Gamma\circ \Delta =\dim \Gamma\cdot\dim \Delta$ and $\dim_{\param}\Gamma\circ\Delta = \dim_{\param}\Gamma\cdot \dim\Delta+\dim_{\param}\Delta$.
\end{definition}

\begin{lemma}
Let $\Gamma\colon L(\MA)\to L(\MB)$ and $\Delta\colon L(\MB)\to L(\MC)$ be codes. The following conditions hold:
\begin{enumerate}[label=\arabic*)]
\item\label{le:int0}$\Id_{L(\MA)}\circ\Gamma\approx\Gamma\circ\Id_{L(\MB)}\approx\Gamma$, where $\approx$ is equivalence of codes (see Definition~\ref{def:equiv_codes});
\item If $\Gamma,\Delta$ are absolute, then $\Gamma\circ\Delta$ is absolute too; 
\item If $\Gamma,\Delta$ are injective, then $\Gamma\circ\Delta$ is injective too;
\item\label{le:int8} If $\Gamma,\Delta$ are Diophantine, then $\Gamma\circ\Delta$ is Diophantine too. 
\end{enumerate}
\end{lemma}

\begin{proof}
    For~\ref{le:int0} see Remarks~\ref{rem:Id1} and~\ref{rem:Id2}; for~\ref{le:int8} see Remark~\ref{rem:Diophantine}.
\end{proof}

The following is an important technical result on the transitivity of interpretations.

\begin{lemma}\label{le:int-transitivity}
Let $\MA=\langle A;L(\MA)\rangle$, $\MB=\langle B;L(\MB)\rangle$ and $\MC=\langle C;L(\MC)\rangle$ be algebraic structures and $\Gamma\colon L(\MA)\to L(\MB)$, $\Delta\colon L(\MB)\to L(\MC)$ be codes. Then the following holds:
\begin{enumerate}[label=(\arabic*)]
\item If $\MA\stackrel{\Gamma}{\rightsquigarrow}\MB$ and 
$\MB\stackrel{\Delta}{\rightsquigarrow}\MC$ then $\MA\stackrel{\Gamma\circ\Delta}{\rightsquigarrow}\MC$;
\item\label{item:reg} If $(\Gamma,\phi)\colon\MA \rightsquigarrow \MB$ and $(\Delta,\psi)\colon\MB \to \MC$ are regular interpretations then $(\Gamma \circ \Delta, \phi_\Delta\wedge\psi)\colon\MA \rightsquigarrow \MC$ is also a regular interpretation.
\end{enumerate}

Furthermore, the following conditions hold:
\begin{enumerate}[label=\roman*)]
    \item If $\bar p,\bar q$ are parameters and $\mu_\Gamma, \mu_\Delta$ are  coordinate maps of interpretations $\Gamma,\Delta$ then $(\bar{\bar p},\bar q)$, where $\bar{\bar p} \in \mu_\Delta^{-1}(\bar p)$, are parameters for $\Gamma\circ\Delta$; 
    \item\label{le:int2} For any coordinate map $\mu_\Delta$ of the interpretation $\MB\simeq\Delta(\MC,\bar q)$ and any tuple $\bar{\bar p} \in \mu_\Delta^{-1}(\bar p)$ the $L(\MA)$\=/structure $\Gamma\circ\Delta(\MC,\bar{\bar p},\bar q)$ is well-defined and isomorphic to $\MA$;
    \item The $L(\MA)$\=/structure $\Gamma\circ\Delta (\MC, \bar{\bar p},\bar q)$ does not depend on the choice of $\bar{\bar p} \in \mu_\Delta^{-1}(\bar p)$, when $\mu_\Delta$ is fixed; 
    \item $\mu_\Gamma\circ\mu_\Delta=\mu_{\Gamma}\circ\mu^n_{\Delta}\big|_{U_{\Gamma\circ\Delta}(\MC, \bar{\bar p},\bar q)}$ is a coordinate map of the interpretation $\MA\simeq {\Gamma\circ\Delta}(\MC,\bar{\bar p},\bar q)$ and any coordinate map $\mu_{\Gamma\circ\Delta}\colon U_{\Gamma\circ\Delta}(\MC, \bar{\bar p},\bar q)\to A$ has a form $\mu_{\Gamma1}\circ\mu_\Delta$ for a suitable coordinate map $\mu_{\Gamma1}$ of the interpretation $\MA\simeq\Gamma(\MB,\bar p)$, provided $\mu_\Delta$ is fixed.    
\end{enumerate}
\end{lemma}

\begin{proof}
Let $n=\dim\Gamma$ and $m=\dim\Delta$. By the Reduction Theorem~\ref{le:interpr_corol} for any tuples $\bar b_i$ from $B^n$ and any preimages $\bar{\bar c}_i\in\mu_{\Delta}^{-1}(\bar b_i)$, $\bar{\bar p}\in\mu_{\Delta}^{-1}(\bar p)$ one has 
\begin{equation}\label{eq:U}
\MB\models U_\Gamma (\bar b_0,\bar p) \iff \MC\models U_{\Gamma\circ\Delta}(\bar{\bar c}_0,\bar{\bar p},\bar q);
\end{equation}
\begin{equation}\label{eq:E}
\MB\models E_\Gamma (\bar b_1,\bar b_2,\bar p) \iff \MC\models E_{\Gamma\circ\Delta}(\bar{\bar c}_1,\bar{\bar c}_2, \bar{\bar p},\bar q);
\end{equation}
\begin{equation}\label{eq:Q}
\MB\models Q_\Gamma (\bar b_1,\ldots, \bar b_{t_Q},\bar p) \iff \MC\models Q_{\Gamma\circ\Delta}(\bar{\bar c}_1, \ldots,\bar{\bar c}_{t_Q}, \bar{\bar p},\bar q)
\end{equation}
for every $Q\in L(\MA)$. Here if $\bar{\bar c}_i\in C^{nm}$ and the right parts of~\eqref{eq:U}, \eqref{eq:E}, \eqref{eq:Q} are true, then $\bar{\bar c}_i \in \mu_{\Delta}^{-1}(B^n)$ and the left parts are true as well.

Furthermore, since $\MB\models\AC_\Gamma(\bar p)$, then, by Lemma~\ref{psi+}, one has $\MC\models\AC_{\Gamma\circ\Delta}(\bar{\bar p},\bar q)$. Thus $L(\MA)$\=/structure $\Gamma\circ \Delta(\MC,\bar{\bar p},\bar q)$ is well-defined and it does not depend on the choice of $\bar{\bar p}\in \mu_{\Delta}^{-1}(\bar p)$ provided $\mu_\Delta$ is fixed. Now we should  
check that $\MA\simeq\Gamma\circ \Delta(\MC,\bar{\bar p},\bar q)$.

Remind that $A_{\Gamma\circ\Delta}=U_{\Gamma\circ\Delta}(\MC,\bar{\bar p},\bar q)$. Consider the map $\mu_{\Gamma\circ\Delta}\colon A_{\Gamma\circ\Delta}\to A$, $\mu_{\Gamma\circ\Delta}=\mu_{\Gamma}\circ\mu_{\Delta}$ (more precisely, $\mu_{\Gamma\circ\Delta}=\mu_{\Gamma}\circ\mu^n_{\Delta}\big|_{A_{\Gamma\circ\Delta}}$). As $\mu_{\Gamma}$ is surjective and~\eqref{eq:U} holds, $\mu_{\Gamma\circ \Delta}$ is well-defined and also surjective. Because $\bar\mu_{\Gamma}\colon U_\Gamma(\MB,\bar p)/{\sim_\Gamma}\, \to A$ is bijective and~\eqref{eq:E} holds, $\bar\mu_{\Gamma\circ\Delta}\colon A_{\Gamma\circ\Delta}/{\sim_{\Gamma\circ\Delta}} \,\to A$ is well-defined and bijective too. Finally, since $\bar\mu_{\Gamma}$ is an $L(\MA)$\=/isomorphism and due to~\eqref{eq:Q} $\bar\mu_{\Gamma\circ\Delta}\colon \Gamma\circ \Delta(\MC,\bar{\bar p},\bar q) \to \MA$ is an $L(\MA)$\=/isomorphism as well. Thereby, $\MA$ is interpretable in $\MC$ with the code $\Gamma\circ \Delta$ and parameters $(\bar{\bar p},\bar q)$. 

To prove~\ref{item:reg}, suppose that $(\Gamma,\phi), (\Delta,\psi)$ are regular interpretations. Then any tuple $\bar r\in \phi_\Delta\wedge\psi(\MC)$ has a form $\bar r=(\bar{\bar p},\bar q)$, where $\bar q\in \psi(\MC)$, $\bar{\bar p}\in \phi_\Delta(\MC,\bar q)$. As $\bar q\in \psi(\MC)$ there exists an interpretation $\MB\simeq\Delta(\MC,\bar q)$. 
For any coordinate map $\mu_\Delta$ of the interpretation $\MB\simeq\Delta(\MC,\bar q)$ the tuple $\bar p=\mu_\Delta(\bar{\bar p})$ is in $\phi(\MB)$, by the Reduction Theorem~\ref{le:interpr_corol}. Thus there exists an interpretation $\MA\simeq\Gamma(\MB,\bar p)$. So the $L(\MA)$\=/structure $\Gamma\circ\Delta(\MC,\bar r)=\Gamma\circ\Delta(\MC,\bar {\bar p},\bar q)$ is well-defined and isomorphic to $\MA$, due to item~\ref{le:int2}. 
\end{proof}

\begin{remark}\label{rem:composition}
Note that by construction, if $\Gamma$ has parameters $\bar p$, the location of ${\Gamma\circ\Delta} (\MC, \bar{\bar p},\bar q)$ in $\MC$,  i.\,e., the set  $U_{\Gamma\circ\Delta}(\MC,\bar{\bar p},\bar q)\subseteq C^{\dim \Gamma\circ\Delta}$, depends on the choice of $\mu_\Delta$, as well as the relation $\sim_{\Gamma\circ\Delta}$,  the constants $c_{\Gamma\circ\Delta}$, the functions $f_{\Gamma\circ\Delta}$ and the predicates $R_{\Gamma\circ\Delta}$. On the other hand, the structure $\Gamma\circ\Delta (\MC, \bar{\bar p},\bar q)$ does not depend on the choice of $\mu_\Gamma$. So the algebraic structure $\Gamma\circ\Delta (\MC,  \bar{\bar p},\bar q)$ is uniquely defined if a coordinate map $\mu_\Delta$ is fixed. And if $\Gamma$ is an absolute code then the algebraic structure $\Gamma\circ\Delta (\MC, \bar{\bar p},\bar q)$ does not depend on $\mu_\Gamma$ and $\mu_\Delta$. 
\end{remark}

Here a natural question arises: will the interpretations $\Gamma\circ\Delta (\MC, \bar{\bar p}_1,\bar q)$ and $\Gamma\circ\Delta (\MC, \bar{\bar p}_2,\bar q)$ be homotopic, if $\bar{\bar p}_1\in \mu_{\Delta 1}^{-1}(\bar p)$, $\bar{\bar p}_2\in \mu_{\Delta 2}^{-1}$, and $\mu_{\Delta 1}, \mu_{\Delta 2}$ are coordinate maps of the interpretation $\MB\stackrel{\Delta,\bar q}{\rightsquigarrow}\MC$? The negative answer is given by the example below.

\begin{example}\label{ex:H}
Suppose that $A=B=\Q$, $L(\MA)=\{<,0\}$ and $L(\MB)=\{<\}$. So that $\MA$ is the rational numbers with linear order and constant $0$, while $\MB$ is the rational numbers with linear order only. We define codes of interpretations $\MA\stackrel{\Gamma}{\rightsquigarrow}\MB$ and 
$\MB\stackrel{\Delta}{\rightsquigarrow}\MA$ by $\Gamma=\{U_\Gamma(x,y)=(x=x), E_\Gamma(x_1,x_2,y)=(x_1=x_2),<_\Gamma(x_1,x_2,y)=(x_1<x_2), 0_\Gamma(x,y)=(x=y)\}$ and  
$\Delta=\{U_\Delta(x)=(x=x), E_\Delta(x_1,x_2)=(x_1=x_2),<_\Delta(x_1,x_2)=(x_1<x_2)\}$. Here, the interpretation $\Gamma$ has $0\in \MB$ as a parameter and $\Delta$ is absolute. We fix two coordinate maps $\mu_{\Delta 1}, \mu_{\Delta 2}$ of the interpretation $\MB\stackrel{\Delta}{\rightsquigarrow}\MA$, where $\mu_{\Delta 1}\colon x \to x$ and $\mu_{\Delta 2}\colon x \to x-1$. Then the interpretations $\Gamma\circ\Delta (\MA, 0)$ and $\Gamma\circ\Delta (\MA, 1)$ are not homotopic, since there is no $L(\MA)$\=/definable order preserving bijection $f\colon \Q\to \Q$ with $f(0)=1$. Indeed, assume the opposite: there exists such bijection $f$ and an $L(\MA)$\=/formula $\theta(x_1,x_2,c_1,\ldots,c_s)$ that defines the graph of $f$, where $c_1,\ldots,c_s\in\Q$. Then $f(0)> 0$, $f(f(0))>f(0)$ and so on, that is, there exists an infinite set $X\subseteq \Q$ such that $f(x)> x$ for all $x\in X$. Since order on $\Q$ is linear, dense and without endpoints, then formula $\theta$ is an finite disjunction of finite conjunctions of the types $x_i=x_j$, $x_i=c_j$, $x_i<x_j$, $x_i<c_j$ or $x_i>c_j$ (see~\cite[Theorem~3.1.3]{Marker}). Take any element $x\in X$ such that $x>c_i$ for all $i=1,\ldots,s$. One has $\MA\models \theta(x,f(x),c_1,\ldots,c_s)$, but at the same time, $\MA\models \theta(x,f(x)+1,c_1,\ldots,c_s)$, therefore, $\theta$ does not define the graph of the function $f$.
\end{example}

\begin{remark}\label{remark2}
Let $(\Gamma,\phi), (\Delta,\psi)$ be regular interpretations as in item~\ref{item:reg} of Lemma~\ref{le:int-transitivity}. As is was mentioned, for any tuple $(\bar {\bar p},\bar q)\in\phi_\Delta\wedge\psi(\MC)$ and any coordinate map $\mu_\Delta$ of the interpretation $\MB\simeq\Delta(\MC,\bar q)$ the image $\mu_\Delta({\bar{\bar p}})$ is in $\phi(\MB)$. And conversely, for any tuples of parameters $\bar p\in \phi(\MB)$ and $\bar q\in \psi(\MC)$ and any coordinate map $\mu_\Delta$ of the interpretation $\MB\simeq\Delta(\MC,\bar q)$, and any preimages $\bar{\bar p} \in \mu_\Delta^{-1}(\bar p)$ the tuple $(\bar {\bar p},\bar q)$ is in $\phi_\Delta\wedge\psi(\MC)$.  
\end{remark}

The following lemma says how the serial translation of formulas works if one has a series of interpretations $\MA\stackrel{\Gamma}{\rightsquigarrow}\MB$ and 
$\MB\stackrel{\Delta}{\rightsquigarrow}\MC$.

\begin{lemma}\label{ABC}
Let $(\Gamma,\bar p, \mu_\Gamma)$ be an interpretation of $\MA=\langle A;L(\MA)\rangle$ in $\MB$, and $(\Delta,\bar q, \mu_\Delta)$ be an interpretation of $\MB$ in $\MC$. Then for any formula $\varphi(x_1,\ldots,x_m)$ in $L(\MA)\cup A$ one has
$$
\varphi_{\Gamma\circ\Delta,\,(\bar{\bar p},\bar q),\,\mu_{\Gamma\circ\Delta}}(\MC)= (\varphi_{\Gamma,\bar p, \mu_\Gamma})_{\Delta,\bar q, \mu_\Delta}(\MC),
$$
where $\mu_{\Gamma\circ\Delta}=\mu_\Gamma\circ\mu_\Delta$, $\bar{\bar p}\in \mu_\Delta^{-1}(\bar p)$.
\end{lemma}

\begin{proof}
We use the Reduction Theorem~\ref{le:interpr_corol} all the time. Take any elements $\bar{\bar c}_1,\ldots,\bar{\bar c}_m$ in $\MC$. If $\MC\models \varphi_{\Gamma\circ\Delta,(\bar{\bar p},\bar q),\mu_{\Gamma\circ\Delta}}(\bar{\bar c}_1,\ldots,\bar{\bar c}_m)$, then $\bar{\bar c}_1,\ldots,\bar{\bar c}_m\in U_{\Gamma\circ\Delta}(\MC,\bar{\bar p},\bar q)$. Let $\bar b_i=\mu_\Delta(\bar{\bar c}_i)$ and $a_i=\mu_\Gamma(\bar b_i)$, i.\,e., $a_i=\mu_{\Gamma\circ\Delta}(\bar{\bar c}_i)$, $i=1,\ldots,m$. 
Since $\MA\stackrel{\Gamma\circ\Delta}{\rightsquigarrow}\MC$ one has $\MA\models\varphi(a_1,\ldots,a_m)$. Therefore, $\MA\stackrel{\Gamma}{\rightsquigarrow}\MB$ gives $\MB\models\varphi_{\Gamma,\bar p, \mu_\Gamma}(\bar b_1,\ldots,\bar b_m)$, and $\MB\stackrel{\Delta}{\rightsquigarrow}\MC$ gives $\MC\models(\varphi_{\Gamma,\bar p, \mu_\Gamma})_{\Delta,\bar q, \mu_\Delta}(\bar{\bar c}_1,\ldots,\bar{\bar c}_m)$.

Conversely, if $\MC\models(\varphi_{\Gamma,\bar p, \mu_\Gamma})_{\Delta,\bar q, \mu_\Delta}(\bar{\bar c}_1,\ldots,\bar{\bar c}_m)$, then $\bar{\bar c}_1,\ldots,\bar{\bar c}_m$ are in $U_{\Delta}(\MC,\bar q)$ and $\MB\models\varphi_{\Gamma,\bar p, \mu_\Gamma}(\bar b_1,\ldots,\bar b_m)$ for $\bar b_i=\mu_\Delta(\bar{\bar c}_i)$,  $i=1,\ldots,m$. Therefore $\bar b_1,\ldots,\bar b_m\in U_\Gamma(\MB,\bar p)$ and $\MA\models\varphi(a_1,\ldots,a_m)$ for $a_i=\mu_\Gamma(\bar b_i)$, $i=1,\ldots,m$. And since $\bar{\bar c}_i\in \mu^{-1}_{\Gamma\circ\Delta}(a_i)$, $i=1,\ldots,m$, one has $\MC\models \varphi_{\Gamma\circ\Delta,(\bar{\bar p},\bar q),\mu_{\Gamma\circ\Delta}}(\bar{\bar c}_1,\ldots,\bar{\bar c}_m)$.
\end{proof}

While proving Lemma~\ref{le:int-transitivity}, we use Lemma~\ref{psi+} just in one direction. If we read its result in another direction, we get the following result.

\begin{lemma}\label{le:Aut3}
Let $\Gamma\colon L(\MA)\to L(\MB)$ and $\Delta\colon L(\MB)\to L(\MC)$ be codes as before, $\MC$ be an algebraic structure and $\bar q,\bar{\bar p}$ be tuples from $\MC$, $|\bar q|=\dim_{\param}\Delta$, $|\bar{\bar p}|=\dim\Delta\cdot\dim_{\param}\Gamma$, such that $\MC\models\AC_\Delta(\bar q)$ and $\MC\models\AC_{\Gamma\circ\Delta}(\bar{\bar p},\bar q)$, i.\,e., the $L(\MB)$\=/structure $\Delta(\MC,\bar q)$ and the $L(\MA)$\=/structure $\Gamma\circ\Delta(\MC,\bar{\bar p},\bar q)$ are well-defined. Then $\Delta(\MC,\bar q)\models\AC_\Gamma(\bar{\bar p}/{\sim_\Delta})$, so that the $L(\MA)$\=/structure $\Gamma(\Delta(\MC,\bar q),\bar{\bar p}/{\sim_\Delta})$ is well-defined and isomorphic to $\Gamma\circ\Delta(\MC,\bar{\bar p},\bar q)$ by means of the isomorphism $\bar\mu_{\Gamma\circ\Delta}\colon {\Gamma\circ\Delta}(\MC,\bar{\bar p},\bar q)\to \Gamma(\Delta(\MC,\bar q),\bar{\bar p}/{\sim_\Delta})$, such that
$$
\bar\mu_{\Gamma\circ\Delta}\colon \bar{\bar c}/{\sim_{\Gamma\circ\Delta}}\;\to\;(\bar{\bar c}/{\sim_\Delta})/{\sim_\Gamma}, \quad \bar{\bar c}\in U_{\Gamma\circ\Delta}(\MC,\bar{\bar p},\bar q).
$$
\end{lemma}

\begin{proof}
The result of Lemma~\ref{psi+}, applied to the interpretation $\Delta(\MC,\bar q)\stackrel{\Delta}{\rightsquigarrow}\MC$, gives that $\MC\models\AC_{\Gamma\circ\Delta}(\bar{\bar p},\bar q)$ entails that $\bar{\bar p}$ is in $U_\Delta(\MC,\bar q)$ and $\Delta(\MC,\bar q)\models\AC_\Gamma(\bar{\bar p}/{\sim_\Delta})$. Therefore, the $L(\MA)$\=/structure $\Gamma(\Delta(\MC,\bar q),\bar{\bar p}/{\sim_\Delta})$ is well-defined and one has the interpretation $\Gamma(\Delta(\MC,\bar q),\bar{\bar p}/{\sim_\Delta})\stackrel{\Gamma}{\rightsquigarrow}\Delta(\MC,\bar q)$. And by Lemma~\ref{le:int-transitivity}, there exists an interpretation $\Gamma(\Delta(\MC,\bar q),\bar{\bar p}/{\sim_\Delta})\stackrel{\Gamma\circ\Delta}{\rightsquigarrow}\MC$ and $L(\MA)$\=/structures $\Gamma(\Delta(\MC,\bar q),\bar{\bar p}/{\sim_\Delta})$ and ${\Gamma\circ\Delta}(\MC,\bar{\bar p},\bar q)$ are isomorphic. Let $\mu_{\Gamma\circ\Delta}$ be the coordinate map of this interpretation, which equals to the composition $\mu_\Gamma\circ\mu_\Delta$ of natural coordinate maps $\mu_\Gamma\colon U_{\Gamma}(\Delta(\MC,\bar q),\bar{\bar p}/{\sim_\Delta})\to U_{\Gamma}(\Delta(\MC,\bar q),\bar{\bar p}/{\sim_\Delta})/{\sim_\Gamma}$ and $\mu_\Delta\colon U_\Delta(\MC,\bar q)\to U_\Delta(\MC,\bar q)/{\sim_\Delta}$. So, $\mu_{\Gamma\circ\Delta}$ sends a tuple $\bar{\bar c}\in U_{\Gamma\circ\Delta}(\MC,\bar{\bar p},\bar q)$ to $(\bar{\bar c}/{\sim_\Delta})/{\sim_\Gamma}$, therefore, the corresponding isomorphism $\bar\mu_{\Gamma\circ\Delta}$ sends $\bar{\bar c}/{\sim_{\Gamma\circ\Delta}}$ to $(\bar{\bar c}/{\sim_\Delta})/{\sim_\Gamma}$, as required.
\end{proof}

\subsection{Invertible interpretations with parameters}\label{subsec:invertible_int}

The following definition uses the notion of composition of interpretations and notations from Lemma~\ref{le:int-transitivity}.

\begin{definition}
We say that an interpretation $(\Gamma,\bar p,\mu_\Gamma)$ of $\MA$ in $\MB$ has a {\em right inverse} if there exists an interpretation $(\Delta,\bar q,\mu_\Delta)$ of $\MB$ in $\MA$ such that the following conditions hold:
\begin{enumerate}[label=\arabic*)]
\item  the algebraic structure $\Gamma\circ\Delta(\MA,\bar{\bar p},\bar q)$ is uniquely defined (see Remark~\ref{rem:composition}) and isomorphic to $\MA$; 
\item the interpretation $(\Gamma\circ\Delta,\bar{\bar p},\bar q)$ is homotopic to the identical interpretation $\Id_{L(\MA)}$, i.\,e., there exists a definable coordinate map $\mu_\MA\colon U_{\Gamma \circ \Delta}(\MA,\bar{\bar p},\bar q) \to A$. 
\end{enumerate}
In this case we refer to $(\Delta,\bar q,\mu_\Delta)$ as a {\em right inverse for $(\Gamma,\bar p,\mu_\Gamma)$} and to $(\Gamma,\bar p,\mu_\Gamma)$ as a {\em left inverse for $(\Delta,\bar q,\mu_\Delta)$}.
\end{definition}

\begin{remark}
Note that if a right inverse $(\Delta, \bar{q}, \mu_\Delta)$ of $(\Gamma, \bar{p}, \mu_\Gamma)$ exists, it is not necessarily unique. Furthermore, if an interpretation $(\Delta^\ast,\bar q^\ast,\mu_{\Delta^\ast})$ of $\MB$ in $\MA$ is strongly homotopic to $(\Delta,\bar q,\mu_\Delta)$, then $(\Delta^\ast,\bar q^\ast,\mu_{\Delta^\ast})$ is also a right inverse of $(\Gamma,\bar p,\mu_\Gamma)$.  We will prove this in one of the subsequent papers of this series. 
\end{remark}

\begin{example}\label{ex:O3}
For a finitely generated ring $\Aint$ of algebraic integers, interpretations $\Gamma\colon\Z\rightsquigarrow\Aint$ and $\Delta\colon\Aint\rightsquigarrow\Z$ from Examples~\ref{ex:O1} and~\ref{ex:O2} are both right and left inverses to each other. Here $\Gamma$ and $\Delta$ are absolute, but the connector of the homotopy between $\Delta\circ\Gamma$ and $\Id_{L_\ring}$ has parameters.
\end{example}

\begin{definition} \label{de:invert}
An interpretation $(\Gamma,\bar p,\mu_\Gamma)$ of $\MA$ in $\MB$ is called {\em right-invertible} (correspondingly, {\em left-invertible}; {\em two-sided invertible}), if it has a right inverse $(\Delta,\bar q,\mu_\Delta)$ (correspondingly, a left inverse; a right inverse $(\Delta,\bar q,\mu_\Delta)$ and a left inverse $(\Delta^\ast,\bar q^\ast,\mu_{\Delta^\ast})$).
If $\Gamma, \Delta$ are absolute and $\mu_\MA\colon U_{\Gamma \circ \Delta}(\MA,\bar{\bar p},\bar q) \to A$  is $0$\=/definable, then we term $\Gamma$ {\em absolutely right-invertible}. We define similarly {\em absolutely left-invertible} and{\em absolutely two-sided invertible} interpretations. 
\end{definition}

Usually, by {\em invertible interpretation} we mean a right-invertible one.

\begin{remark}\label{rem:0}
Suppose that $(\Gamma,\bar p,\mu_\Gamma)$ is a right-invertible interpretation of $\MA$ into $\MB$ with a right inverse $(\Delta,\bar q,\mu_\Delta)$ and definable coordinate map $\mu_\MA\colon U_{\Gamma \circ \Delta}(\MA,\bar{\bar p},\bar q) \to A$. Recall that the location of $\Gamma\circ\Delta(\MA,\bar{\bar p},\bar q)$ in $\MA$ depends on $\mu_\Delta$ and does not depend on $\mu_\Gamma$ (see Remark~\ref{rem:composition}). If we change $\mu_\Delta$, the coordinate map $\mu_\MA$ may no longer be definable in $\MA$. At the same time, we may always replace $\mu_\Gamma$ with $\mu_{\Gamma 0}=\bar\mu_\MA\circ \bar \mu^{-1}_{\Gamma\circ \Delta}\circ \mu_\Gamma$, where $\mu_{\Gamma\circ\Delta}=\mu_\Gamma\circ\mu_\Delta$, and get that $\mu_{\Gamma 0}\circ\mu_\Delta= \mu_\MA$$\mu_\Gamma\circ\mu_\Delta$ is definable in $\MA$. 
\end{remark}

\begin{definition} 
An interpretation $(\Gamma,\bar p,\mu_\Gamma)$ of $\MA$ in $\MB$ is called {\em strongly right-invertible}, if it has a {\em strong right inverse}, i.\,e., a right inverse $(\Delta,\bar q,\mu_\Delta)$, such that the composition $\mu_\Gamma\circ\mu_\Delta$ is definable in $\MA$.
\end{definition}

\begin{remark}\label{rigid3}
So, if $(\Gamma, \bar p, \mu_\Gamma)$ is a right-invertible interpretation, then we can always redefine the coordinate map $\mu_\Gamma$ by a new one $\mu_{\Gamma0}$ and obtain a strongly right-invertible interpretation $(\Gamma, \bar p, \mu_{\Gamma 0})$. And if $\MA$ is rigid, then an interpretation $(\Gamma, \bar p, \mu_\Gamma)$ of $\MA$ in $\MB$ is right-invertible if and only if it is strongly right-invertible.
\end{remark}

Further, one has the following continuation of the results of Lemma~\ref{cor2} and Corollary~\ref{cor2.0}.

\begin{lemma}\label{lem:inv_def}
Suppose there exists a strongly right-invertible interpretation  $(\Gamma,\bar p,\mu_\Gamma)$ of $\MA=\langle A;L(\MA)\rangle$ into $\MB=\langle B;L(\MB)\rangle$. Then a subset $X\subseteq A^m$ is definable in $\MA$ if and only if the set $\mu_\Gamma^{-1}(X)\subseteq B^{m\cdot n}$ is definable in $\MB$, $n=\dim\Gamma$.
\end{lemma}

\begin{proof}
Let $(\Delta,\bar q,\mu_\Delta)$ be a right inverse of $(\Gamma,\bar p,\mu_\Gamma)$, such that the composition $\mu_\Gamma\circ\mu_\Delta$ is definable in $\MA$ by a formula $\theta(\bar{\bar x},x,\bar{\bar p},\bar q,\bar r)$, $\bar r\in \MA$. By Lemma~\ref{cor2}, if $X$ is definable in $\MA$, then $Y=\mu_\Gamma^{-1}(X)$ is definable in $\MB$. And if $Y$ is definable in $\MB$ by a formula $\chi(\bar x_1,\ldots,\bar x_m,\bar b)$ in $L(\MB)\cup B$, $\bar b\in \MB$, then the set $Z=\mu_\Delta^{-1}(Y)$ is definable in $\MA$ by formula $\chi_\Delta(\bar{\bar x}_1,\ldots,\bar{\bar x}_m,\bar{\bar b},\bar q)$, $\bar{\bar b}\in \mu_\Delta^{-1}(\bar b)$. Therefore, $X$ is also definable by the formula 
$$
\pi(x_1,\ldots,x_m,\bar{\bar b},\bar{\bar p},\bar q,\bar r)=\exists \,\bar{\bar x}_1\:\ldots \exists \,\bar{\bar x}_m\:( \chi_\Delta(\bar{\bar x}_1,\ldots,\bar{\bar x}_m,\bar{\bar b},\bar q)\wedge \bigwedge\limits_{i=1}^m \theta(\bar{\bar x}_i,x_i,\bar{\bar p},\bar q,\bar r)\,).
$$
\end{proof}

\begin{corollary}
Suppose there exists a strongly absolutely right-invertible interpretation  $(\Gamma,\emptyset,\mu_\Gamma)$ of $\MA=\langle A;L(\MA)\rangle$ into $\MB=\langle B;L(\MB)\rangle$. Then a subset $X\subseteq A^m$ is $0$\=/definable in $\MA$ if and only if the set $\mu_\Gamma^{-1}(X)\subseteq B^{m\cdot n}$ is $0$\=/definable in $\MB$, $n=\dim\Gamma$.
\end{corollary}

The following fact is a direct continuation of Remarks~\ref{lem:AB1} and~\ref{lem:AB2}.

\begin{remark}\label{lem:AB3}
Let $(\Gamma,\bar p,\mu_{\Gamma})$ be a strongly right-invertible interpretation of $\MA=\langle A;L(\MA)\rangle$ into $\MB=\langle B;L(\MB)\rangle$ with strong right inverse  $(\Delta, \bar q,\mu_{\Delta})$. Then $(\Gamma_{A}, \emptyset,\mu_{\Gamma})$ is an absolutely strongly right-invertible interpretation of $\MA_A$ into $\MB_B$ with strong right inverse  $(\Delta_{B},\emptyset,\mu_{\Delta})$. If the languages $L(\MA)$ and $L(\MB)$ are finite, then the inverse statement also holds. 
\end{remark}

Continuing with Examples~\ref{ex:BS1} and~\ref{ex:BS2} about the metabelian Baumslag\,--\,Solitar group $\BS(1,k)$, $k>1$, we mention the following fact.

\begin{lemma}[\cite{BS, Khelif}]\label{ex:BS5}
The interpretation $\Gamma\colon\Z\rightsquigarrow \BS(1,k)$ with parameter $b=(0,1)$ from Example~\ref{ex:BS2} is a right inverse for the absolute interpretation $\Delta\colon\BS(1,k)\rightsquigarrow\Z$ from Example~\ref{ex:BS1}.   
\end{lemma}

\begin{proof}
Indeed, a homotopy between $(\Delta\circ\Gamma,b)$ and $\Id_{L_\group}$ is given by the formula
$$
\theta_{\BS(1,k)}(z,i,m,\:x, \: a, b)\;=\;([i,b]=[m,b]=e)\:\wedge\:\tau(a,\,i \, x \, m^{-1} i^{-1},\, z,\,b),
$$
where $a=(1,0)$ and $\tau$ is a formula from~\cite[Lemma~6]{BS}. The formula $\theta_{\BS(1,k)}$ defines the composition $\mu_\Delta\circ\mu_\Gamma\colon (b^z,b^i,b^m) \to b^{-i}a^zb^ib^m=(zk^i,m)$, $z,i,m\in \Z$.
\end{proof}

\subsection{Regularly invertible interpretations}\label{subsec:reg_invertible_int}

A regular analog of invertible interpretation is the following.

\begin{definition}
We say that a regular interpretation $(\Gamma,\phi)\colon \MA\rightsquigarrow\MB$ has a {\em regular right inverse} if 
\begin{enumerate}[label=(\arabic*)]
\item there exists a regular interpretation $(\Delta,\psi)\colon \MB\rightsquigarrow\MA$ (a {\em regular right inverse for $(\Gamma,\phi)$}), such that
\item the regular interpretations $(\Gamma\circ\Delta,\phi_\Delta\wedge \psi)$  and $\Id_{L(\MA)}$ are regularly homotopic, i.\,e., there exist a connecting formula  $\theta(\bar u,x,\bar t,\bar w)$ and a formula-descriptor $\delta(\bar t,\bar w)$ in $L(\MA)$, where $|\bar u|=\dim\Gamma\cdot\dim\Delta$, $|\bar t|=\dim_{\param}\Gamma\cdot\dim\Delta+\dim_{\param}\Delta$, such that for every tuples of parameters $\bar a\in \phi_\Delta\wedge\psi(\MA)$ and $\bar r\in \delta(\bar a,\MA)$ (there exists such $\bar r$ for each $\bar a$) the formula $\theta(\bar u, x, \bar a, \bar r)$ defines an isomorphism $\Gamma \circ \Delta(\MA,\bar a) \simeq \MA$. 
\end{enumerate}
In this case, we refer to $(\Gamma,\phi)$ as a {\em regular left inverse} for $(\Delta,\psi)$. 
\end{definition}
 
\begin{definition}
A regular interpretation $(\Gamma,\phi)\colon \MA\rightsquigarrow\MB$ is called {\em regularly right-invertible} (correspondingly, {\em regularly left-invertible}; {\em regularly two-sided invertible}), if it has a regular right inverse (correspondingly, a regular left inverse; a regular right inverse and a regular left inverse).
\end{definition}

Recall that regular homotopy has a number of important special cases that are commonly encountered in practice. They are listed in Definition~\ref{def:reg_hom_without}. All of them also apply to regular invertibility.

\begin{remark}
As before, in the case of regularly right-invertible interpretation $(\Gamma,\phi)$ with right inverse $(\Delta,\psi)$, a connector $\theta$ and a descriptor $\delta$, for any tuples of parameters $\bar a=(\bar{\bar p},\bar q)\in \phi_\Delta\wedge\psi(\MA)$, $\bar r\in\delta(\bar a,\MA)$ and any coordinate map $\mu_\Delta\colon U_\Delta(\MA,\bar q)\to B$ there exists a coordinate map $\mu_\Gamma\colon U_\Gamma(\MB,\mu_\Delta(\bar{\bar p}))\to A$ such that the composition $\mu_\Gamma\circ\mu_\Delta$ is definable by the formula $\theta(\bar u,x, \bar a,\bar r)$.
\end{remark}

The following result is an analog of Lemma~\ref{lem:inv_def} for regularly invertible interpretations. Its premise, namely restrictions on sets $X$ and $Y$, is not exotic (see Remark~\ref{rem:WSOL-def}).

\begin{lemma}\label{lem:reg_inv_def}
Suppose there exists a regularly right-invertible interpretation $(\Gamma,\phi)$ of $\MA=\langle A;L(\MA)\rangle$ into $\MB=\langle B;L(\MB)\rangle$, $n=\dim\Gamma$; and $X\subseteq A^m$, $Y\subseteq B^{m\cdot n}$ are subsets, such that for any tuple $\bar p\in \phi(\MB)$ and any coordinate map $\mu_\Gamma$ of interpretation $(\Gamma,\bar p)\colon\MA\rightsquigarrow\MB$ one has
$Y=\mu_\Gamma^{-1}(X)$. Then $X$ is $0$\=/definable in $\MA$ if and only if $Y$ is $0$\=/definable in $\MB$.
\end{lemma}

\begin{proof}
Let $(\Delta,\psi)$ be a regular right inverse for $(\Gamma,\phi)$ and $(\theta,\delta)\colon (\Gamma\circ\Delta,\phi_\Delta\wedge\psi)\to\Id_{L(\MA)}$ be a regular homotopy. We use the Reduction Theorem~\ref{le:interpr_corol} and Lemma~\ref{le:interpr_corol_2}  all the time. 
If $X$ is $0$\=/definable in $\MA$ by a formula $\pi(x_1,\ldots,x_m)$ in $L(\MA)$ then $Y$ is $0$\=/definable in $\MB$ by the formula $\pi_{\Gamma,\exists\,\phi}(\bar x_1,\ldots,\bar x_m)$ in $L(\MB)$. Conversely, let $Y$ be absolutely definable by a formula $\chi(\bar x_1,\ldots,\bar x_m)$ in $L(\MB)$. Then $X$ is absolutely definable by the formula 
\begin{multline*}
   \pi(x_1,\ldots,x_m)=\exists\, \bar{\bar y}\:\exists \, \bar z\:\exists \,\bar w\:\exists \,\bar{\bar x}_1\:\ldots \exists \,\bar{\bar x}_m\:(\phi_\Delta(\bar{\bar y},\bar z)\wedge\psi(\bar z)\wedge\delta(\bar{\bar y},\bar z,\bar w)\wedge\\\wedge\chi_{\Delta}(\bar{\bar x}_1,\ldots,\bar{\bar x}_m,\bar z)\wedge \bigwedge\limits_{i=1}^m \theta(\bar{\bar x}_i,x_i,\bar{\bar y},\bar z,\bar w)\,). 
\end{multline*}

Indeed, let $\bar a=(a_1,\ldots,a_m)\in X$. Take any tuples $\bar p\in \phi(\MB)$ and $\bar q\in \psi(\MA)$. Fix a coordinate map $\mu_\Delta$ of the interpretation $(\Delta,\bar q)\colon\MB\rightsquigarrow\MA$ and a tuple $\bar{\bar p}\in \mu_\Delta^{-1}(\bar p)$. Then $\bar{\bar p}\in\phi_\Delta(\MA,\bar q)$. Let $\bar r\in \MA$ be a tuple, such that $\MA\models\delta(\bar{\bar p},\bar q,\bar r)$, and $\mu_\Gamma$ a coordinate map of the interpretation $(\Gamma,\bar p)\colon\MA\rightsquigarrow\MB$, such that the composition $\mu_\Gamma\circ\mu_\Delta$ is defined by the formula $\theta(\bar{\bar x}, x, \bar{\bar p},\bar q,\bar r)$. Then any tuple $\bar{\bar a}=(\bar a_1,\ldots,\bar a_m)\in \mu_\Gamma^{-1}(\bar a)$ is in $Y=\chi(\MB)$, therefore, any tuple $(\bar{\bar a}_1,\ldots,\bar{\bar a}_m)\in \mu_\Delta^{-1}(\bar{\bar a})$ is in $\chi_\Delta(\MA,\bar q)$. And since $\mu_\Gamma\circ\mu_\Delta(\bar{\bar a}_i)=\mu_\Gamma(\bar a_i)=a_i$, i.\,e., $\MA\models\theta(\bar {\bar a}_i,a_i,\bar{\bar p},\bar q,\bar r)$, $i=1,\ldots,m$, then one has $\bar a\in \pi(\MA)$.

Conversely, if $\bar a\in \pi(\MA)$, then there exist tuples $\bar q\in \psi(\MA)$, $\bar {\bar p}\in \phi_\Delta(\MA,\bar q)$, $\bar r\in \delta(\MA,\bar{\bar p},\bar q)$, $(\bar{\bar a}_1,\ldots,\bar{\bar a}_m)\in \chi_\Delta(\MA,\bar q)$, such that $\MA\models\theta(\bar {\bar a}_i,a_i,\bar{\bar p},\bar q,\bar r)$, $i=1,\ldots,m$. Fix a coordinate map $\mu_\Delta$ of the interpretation $(\Delta,\bar q)$. Then $\bar p=\mu_\Delta(\bar{\bar p})$ is in $\phi(\MB)$ and $\bar{\bar a}=(\bar a_1,\ldots,\bar a_m)=\mu_\Delta(\bar{\bar a}_1,\ldots,\bar{\bar a}_m)$ is in $Y=\chi(\MB)$. Let $\mu_\Gamma$ be a coordinate map of the interpretation $(\Gamma,\bar p)$, such that the composition $\mu_\Gamma\circ\mu_\Delta$ defines by the formula $\theta(\bar{\bar x}, x, \bar{\bar p},\bar q,\bar r)$. Then $\mu_\Gamma\circ\mu_\Delta(\bar{\bar a}_1,\ldots,\bar{\bar a}_m)=(a_1,\ldots,a_m)=\bar a$. Since $\mu_\Gamma\circ\mu_\Delta(\bar{\bar a}_1,\ldots,\bar{\bar a}_m)=\mu_\Gamma(\bar {\bar a})$ and $\bar{\bar a}\in Y$, then $\mu_\Gamma(\bar {\bar a})\in X$, i.\,e., one has $\bar a\in X$. This proves the lemma.
\end{proof}

Actually, the lemma above repeats the result~\cite[Lemma~4.9]{KhMS}. Note that the formulation of~\cite[Corollary~4.3]{KhMS} contains some inconsistencies; the intended statement there is precisely Lemma~\ref{lem:reg_inv_def}. 

\subsection{Bi-interpretability with parameters}

Let us fix two algebraic structures $\MA=\langle A;L(\MA)\rangle$ and $\MB=\langle B;L(\MB)\rangle$.

As we have already mentioned, a {\em bi-interpretation} between $\MA$ and $\MB$ is a pair of mutual interpretations $\MA\stackrel{\Gamma}{\rightsquigarrow}\MB$ and  $\MB\stackrel{\Delta}{\rightsquigarrow}\MA$, such that $\Delta$ is simultaneously a right and left inverse for $\Gamma$. But below, we will also give a stronger version of this concept. As before, we will use the notion of composition of interpretations and notations from Lemma~\ref{le:int-transitivity}. Let us remind again that when we have interpretations $\MA\stackrel{\Gamma,\bar p}{\rightsquigarrow}\MB$ and  $\MB\stackrel{\Delta,\bar q}{\rightsquigarrow}\MA$ the location of $\Gamma\circ\Delta(\MA,\bar{\bar p},\bar q)$ in $\MA$ depends on a coordinate map $\mu_\Delta$ of interpretation $\Delta$ and does not depend on a coordinate map $\mu_\Gamma$ of interpretation $\Gamma$ as well as the location of $\Delta\circ\Gamma(\MB,\bar{\bar q},\bar p)$ in $\MB$ depends on $\mu_\Gamma$ and does not depend on $\mu_\Delta$ (see Remark~\ref{rem:composition}).

\begin{definition}\label{def:bi}
Algebraic structures $\MA$ and $\MB$ are called {\em bi-interpretable} (with parameters) in each other if the conditions~\ref{bi1} and \ref{bi2} below hold:
\begin{enumerate}[label=(\arabic*)]
\item\label{bi1} there exists an interpretation $(\Gamma,\bar p)$ of $\MA$ into $\MB$ and an interpretation $(\Delta,\bar q)$ of $\MB$ into $\MA$, so the algebraic structures $\Gamma\circ\Delta(\MA,\bar{\bar p},\bar q)$ and  $\Delta\circ\Gamma(\MB,\bar{\bar q},\bar p)$  are well-defined and $\Gamma\circ\Delta(\MA,\bar{\bar p},\bar q)$ is isomorphic to $\MA$, while $\Delta\circ\Gamma(\MB,\bar{\bar q},\bar p)$ is isomorphic to $\MB$; 
\item\label{bi2} any of the following two equivalent conditions (see Lemma~\ref{le:36} below) holds:
\begin{enumerate}[label=(\roman*)]
\item\label{bi:item1} $(\Gamma\circ\Delta,\bar{\bar p},\bar q)$ is homotopic to $\Id_{L(\MA)}$, and $(\Delta\circ\Gamma,\bar{\bar q},\bar p)$ is homotopic to $\Id_{L(\MB)}$ for some coordinate maps $\mu_\Gamma$ and $\mu_\Delta$ of interpretations $\Gamma$ and $\Delta$ correspondingly, i.\,e., there exist definable coordinate maps $\mu_\MA\colon U_{\Gamma \circ \Delta}(\MA,\bar{\bar p},\bar q) \to A$ and $\mu_\MB\colon U_{\Delta \circ \Gamma}(\MB,\bar{\bar q},\bar p) \to B$;
\item\label{bi:item2} there exist coordinate maps $\mu_{\Gamma 0}$ and $\mu_\Delta$ of interpretations $\Gamma$ and $\Delta$ correspondingly such that the composition $\mu_{\Gamma 0}\circ\mu_\Delta\colon U_{\Gamma \circ \Delta}(\MA,\bar{\bar p},\bar q) \to A$ is definable in $\MA$ and there exist coordinate maps $\mu_\Gamma$ and $\mu_{\Delta 0}$ of interpretations $\Gamma$ and $\Delta$ correspondingly such that the composition $\mu_{\Delta 0}\circ\mu_\Gamma\colon U_{\Delta \circ \Gamma}(\MB,\bar{\bar q},\bar p) \to B$ is definable in $\MB$.
\end{enumerate}
\noindent 
We say that $\MA$ and $\MB$ are {\em strongly bi-interpretable} if they are bi-interpretable and the following strengthening of item~\ref{bi:item2} holds:
\begin{enumerate}[label=(\roman*), resume]
\item there exist coordinate maps $\mu_\Gamma$ and $\mu_\Delta$ of interpretations $\Gamma$ and $\Delta$ correspondingly such that the composition $\mu_\Gamma\circ\mu_\Delta\colon U_{\Gamma \circ \Delta}(\MA,\bar{\bar p},\bar q) \to A$ is definable in $\MA$ and the composition $\mu_\Delta\circ\mu_\Gamma\colon U_{\Delta \circ \Gamma}(\MB,\bar{\bar q},\bar p) \to B$ is definable in $\MB$.
\end{enumerate}
\end{enumerate}
Additionally, we  say that
\begin{itemize}
\item $\MA$ is {\em (strongly) half-absolutely bi-interpretable} with $\MB$ if the interpretation $\Gamma$ is absolute;
\item $\MA$ and $\MB$ are {\em quasi-absolutely (strongly) bi-interpretable}, if both the interpretations $\Gamma$, $\Delta$ are absolute;
\item $\MA$ and $\MB$ are {\em absolutely (strongly) bi-interpretable}, if the interpretations $\Gamma$, $\Delta$ are absolute and $\mu_\MA,\mu_\MB$ ($\mu_{\Gamma 0}\circ\mu_\Delta$, $\mu_{\Delta 0}\circ\mu_\Gamma$, $\mu_\Gamma\circ\mu_\Delta$, $\mu_\Delta\circ\mu_\Gamma$) are $0$\=/definable;
\item $\MA$ is {\em (strongly) half-injectively bi-interpretable} with $\MB$ if the interpretation $\Gamma$ is injective;
\item $\MA$ and $\MB$ are   {\em (strongly) injectively bi-interpretable}, if both the interpretations $\Gamma$ and $\Delta$ are injective. (Strong) injective bi-interpretability is also called {\em (strong) bi-definability}. 
\end{itemize}
\end{definition}

\begin{lemma} \label{le:36}
Suppose that $\MA\stackrel{\Gamma,\bar p}{\rightsquigarrow}\MB$ and $\MB\stackrel{\Delta,\bar q}{\rightsquigarrow}\MA$ are interpretations. Then the items~\ref{bi:item1} and~\ref{bi:item2} in Definition~\ref{def:bi} are equivalent.  
\end{lemma}

\begin{proof}
Indeed, if the item~\ref{bi:item1} holds, then we use the argument from Remark~\ref{rem:0} to construct $\mu_{\Gamma 0}$ and $\mu_{\Delta 0}$ as needed. And vice versa, if the item~\ref{bi:item2} holds, then for the pair of coordinate maps $\mu_\Gamma,\mu_\Delta$ the interpretations $(\Gamma\circ\Delta,\bar{\bar p},\bar q)$ and $\Id_{L(\MA)}$, as well as $(\Delta\circ\Gamma,\bar{\bar q},\bar p)$ and $\Id_{L(\MB)}$, are homotopic.
\end{proof}

We denote bi-interpretation by $(\Gamma; \Delta)$, or by $(\Gamma, \bar p; \Delta, \bar q)$, or by $(\Gamma, \bar p, \mu_\Gamma; \Delta, \bar q, \mu_\Delta)$, sometimes by $(\Gamma,\bar p, \theta_\MB;\Delta,\bar q,\theta_\MA)$ or by $(\Gamma,\bar p,\mu_\Gamma, \theta_\MB;\Delta,\bar q,\mu_\Delta,\theta_\MA)$, where $\theta_\MA$ is a formula that defines $\mu_\MA$ and $\theta_\MB$ is a formula that defines $\mu_\MB$. Here notion $(\Gamma;\Delta)$ does not mean that interpretations $\Gamma,\Delta$ are absolute, but simply in this case, there is no need to focus on the parameters. If $\Gamma,\Delta$ are both absolute, we write also $(\Gamma,\emptyset;\Delta,\emptyset)$.

The following results show that sometimes invertible interpretations are in fact bi-interpretations.

\begin{lemma}\label{i-rigid}
Let $\MB$ be an i\=/rigid algebraic structure and $(\Gamma,\bar p,\mu_\Gamma)\colon\MA\rightsquigarrow\MB$ be a right-invertible interpretation with a right inverse $(\Delta,\bar q,\mu_\Delta)$. Then $(\Gamma,\bar p,\mu_\Gamma;\Delta,\bar q,\mu_\Delta)$ is a bi-interpretation between $\MA$ and $\MB$. 
\end{lemma}

\begin{proof}
    From the conditions of the lemma $(\Gamma\circ\Delta,\bar{\bar p},\bar q)$ is homotopic to $\Id_{L(\MA)}$. And since $\MB$ is i\=/rigid, $(\Delta\circ\Gamma,\bar{\bar q},\bar p)$ is homotopic to $\Id_{L(\MB)}$.
\end{proof}

\begin{lemma}\label{rigid+i-rigid}
Let $\MB$ be a rigid and i\=/rigid algebraic structure and $(\Gamma,\bar p,\mu_\Gamma)\colon\MA\rightsquigarrow\MB$ be a strongly right-invertible interpretation with strong right inverse $(\Delta,\bar q,\mu_\Delta)\colon\MB\rightsquigarrow\MA$. Then $(\Gamma,\bar p,\mu_\Gamma;\Delta,\bar q,\mu_\Delta)$ is a strong bi-interpretation between $\MA$ and $\MB$. 
\end{lemma}

\begin{proof}
Since $\MB$ is i\=/rigid, interpretations $\Delta\circ\Gamma(\MB,\bar{\bar q},\bar p)$ and $\Id_{L(\MB)}$ are homotopic. And since, $\MB$ is rigid, there exits a unique coordinate map $\mu_{\Delta\circ\Gamma}\colon U_{\Delta\circ\Gamma}(\MB,\bar{\bar q},\bar p)\to B$, so it is $\mu_\Delta\circ\mu_\Gamma$ and it is definable in $\MB$, as required. 
\end{proof}

\begin{corollary}\label{cor:rigid+i-rigid}
Let $\MB$ be a rigid and i\=/rigid algebraic structure and $(\Gamma,\bar p)\colon\MA\rightsquigarrow\MB$ be a right-invertible interpretation with right inverse $(\Delta,\bar q)$. Then $(\Gamma,\bar p;\Delta,\bar q)$ is a strong bi-interpretation between $\MA$ and $\MB$ for some coordinate maps $\mu_\Gamma$ and $\mu_\Delta$. 
\end{corollary}

\begin{remark}
Assume that $\MA$ is half-absolute bi-interpretable with  $\MB$, i.\,e., in the notations of Definition~\ref{def:bi}, the interpretation $\Gamma$ is absolute. Then there exist coordinate maps $\mu_\Gamma$ of $\Gamma$, $\mu_\Delta$ of $\Delta$ and $\mu_\MA$ of $\Gamma\circ\Delta$ such that  $\mu_\MA\colon U_{\Gamma \circ \Delta}(\MA,\bar q) \to A$ and $\mu_\Delta\circ\mu_\Gamma\colon U_{\Delta \circ \Gamma}(\MB,\bar{\bar q}) \to B$ are definable in $\MA$ and $\MB$ correspondingly. To achieve this, we may have to redefine the coordinate map $\mu_\Delta$, like in Remark~\ref{rem:0}, but since the interpretation $\Gamma$ is absolute, this will not affect the layout of $L(\MA)$\=/structure $\Gamma\circ\Delta(\MA,{\bar q})$. Note that there is an inaccuracy in~\cite[Remark~4.1~c)]{KhMS}, where it is claimed that the strengthening as above can be done for arbitrary bi-interpretations. However, in general, we cannot do this, since changing any of the coordinate maps $\mu_\Gamma$ or $\mu_\Delta$ can make maps $U_{\Gamma \circ \Delta}(\MA,\bar{\bar p},\bar q) \to A$ or $U_{\Delta \circ \Gamma}(\MB,\bar{\bar q},\bar p) \to B$ non-definable (see Example~\ref{ex:H}).
\end{remark}

\begin{lemma}\label{rigid4}
   Suppose that $\MA$ is a rigid algebraic structure and $(\Gamma,\emptyset,\theta_\MB;\Delta,\bar q, \theta_\MA)$ is a half-absolute bi-interpretation of $\MA$ with $\MB$. Then $(\Gamma,\emptyset,\theta_\MB;\Delta,\bar q, \theta_\MA)$ is a strong half-absolute bi-interpretation.
\end{lemma}

\begin{proof}
    Indeed, since $\Gamma$ is absolute, there exist coordinate maps $\mu_\Gamma$ and $\mu_\Delta$, such that $\mu_\Delta\circ\mu_\Gamma$ is definable in $\MB$. And since the group $\Aut(\MA)$ is trivial, $\mu_\Gamma\circ\mu_\Delta=\mu_\MA$ is definable in $\MA$.
\end{proof}

We will sometimes refer to an ordinary bi-interpretation as a {\em weak bi-interpretation} for contrast. Observe that the notion of a bi-interpretation defined in the books~\cite{Hodges} and~\cite{KhMS} is weak, but in the paper~\cite{AKNS} it is strong. There are many interesting applications of strong bi-interpretations that we cannot derive from the weak ones. At the same time, all the consequences of the weak bi-interpretation, known to us, are in fact consequences of weaker conditions, namely, of right or two-sided invertible interpretations. However, it is important to mention that right now we do not have examples of weak bi-interpretations of algebraic structures that are not strong, so we leave this as an open problem.

\begin{problem}
Find two algebraic structures $\MA$ and $\MB$, which are weakly bi-interpretable, but are not strongly bi-interpretable. 
\end{problem}

The following example shows that there are structures $\MA$ and $\MB$ such that $\MA$ is invertibly interpretable in $\MB$ but not bi-interpretable in $\MB$. 

\begin{example}
    The Heisenberg group $\UT_3(\Z)$ and $\Z$ are mutually interpretable in each other (see Examples~\ref{ex:UT1},~\ref{ex:UT2}). Since any interpretation $\Z\rightsquigarrow\Z$ is homotopic to $\Id_{L(\Z)}$~\cite[Lemma~2.7]{AKNS}, there exists an invertible interpretation $\Z\rightsquigarrow \UT_3(\Z)$. At the same time, there is no bi-interpretation between $\Z$ and $\UT_3(\Z)$~\cite{KhMS, Khelif, Nies2}. 
\end{example}

\subsection{Absolute bi-interpretability}\label{subsec:abs_bi}

Strong absolute bi-interpretation is one of the closest model-theoretic relations between algebraic structures $\MA$ and $\MB$. 

\begin{example}
Algebraic structures $\N$, $\Z$ and $\Q$ are absolutely strongly injectively bi-in\-ter\-pre\-table pairwise.    
\end{example}

\begin{example}
In~\cite{ABM} M.\,G.\,Amaglobeli, T.\,Z.\,Bokelavadze, 
A.\,G.\,Myasnikov showed that a nilpotent Lie algebra $L$ over a field $k$ of characteristic zero in the language of $k$\=/algebras and the Hall nilpotent $k$\=/group $G(L)$ in the language of $k$\=/groups are absolutely injectively strongly bi-interpretable in each other by equations.
\end{example}

Every $L(\MA)$\=/isomorphism $h\colon\MA\to\MA^\prime$ gives rise an absolute strong injective bi-interpretability $(\Id_{L(\MA)},\emptyset,h^{-1};\Id_{L(\MA)},\emptyset,h)$.

The following fact is a direct consequence of Remarks~\ref{lem:AB1} and~\ref{lem:AB3}.

\begin{lemma}\label{lem:bi}
Let $\MA$ and $\MB$ be algebraic structures. If $\MA$ and $\MB$ are strongly (injectively) bi-interpretable by means of $(\Gamma,\bar p,\mu_\Gamma;\Delta,\bar q,\mu_\Delta)$, then $\MA_A$ and $\MB_B$ are strongly (injectively) $0$\=/bi-interpretable by means of $(\Gamma_A,\emptyset,\mu_\Gamma;\Delta_B,\emptyset,\mu_\Delta)$. If the languages $L(\MA)$ and $L(\MB)$ are finite, then the inverse statement also holds. We cannot say the same for a weak bi-interpretation.
\end{lemma}

Absolute interpretation is extremely rare. It is directly related to the following fact.

\begin{theorem}[Sec.~5.4, Ex.~8\,(b),~\cite{Hodges}]\label{th:Aut}
If algebraic structures $\MA$ and $\MB$ are absolutely bi-interpretable, then one has 
$$
\Aut(\MA)\simeq\Aut(\MB).
$$
\end{theorem}

\begin{corollary}
Suppose that $\MA$ and $\MB$ are absolutely bi-interpretable. Then $\MA$ is rigid if and only if $\MB$ is rigid.  
\end{corollary}

\begin{corollary}
  If $\MA$ is rigid and $(\Gamma,\emptyset,\theta_\MB;\Delta,\emptyset, \theta_\MA)$ is an absolute bi-interpretation between $\MA$, then it is a strong absolute bi-interpretation.
\end{corollary}

\begin{corollary}
The ring of integers $\Z$ and the metabelian Baumslag\,--\,Solitar group $\BS(1,k)$, $k>1$, are not absolutely bi-interpretable. 
\end{corollary}

\begin{proof}
Indeed, the group $\Aut(\Z)$  is trivial, but $\Aut(\BS(1,k))$ is not (see, for example,~\cite[Lemma~2]{BS}).  
\end{proof}

\begin{corollary}
An arbitrary ring of algebraic integers $\Aint$ is mutually absolutely interpretable with $\Z$. However, $\Aint$ and $\Z$ are not necessarily absolutely bi-interpretable. Indeed, $\Z$ and the ring of Gaussian integers $\Z[i]$ are not absolutely bi-interpretable since $\Aut(\Z[i]) \neq e$.  
\end{corollary}

In practice, regular bi-interpretations are more common, which has many pleasant consequences.

\subsection{Regular bi-interpretability}

Now we introduce a regular analog of bi-interpretation.

\begin{definition}\label{def:reg}
Algebraic structures $\MA$ and $\MB$ are called {\em regularly bi-interpretable}, if 
\begin{enumerate}[label=(\Roman*)]
\item there exist a regular interpretation $(\Gamma,\phi)$ of $\MA$ in $\MB$ and a regular interpretation $(\Delta,\psi)$ of $\MB$ in $\MA$, such that 
\item $(\Gamma\circ\Delta, \phi_\Delta\wedge\psi)$ is regularly homotopic to $\Id_{L(\MA)}$, and $(\Delta\circ\Gamma,\psi_\Gamma\wedge \phi)$ is regularly homotopic to $\Id_{L(\MB)}$, i.\,e.,
\begin{enumerate}[label=(\alph*)]
\item there exist a connecting formula  $\theta_\MA(\bar u, x, \bar{\bar y},\bar z,\bar w)$ and a formula-descriptor $\delta_\MA(\bar{\bar y},\bar z,\bar w)$ in $L(\MA)$, where $|\bar u|={\dim\Gamma\cdot\dim\Delta}$, $|\bar{\bar y}|=\dim_\param \Gamma\cdot\dim\Delta$, $|\bar z|=\dim_\param\Delta$, such that $\MA\models\forall\,\bar{\bar y}\:\forall\,\bar z\:(\phi_\Delta(\bar{\bar y},\bar z)\wedge\psi(\bar z)\longrightarrow\exists\,\bar w\:\delta_\MA(\bar{\bar y},\bar z,\bar w))$ and for any tuples $\bar q\in \psi(\MA)$, $\bar{\bar p}\in \phi_\Delta(\MA,\bar q)$, $\bar r\in\delta_\MA(\bar{\bar p},\bar q,\MA)$  the formula $\theta_\MA(\bar u, x, \bar{\bar p},\bar q,\bar r)$ defines some coordinate map $U_{\Gamma\circ\Delta}(\MA,\bar{\bar p},\bar q)\to A$;
\item there exist connecting formula $\theta_\MB(\bar u, x, \bar{\bar z},\bar y, \bar v)$ and a formula-descriptor $\delta_\MB(\bar{\bar z},\bar y,\bar v)$ in $L(\MB)$, where $|\bar u|={\dim\Gamma\cdot\dim\Delta}$, $|\bar{\bar z}|=\dim_\param\Delta\cdot\dim\Gamma$, $|\bar y|=\dim_\param\Gamma$, such that $\MB\models\forall\,\bar{\bar z}\:\forall\,\bar y\:(\psi_\Gamma(\bar{\bar z},\bar y)\wedge\phi(\bar y)\longrightarrow\exists\,\bar v\:\delta_\MB(\bar{\bar z},\bar y,\bar v))$ and for any tuples $\bar p\in\phi(\MB)$, $\bar{\bar q}\in \psi_\Gamma(\MB,\bar p)$, $\bar s\in \delta_\MB(\bar{\bar q},\bar p,\MB)$ 
the formula $\theta_\MB(\bar u, x, \bar{\bar q},\bar p, \bar s)$ defines some coordinate map $U_{\Delta\circ\Gamma}(\MB,\bar{\bar q},\bar p)\to B$. 
\end{enumerate}
\end{enumerate}
\end{definition}

Regular bi-interpretation as above between $\MA$ and $\MB$ are denoted by $(\Gamma;\Delta)$, or by $(\Gamma,\phi;\Delta,\psi)$, or by $(\Gamma,\phi,\theta_\MB,\delta_\MB;\Delta,\psi,\theta_\MA,\delta_\MA)$.

In Definition~\ref{def:reg_hom_without} we describe several special cases of regular homotopy that commonly arise in research: (1)~when one (or both) of the interpretations is absolute, (2)~without a descriptor, (3)~with a removable descriptor, and (4)~with a split descriptor. Each of these variants applies equally to regular bi-interpretability. 

In practice, the two most important kinds of bi-interpretations are the strong bi-interpretation with parameters and the regular bi-interpretation (see Section~\ref{sec:app}).
Moreover, quite often a regular interpretation yields a strong bi-interpretation with parameters once suitable parameters are substituted. To account for bi-interpretations of this kind, we introduce the following definition.

\begin{definition}\label{def:st_reg}
We say that $\MA$ and $\MB$ are {\em strongly regularly bi-interpretable without  descriptors} (or {\em without homotopy descriptors}), if
\begin{enumerate}[label=(\Roman*)]
\item there exist a regular interpretation $(\Gamma,\phi)$ of $\MA$ in $\MB$ and a regular interpretation $(\Delta,\psi)$ of $\MB$ in $\MA$, such that 
\item\label{item:reg2} $(\Gamma\circ\Delta, \phi_\Delta\wedge\psi)$ is regularly homotopic to $\Id_{L(\MA)}$ without a descriptor, and $(\Delta\circ\Gamma,\psi_\Gamma\wedge \phi)$ is regularly homotopic to $\Id_{L(\MB)}$ without a descriptor, i.\,e.,
\begin{enumerate}[label=(\alph*)]
\item there exists a connecting formula $\theta_\MA(\bar u, x, \bar{\bar y},\bar z)$ in $L(\MA)$, where $|\bar u|={\dim\Gamma\cdot\dim\Delta}$, $|\bar{\bar y}|=\dim_\param \Gamma\cdot\dim\Delta$, $|\bar z|=\dim_\param\Delta$, such that for any tuples $\bar q\in \psi(\MA)$, $\bar{\bar p}\in \phi_\Delta(\MA,\bar q)$ the formula $\theta_\MA(\bar u, x, \bar{\bar p},\bar q)$ defines some coordinate map $U_{\Gamma\circ\Delta}(\MA,\bar{\bar p},\bar q)\to A$;
\item there exists connecting formula $\theta_\MB(\bar u, x, \bar{\bar z},\bar y)$ in $L(\MB)$, where $|\bar u|={\dim\Gamma\cdot\dim\Delta}$, $|\bar{\bar z}|=\dim_\param\Delta\cdot\dim\Gamma$, $|\bar y|=\dim_\param\Gamma$, such that for any tuples $\bar p\in\phi(\MB)$, $\bar{\bar q}\in \psi_\Gamma(\MB,\bar p)$ the formula $\theta_\MB(\bar u, x, \bar{\bar q},\bar p)$ defines some coordinate map $U_{\Delta\circ\Gamma}(\MB,\bar{\bar q},\bar p)\to B$; 
\end{enumerate}
\item\label{item:reg3} furthermore, for any 
tuples of parameters $\bar p\in \phi(\MB)$ and $\bar q\in \psi(\MA)$  there exists a pair of coordinate maps $(\mu_\Gamma,\mu_\Delta)$ for interpretations $(\Gamma,\bar p)$ and $(\Delta,\bar q)$, such that for any $\bar{\bar p}\in \mu^{-1}_\Delta(\bar p)$ and $\bar{\bar q}\in \mu^{-1}_\Gamma(\bar q)$ the coordinate maps ${\mu_\Gamma\circ\mu_\Delta\colon} U_{\Gamma\circ\Delta}(\MA,\bar{\bar p},\bar q)\to A$ and ${\mu_\Delta\circ\mu_\Gamma\colon} U_{\Delta\circ\Gamma}(\MB,\bar{\bar q},\bar p)\to B$ are defined in $\MA$ and $\MB$ correspondingly by the formulas $\theta_\MA(\bar u, x, \bar{\bar p},\bar q)$ and $\theta_\MB(\bar u, x, \bar{\bar q},\bar p)$.
\end{enumerate}
\end{definition}

We note that strong regular bi-interpretations without descriptors do indeed occur in a number of settings. However, we do not currently know of any consequences of this notion that are not already implied either by regular bi-interpretation or by strong bi-interpretation with parameters. There are several conceivable ways to extend strong regular bi-interpretation to incorporate descriptors, but their practical relevance is unclear, and at present, we have no examples in which a regular bi-interpretation can be obtained using only descriptors. For these reasons, we will not pursue this approach further in this article, leaving it for future research.

\begin{remark}
Note that item~\ref{item:reg3} in Definition~\ref{def:st_reg} does not imply item~\ref{item:reg2}. Indeed, the first statement in item~\ref{item:reg2} must hold for all tuples of parameters from the set $X = \phi_\Delta \wedge \psi(\MA)$. By Lemma~\ref{le:int-transitivity} and Remark~\ref{remark2}, we have
$$
X = \{(\bar{\bar p}, \bar q) \mid \bar q \in \psi(\MA),\ \bar p \in \phi(\MB),\ \bar{\bar p} \in \mu_\Delta^{-1}(\bar p)\},
$$
where $\mu_\Delta$ ranges over all coordinate maps of the interpretation $\MB \simeq \Delta(\MA, \bar q)$. Now denote by $\mu_\Delta^{\bar p, \bar q}$ the coordinate map $\mu_\Delta$ of the interpretation $\MB \simeq \Delta(\MA, \bar q)$ from item~\ref{item:reg3} corresponding to the pair of parameters $(\bar p, \bar q)$. The statement in item~\ref{item:reg3} must hold for all tuples of parameters from the set
$$
Y = \{(\bar{\bar p}, \bar q) \mid \bar q \in \psi(\MA),\ \bar p \in \phi(\MB),\ \bar{\bar p} \in (\mu_\Delta^{\bar p, \bar q})^{-1}(\bar p)\}.
$$
Thus $Y \subseteq X$, but it may happen that $Y \neq X$.
\end{remark}

Naturally, analogously to bi-interpretations with parameters, we define the notions of 
\begin{itemize}\setlength\itemsep{0em}
    \item {\em regular injective (strong) bi-interpretation} or {\em bi-definability};
    \item {\em regular half-injective (strong) bi-interpretation};
    \item {\em regular half-absolute (strong) bi-interpretation};
    \item {\em regular quasi-absolute (strong) bi-interpretation}.
\end{itemize}
We usually encounter all of them in practice.

\begin{example}
In~\cite{MS} A.\,G.\,Myasnikov and M.\,Sohrabi proved that the group $\SL_n(F)$ over a field $F$ is regularly strongly half-absolute injectively bi-interpretable (without descriptors) with the field $F$ if $n\geqslant 3$.
\end{example}

\begin{example}
Suppose that $G(R) = G_{\mathrm{ad}}(\Phi, R)$ is an adjoint Chevalley group of $\mathrm{rank} > 1$, $R$ is a commutative ring (with $\frac{1}{2}$
for the root systems $B_{\ell}$, $C_{\ell}$, $F_4$, $G_2$ and with $\frac{1}{3}$ for $G_2$). E.\,Bunina and P.\,Gvozdevsky have proved that the group $G(R)$ is regularly strongly half-absolute injectively bi-interpretable  (without descriptors) with the ring $R$~\cite{BuninaGv}.
\end{example}

\begin{example}
B.\,A.\,Berger and A.\,G.\,Myasnikov in~\cite{Berger} proved that for any natural $n\geqslant 3$ there exist a code $\Gamma$ with a formula-descriptor $\phi$ in the language $\{\subseteq\}$ and an absolute code $\Delta$ in the language of rings $\{+,\,\cdot\,\}$ that for any field $F$ and any linear space $V$ of dimension $n$ over $F$ give a regular injective strong half-absolute bi-interpretation of projective space $\mathrm{P}(V)$ in $F$  (without homotopy descriptors).
\end{example}

\begin{example}\label{ex:BS}
E.\,Daniyarova and A.\,Myasnikov show in~\cite{BS} that the non-abelian metabelian Baumslag\,--\,Solitar group $\BS(1,k)$, $k> 1$, is regularly strongly half-absolute bi-interpretable with ring of integers $\Z$ (without descriptors). At the same time, $\Z$ is regularly strongly half-injectively bi-interpretable with $\BS(1,k)$  (without descriptors).
\end{example}

\begin{example}
Any finitely generated ring of algebraic integers $\Aint$ is strongly half-absolute injectively bi-interpretable with $\Z$ (without descriptors). 
\end{example}

If algebraic structures $\MA$ and $\MB$ are strongly regularly bi-interpretable without descriptors, then we may use all conclusions from both regular bi-interpretability and strong bi-interpretation with parameters. Unlike bi-interpretation with parameters, we know results on regular bi-interpretation, which deepen the corresponding results on regularly invertible interpretations.

\begin{problem}
Find two algebraic structures $\MA$ and $\MB$, which are regularly bi-interpretable, but are not strongly regularly bi-interpretable. 
\end{problem}

Let us formulate and prove regular analogs of Lemmas~\ref{i-rigid} and~\ref{rigid+i-rigid}.

\begin{lemma}\label{reg_i-rigid}
Let $\MB$ be a regularly i\=/rigid algebraic structure and $(\Gamma,\phi)\colon\MA\rightsquigarrow\MB$ be regular right-invertible interpretation with right inverse $(\Delta,\psi)$. Then $(\Gamma,\phi;\Delta,\psi)$ is a regular bi-interpretation between $\MA$ and $\MB$. 
\end{lemma}

\begin{proof}
    By the condition, $(\Gamma\circ\Delta, \phi_\Delta\wedge\psi)$ is regularly homotopic to $\Id_{L(\MA)}$. And since $\MB$ is regularly i\=/rigid, $(\Delta\circ\Gamma,\psi_\Gamma\wedge \phi)$ is regularly homotopic to $\Id_{L(\MB)}$.
\end{proof}

\begin{lemma}\label{reg_rigid+i-rigid}
Let $\MB$ be a rigid and regularly i\=/rigid algebraic structure and $(\Gamma,\emptyset)\colon\MA\rightsquigarrow\MB$ be an absolute regularly right-invertible interpretation with right inverse $(\Delta,\psi)$ and without homotopy descriptor. Then $(\Gamma,\emptyset;\Delta,\psi)$ is a regular strong bi-interpretation between $\MA$ and $\MB$ without homotopy descriptors. 
\end{lemma}

\begin{proof}
By Lemma~\ref{reg_i-rigid}, $(\Gamma,\emptyset;\Delta,\psi)$ is a regular bi-interpretation between $\MA$ and $\MB$. Let $(\theta_\MA,\emptyset)\colon (\Gamma\circ\Delta,\psi)\to \Id_{L(\MA)}$ and $(\theta_\MB,\emptyset)\colon (\Delta\circ\Gamma,\psi_\Gamma)\to \Id_{L(\MB)}$ be corresponding regular homotopies without homotopy descriptors. We need to check that condition~\ref{item:reg3} from Definition~\ref{def:st_reg} hold. Take any tuple of parameters $\bar q\in \psi(\MA)$ and some coordinate map $\mu_\Delta$ of the interpretation $(\Delta,\bar q)$. Suppose that $\mu_\Gamma$ is a coordinate map of $\Gamma$, such that the formula $\theta_\MA(\bar u,x,\bar q)$ defines the coordinate map $\mu_\Gamma\circ\mu_\Delta\colon U_{\Gamma\circ\Delta}(\MA,\bar q)\to A$. For any $\bar{\bar q}\in \mu^{-1}_\Gamma(\bar q)$ the formula $\theta_\MB(\bar u, x, \bar{\bar q})$ defines some coordinate map $\mu_{\Delta\circ\Gamma}\colon U_{\Delta\circ\Gamma}(\MB,\bar {\bar q})\to B$. Since $\MB$ is rigid, $\mu_{\Delta\circ\Gamma}=\mu_\Delta\circ\mu_\Gamma$, so the pair $(\mu_\Gamma,\mu_\Delta)$ is the desired one. 
\end{proof}

In Example~\ref{ex1} below, one can find a case of two algebraic structures, which are strongly bi-interpretable with parameters, but not regularly bi-interpretable.

\section{Applications of interpretations}\label{sec:app}

The different interpretability options described above allow for different conclusions to be drawn. Some of them are known, and others will be proved in subsequent papers in this series.

\subsection{Undecidability of elementary theory}\label{subsec:undecidability}

In this section, we assume that all the theories $T$ in a language $L$ that we consider are deductively closed, i.e., for any $L$\=/sentence $\psi$, if $T \vdash \psi$, then $\psi \in T$. Recall that a theory $T$ in a language $L$ is called \emph{decidable} if there exists an algorithm which, given any sentence $\psi$ of $L$, determines whether or not $\psi$ belongs to $T$. If no such algorithm exists, we say that $T$ is \emph{undecidable}. Moreover, $T$ is said to be \emph{hereditarily undecidable} if every subtheory $T_0 \subseteq T$ is undecidable. 

A.\,Tarski proved that the elementary theory $\Th(\Z)$ of the ring of integers $\Z$ is hereditarily undecidable (see~\cite{TRM}). Yu.\,Ershov et al. later demonstrated that every finitely axiomatizable undecidable theory is hereditarily undecidable~\cite[Corollary~3.4.1]{Ershov-Lavrov}.

The main result of this section is that interpretations with parameters transfer the property of hereditary undecidability. Originally, these results were central to studying interpretations of algebraic structures. Nowadays, there are several variations, and we mention here the one that is closest to our exposition. We also provide a proof adapted to our terminology and notation.

\begin{theorem}[{\cite[Ershov, Theorem~2, p.\,272]{Ershov}}]\label{le:undec-Z}
Suppose that the language $L(\MA)$ is finite and the first-order theory $\Th(\MA)$ is hereditarily undecidable. If $\MA$ is interpretable (with parameters) in $\MB$, then the first-order theory $\Th(\MB)$ is hereditarily undecidable too. 
\end{theorem}

\begin{proof}
Suppose that $\MA\simeq\Gamma(\MB,\bar p)$. Since the language $L(\MA)$ is finite, then there exists an $L(\MB)$\=/formula $AC_\Gamma(\bar y)$, such that for any $L(\MB)$\=/structure $\nsB$ and any tuple $\bar q\in \nsB$ an $L(\MA)$\=/structure $\Gamma(\nsB,\bar q)$ is well-defined if and only if $\nsB\models AC_\Gamma(\bar q)$ (see Remark~\ref{single_AC}). Further we use the $(\Gamma,AC_\Gamma)$\=/translation $\psi\to\psi_{\Gamma,\forall\, AC_\Gamma}$ of sentences in $L(\MA)$ into sentences in $L(\MB)$ (see Definition~\ref{def:reg_transl}).
    
Take any theory $T_B$ in the language $L(\MB)$, such that $T_B\subseteq \Th(\MB)$. We need to show that $T_B$ is undecidable. Let $T_A$ be the set of all sentences $\psi$ in $L(\MA)$, such that $\psi_{\Gamma,\forall\, AC_\Gamma}\in T_B$. Denote by $\K$ the class of all algebraic structures in the language $L(\MA)$ of the form $\Gamma(\nsB,\bar q)$, where $\nsB\in \Mod(T_B)$ and $\bar q\in AC_\Gamma(\nsB)$. It is obvious that $\K\ne\emptyset $ since $\MA\in \K$. Let us check that $T_A=\Th(\K)$. Due to Corollary~\ref{le:interpr_corol_ns}, for any $L(\MA)$\=/sentence $\psi$ one has $\Gamma(\nsB,\bar q)\models \psi$ if and only if $\nsB\models\psi_\Gamma(\bar q)$. Thus $\psi\in T_A$ if and only if $\psi_{\Gamma,\forall\, AC_\Gamma}\in T_B$ if and only if $\nsB\models \psi_{\Gamma,\forall\, AC_\Gamma}$ for every algebraic structure $\nsB\in \Mod(T_B)$ if and only if $\nsB\models \psi_\Gamma(\bar q)$ for every $\nsB\in \Mod(T_B)$ and $\bar q\in AC_\Gamma(\nsB)$ if and only if $\Gamma(\nsB,\bar q)\models\psi$ for every $\nsB\in \Mod(T_B)$ and $\bar q\in AC_\Gamma(\nsB)$ if and only if $\psi \in \Th(\K)$. So, $T_A=\Th(\K)$, i.\,e., $T_A$ is a theory in $L(\MA)$ and $\MA\in\Mod(T_A)$, hence $T_A\subseteq \Th(\MA)$. If $T_B$ is decidable, then the translation $\psi\to\psi_{\Gamma,\forall\, AC_\Gamma}$ give an algorithm to answer whether or not $\psi \in T_A$, which states decidability of $T_A$. Since $\Th(\MA)$ is hereditarily undecidable, $T_A$ is undecidable; therefore, $T_B$ is undecidable as well.
\end{proof}

\begin{corollary}
   In Section~\ref{subsec:def_int}, we have mentioned that the ring $\Z$ can be interpreted in several mathematical structures: a ring of algebraic integers $\Aint$, a number field $K$, any finitely generated infinite integral domain, any finitely generated infinite field $F$ with $\mathrm{char} \, F \neq 2$, and a finitely generated non-virtually abelian solvable group $G$. As a result, each of these structures has a hereditarily undecidable elementary theory. 
\end{corollary}

Sometimes, interpretations transfer undecidability itself, but for this, regular interpretability is required. 

\begin{theorem}\label{co:interp1} 
Suppose that $\MA$ and $\MB$ are algebraic structures in computable languages and $\MA$ is regularly interpretable in $\MB$ by means of a computable code $\Gamma$. If the first-order theory $\Th(\MB)$ is decidable, then $\Th(\MA)$ is also decidable. Correspondingly, if $\Th(\MA)$ is undecidable, then  $\Th(\MB)$ is undecidable as well.
\end{theorem}

\begin{proof}
The result is directly deduced from Lemma~\ref{le:interpr_corol_2}. 
\end{proof}

\begin{example}
For any field $F$ and any linear space $V$ of dimension $\geqslant3$ over field $F$ there exists a regular interpretation $F\rightsquigarrow \mathrm{P}(V)$ of field $F$ into projective space $\mathrm{P}(V)$ over $V$~\cite{Berger}. Therefore, if the elementary theory $\Th(F)$ is undecidable (for instance, if $F=\Q$), then the elementary theory $\Th(\mathrm{P}(V))$ is undecidable too.
\end{example}

It is time to consider an example of interpretability with parameters that exclude regularity.

\begin{example}\label{ex1}
Suppose that $c$ is a real number whose decimal notation is not computable. The algebraic structure $\langle \MR; +, \cdot, c\rangle$ is interpretable in $\langle \MR; +, \cdot\rangle$ with parameter $c$, but not regularly interpretable in it, because the elementary theory of the second structure is decidable, but the first is not. Algebraic structures $\langle \MR; +, \cdot, c\rangle$ and $\langle \MR; +, \cdot\rangle$ are also strongly bi-interpretable with parameters in each other, but not regularly bi-interpretable.
\end{example}

\subsection{Diophantine problem and \texorpdfstring{e\=/interpretations}{}}\label{subsec:Diophantine}

Diophantine interpretations are one of the main tools to prove decidability/undecidability of the Diophantine problem in algebraic structures.

Recall that the {\em Diophantine problem} for a given algebraic structure $\MA=\langle A; L(\MA)\rangle$ asks if there exists an algorithm that decides whether or not a given finite system of Diophantine equations, i.\,e., of atomic formulas in $L(\MA)\cup A$, has a solution in $\MA$; in other words, if there is an algorithm that for a given Diophantine sentence $\psi$ of the language $L(\MA)\cup A$ decides whether or not $\MA\models \psi$. Yu.\,Matiyasevich, following the work of M.\,Davis, H.\,Putnam, and J.\,Robinson, proved that the ring $\Z$ has an undecidable Diophantine problem~\cite{Matijasevich}, thus solving the Tenth Hilbert Problem.

\begin{lemma}
Suppose that $\MA$ and $\MB$ are algebraic structures in computable languages and $\MA$ is e\=/interpretable (with parameters) in $\MB$ by means of a computable Diophantine code $\Gamma$.
Then if the Diophantine problem in $\MB$ is decidable, then the Diophantine problem in $\MA$ is also decidable. Correspondingly, if the Diophantine problem in $\MA$ is undecidable, then the Diophantine problem in $\MB$ is undecidable as well. Moreover, if the languages of $\MA$ and $\MB$ are finite, then there is a polynomial-time reduction of the Diophantine problem in $\MA$ to the Diophantine problem in $\MB$.
\end{lemma}

\begin{proof}
It follows from the Reduction Theorem~\ref{le:interpr_corol} and Remarks~\ref{rem:comp-transl}, \ref{rem:Diophantine}.
\end{proof}

Using e\=/interpretation J.\,Denef proved that the ring of polynomials in one variable $K[x]$ over an integral domain $K$ has undecidable Diophantine problem~\cite{Denef1, Denef2}; O.\,Kharlampovich and A.\,Myasnikov proved that the Diophantine problem is undecidable over
\begin{enumerate}[label=(\arabic*)]
    \item free associative algebras $A$ over a field $K$,
    \item free Lie algebras $L$ of rank $r\geqslant 3$ a field $K$,
    \item group algebra $K(G)$ of a torsion-free hyperbolic or a toral relatively hyperbolic
group $G$ over a field $K$~\cite{KM_Diophantine1,KM_Diophantine2}.
\end{enumerate}

The Diophantine problem in commutative rings (not necessarily unitary) was studied in detail in~\cite{GMO}, and for finite-dimensional and similar algebras/rings in~\cite{GMO1}. In~\cite{MS} it was shown that the Diophantine problem in the classical matrix groups $\GL_n(R), \SL_n(R), \mathrm{T}_n(R), \UT_n(R)$, $n \geqslant 3$, over associative unitary rings $R$ is polynomial time equivalent (more precisely, Karp equivalent) to the Diophantine problem in $R$. In the case of $\SL_n(R)$, the ring $R$ is assumed to be commutative. Similar results hold for $\mathrm{PGL}_n(R)$ and $\mathrm{PGL}_n(R)$, provided $R$ has no zero divisors (for $\mathrm{PGL}_n(R)$, the ring $R$ is not assumed to be commutative). These results were generalized for Chevalley groups in~\cite{BMP}. The Diophantine problem in solvable groups was intensively studied in algorithmic group theory~\cite{GMO2}; we refer to the paper~\cite{MS} for the latest results and a survey on the known ones. 

It is important to note that all the aforementioned results were obtained through Diophantine interpretations of $\Z$ within the structures being examined.

\subsection{Non-standard models and the first-order classification problem}\label{subsec:ns}

Recall that the {\em first-order classification problem for an algebraic structure $\MA$} is to find an ``algebraic description'' of all models of elementary theory $\Th(\MA)$, i.\,e., of all algebraic structures $\nsA$ in the language $L(\MA)$ first-order
equivalent to $\MA$, $\nsA\equiv\MA$. 

\begin{theorem}[expected in the next paper]\label{th:sk_el}
Let an infinite algebraic structure $\MA$ be interpretable with parameters in an algebraic structure $\MB$, $\MA \simeq \Gamma(\MB,\bar p)$. Then every algebraic structure $\nsA\equiv\MA$ is elementarily embeddable in an algebraic structure $\Gamma(\nsB,\bar q)$ for some $\nsB\equiv\MB$ and $\bar q\in \nsB$, such that $|\nsA|\leqslant |\nsB|\leqslant |\nsA|+|L(\MB)|+\aleph_0$.
\end{theorem}

If the interpretation $\MA \rightsquigarrow \MB$ is regular, then the statement of Theorem~\ref{th:sk_el} can be strengthened.

\begin{theorem}[expected in the next paper] \label{pr:reg-int-elem-equiv}
Let $(\Gamma,\phi)\colon\MA\rightsquigarrow\MB$ be a regular interpretation and $\nsB$ be an $L(\MB)$\=/structure such that $\nsB\equiv\MB$. Then 
\begin{enumerate}[label=(\arabic*)]
\item for every tuple $\bar q \in \phi(\nsB)$ the algebraic structure $\Gamma(\nsB,\bar q)$ is well-defined and one has $\Gamma(\nsB,\bar q) \equiv \MA$;
\item moreover, if interpretation $(\Gamma,\phi)$ is self-homotopic then for all $\bar q_1,\bar q_2\in \phi(\nsB)$ one has $\Gamma(\nsB,\bar q_1)\simeq\Gamma(\nsB,\bar q_2)$, i.\,e., there exists an $L(\MA)$\=/structure $\nsA\equiv\MA$, such that $\nsA\simeq\Gamma(\nsB,\phi)$.
\end{enumerate}
\end{theorem}

In the assumptions of Theorems~\ref{th:sk_el} and~\ref{pr:reg-int-elem-equiv}, algebraic structures of the kind $\Gamma(\nsB,\bar q)$ are called {\em non-standard models of $\MA$} (relative to $\MB$). Regular bi-interpretability between $\MA$ and $\MB$ allows us to describe models of the elementary theory $\Th(\MA)$ by means of models of the elementary theory $\Th(\MB)$; that is, in this case, all models of $\Th(\MA)$ are non-standard relative to $\MB$ and the first-order classification problem for $\MA$ is solved.

\begin{theorem}[expected in the next paper]\label{th:equiv1}
Let $\MA$ and $\MB$ be regularly bi-interpretable in each other, so $\MA \simeq \Gamma(\MB,\phi)$ and $\MB\simeq\Delta(\MA,\psi)$. Then for any $L(\MA)$\=/structure $\nsA$ the following statements are equivalent:
\begin{enumerate}[label=(\arabic*)]
\item $\MA\equiv\nsA$;
\item $\nsA\simeq\Gamma(\nsB,\phi)$ for some $L(\MB)$\=/structure $\nsB\equiv \MB$.  Moreover, for any such $\nsB$ one has  $\nsB\simeq\Delta(\nsA,\psi)$, so  $\nsB$ is defined uniquely up to an isomorphism. 
\end{enumerate}
\end{theorem}

There is also an analogue of the theorem above for the case of regular invertible interpretation, but we omit it here.

\begin{example}
Let $R$ be an associative commutative ring with $1$. It was shown in~\cite{MS} that the special linear group $\SL_n(R)$, $n \geqslant 3,$ is regularly bi-interpretable with $R$. It follows from Theorem~\ref{th:equiv1} that for any group $H$ the following equivalence holds: 
$$
H\equiv \SL_n(R) \; \Longleftrightarrow\;  H \simeq \SL_n(\widetilde{R}) \ \mbox{ for some ring } \ \widetilde{R} \equiv R.
$$
\end{example}

\begin{example}
E.\,Bunina and P.\,Gvozdevsky in~\cite{BuninaGv} proved several results on regular interpretability of Chevalley groups with a ring $R$ as above. Hence, Theorem~\ref{th:equiv1} again provides a tool for solving the first-order classification problem for such groups. 
\end{example}

\begin{example}
Many groups are known to be regularly bi-interpretable with $\Z$: free me\-ta\-be\-lian~\cite{KMfreemetab}, Baumslag\,--\,Solitar groups~\cite{BS, Khelif}, wreath products $\Z \wr \Z$~\cite{Khelif}, and more generally, wreath products $\Z^m \wr \Z^n$~\cite{KMO}, etc. For all of them Theorem~\ref{th:equiv1} solves the first-order classification problem.
\end{example}

When we study non-standard models of $\MA$ by means of a regular interpretation $(\Gamma,\phi)\colon \MA\rightsquigarrow\MB$ into an absolutely rich algebraic structure $\MB$ (for example $\Z$, see Subsection~\ref{subsec:rich}), the explicit form of the interpretation code $\Gamma$ fades into the background, since the following result is true. 

\begin{theorem} [expected in the next paper]
Let a finitely generated structure $\MA$ in a finite language  $L(\MA)$ be regularly interpretable in an absolutely rich algebraic structure $\MB$ in two ways, as $\Gamma_1(\MB,\phi_1)$ and $\Gamma_2(\MB,\phi_2)$. If for an $L(\MB)$\=/structure $\nsB$ one has $\nsB \equiv \MB$, then the algebraic structures 
$\MA_1\simeq\Gamma_1(\nsB,\phi_1)$, $\MA_2\simeq\Gamma_2(\nsB,\phi_2)$ are  well-defined and isomorphic $\MA_1\simeq\MA_2$.
\end{theorem}

\subsection{Stability and saturation}\label{subsec:stable}

Let $\MA=\langle A; L(\MA)\rangle$ be an algebraic structure and $X\subseteq A$ be a subset. The {\em Stone space} $S^{\MA}_m(X)$ consists of all complete $m$\=/types over $X$, i.\,e., of all complete types $\tp(x_1,\ldots,x_m)$ of formulas in $L(\MA)\cup X$, which realize in some model of $\Th_X(\MA)$ (in other words, $\tp(x_1,\ldots,x_m)$ is finitely satisfiable in $\MA_X$)~\cite{Marker}. Let $\kappa$ be an infinite cardinal. Algebraic structure $\MA$ is called {\em $\kappa$\=/saturated} if for any $X\subseteq A$, $|X|<\kappa$, any integer $m$ every complete type $\tp(x_1,\ldots,x_m)\in S_m^{\MA}(X)$ is realized in $\MA$. An algebraic structure $\MA$ in a countable language $L(\MA)$ is called {\em $\kappa$\=/stable} if for any algebraic structure $\nsA\equiv\MA$, $\nsA=\langle \nsuA; L(\MA)\rangle$, and any subset $X\subseteq \nsuA$ if $|X|=\kappa$ then $|S^{\nsA}_m(X)|=\kappa$ for all integers $m$. If $\MA$ is $\omega$\=/stable, then it is $\kappa$\=/stable for any infinite cardinal $\kappa$~\cite[Theorem~4.2.18]{Marker}.

Note that for any interpretation $(\Gamma,\bar p,\mu_\Gamma)\colon\MA\rightsquigarrow\MB$, a subset $X\subseteq A$ and a type $\tp(x_1,\ldots,x_m)\in S_m^{\MA}(X)$, by the Reduction Theorem~\ref{le:interpr_corol}, there exists a type $\tp_\Gamma(\bar x_1,\ldots,\bar x_m)\in S_{m\cdot\dim\Gamma}^{\MB}(\bar\mu^{-1}_\Gamma(X)\cup \bar p)$, such that $\{\psi_{\Gamma,\bar p,\mu_\Gamma}\mid \psi\in \tp(x_1,\ldots,x_m)\}\subseteq \tp_\Gamma(\bar x_1,\ldots,\bar x_m)$; and for every non-equal types $\tp_1(x_1,\ldots,x_m), \tp_2(x_1,\ldots,x_m)\in S_m^{\MA}(X)$ any such $\tp_{1,\Gamma}(x_1,\ldots,x_m)$ and $\tp_{2,\Gamma}(x_1,\ldots,x_m)$ are also non-equal; in particularly, $|S_{m\cdot\dim\Gamma}^{\MB}(\bar\mu^{-1}_\Gamma(X)\cup \bar p)|\geqslant |S_m^{\MA}(X)|$.

\begin{theorem}[{\cite[Exercise~4.5.23]{Marker}}]
Let $\MA$ be interpretable in $\MB$ (with parameters). For any infinite cardinal $\kappa$ if $\MB$ is $\kappa$\=/saturated then $\MA$ is $\kappa$\=/saturated too.
\end{theorem}

\begin{proof}
It directly follows from the Reduction Theorem~\ref{le:interpr_corol}. 
\end{proof}

\begin{theorem}[{\cite[Exercise~4.5.23]{Marker}}]
Let $L(\MA)$ and $L(\MB)$ be countable languages and algebraic structures $\MA$ be interpretable in  $\MB$ (with parameters). For any infinite cardinal $\kappa$ if $\MB$ is  $\kappa$\=/stable then $\MA$ is $\kappa$\=/stable too.
\end{theorem}

\begin{proof}
Suppose that $\MA\simeq\Gamma(\MB,\bar p)$, $\kappa$ is an infinite cardinal, and $\MB$ is $\kappa$\=/stable. Take an arbitrary algebraic structure  $\nsA\equiv\MA$, $\nsA=\langle \nsuA; L(\MA)\rangle$, and a subset $X\subseteq \nsuA$ of cardinality $\kappa$. By Theorem~\ref{th:sk_el}, there exist algebraic structure $\nsB\equiv\MB$ and a tuple $\bar q\in\nsB$, such that $\nsA$ is elementarily embeddable in $\Gamma(\nsB,\bar q)$. Denote by $g\colon\nsA\to \Gamma(\nsB,\bar q)$ the corresponding elementary embedding and by $\mu_\Gamma$ a coordinate map of the interpretation $\Gamma(\nsB,\bar q)\rightsquigarrow \nsB$. Then $|\bar\mu^{-1}_\Gamma(g(X))\cup \bar q)|=|X|=\kappa$ and $|S_m^{\nsA}(X)|\leqslant|S_m^{\Gamma(\nsB,\bar q)}(g(X))|\leqslant|S_{m\cdot\dim\Gamma}^{\nsB}(\bar\mu^{-1}_\Gamma(g(X))\cup \bar q)|=\kappa$, hence $|S^{\nsA}_m(X)|=\kappa$ and $\MA$ is $\kappa$\=/stable.
\end{proof}

\subsection{Isotypeness, definability by types, primarity, homogeneity, atomicity}\label{subsec:types}

In 2012, B.\,Plotkin introduced the notion of isotypic algebraic structures~\cite{Plotkin10}, which plays an important part in logical geometries and logical equivalence of algebraic structures.

\begin{definition}
Algebraic structures $\MA$ and $\nsA$ in a language $L$ are called {\em isotypic}, if for every integer $n\geqslant 1$ every type $\tp(x_1,\ldots,x_n)$ in $L$ that is realized in $\MA$ is also realized in $\nsA$ and vice versa.     
\end{definition}

Note that, if $\MA$ and $\nsA$ are isotypic, then, in particular, $\MA\equiv\nsA$. Thus, Theorem~\ref{th:equiv1} gets the following upgrade at the level of isotypeness.

\begin{theorem}[expected in the next paper]
Let $\MA$ and $\MB$ be regularly bi-interpretable in each other, so $\MA \simeq \Gamma(\MB,\phi)$ and $\MB\simeq\Delta(\MA,\psi)$. Then if an $L(\MB)$\=/structure $\nsB$ is isotypic to $\MB$, then  the structure $\nsA\simeq\Gamma(\nsB,\phi)$ is isotypic to $\MA$. And conversely, if an $L(\MA)$\=/structure $\nsA$ is isotypic to $\MA$, then $\nsB\simeq\Delta(\nsA,\psi)$ is isotypic to $\MB$.
\end{theorem}

Actually, the regular bi-interpretability of $\MA$ and $\MB$ is not necessary for transferring the property of isotypeness between $\MA$ and $\MB$.  This can be achieved through a weaker property as described in the result below.

\begin{theorem}[expected in the next paper]
Let $\MA \simeq \Gamma(\MB,\phi)$ be a regular self-homotopic interpretation. Then if an $L(\MB)$\=/structure $\nsB$ is isotypic to $\MB$, then $\nsA\simeq\Gamma(\nsB,\phi)$ is isotypic to $\MA$. 
\end{theorem}

In 2018, A.\,Myasnikov and N.\,Romanovskii further developed the idea of isotypness~\cite{MR}. They introduced the following definition.

\begin{definition}
An algebraic structure $\MA=\langle A; L(\MA)\rangle$ is {\em definable by types}, if every algebraic structure $\nsA$ in the language $L(\MA)$ that is isotypic to $\MA$ is isomorphic to $\MA$. 
\end{definition}

\begin{theorem}[expected in the next paper]
Let $\MA$ and $\MB$ be regularly bi-interpretable in each other. Then $\MA$ is definable by types if and only if $\MB$ is definable by types.
\end{theorem}

Recall that an algebraic structure $\MA=\langle A; L(\MA)\rangle$ is {\em prime} if it elementarily embeds in any model of $\Th(\MA)$ (any algebraic structure $\nsA\equiv\MA$); $\MA$ is {\em atomic}, if every complete type $\tp(\bar a)$, $\bar a\in \MA$, is {\em principal} (or {\em isolated}), i.\,e., there exists a formula $\phi\in \tp(\bar a)$, such that $\Th(\MA)\cup\{\phi\}\vdash \tp(\bar a)$; and a countable algebraic structure $\MA$ is {\em homogeneous} if for any integer $n\geqslant 1$ and every two tuples $\bar a,\bar b\in A^n$ if $\tp(\bar a)=\tp(\bar b)$ then there is an automorphism $h\in\Aut(\MA)$, such that $h(\bar a)=\bar b$ (the latter is an equivalent form of the definition of homogeneity for countable structures~\cite[Proposition~4.2.13]{Marker}). It is known that an algebraic structure $\MA$ in a countable language $L(\MA)$ is prime if and only if it is countable and atomic, provided that $\Th(\MA)$ has infinite models~\cite[Theorem~4.2.8]{Marker}. Furthermore, if $\MA$ is countable and atomic, then it is homogeneous~\cite[Lemma~4.2.14]{Marker}. It is easy to see that algebraic structures $\Z$ and $\N$ are prime, therefore, atomic and homogeneous.

\begin{theorem}[expected in the next paper]
Let $\MA \simeq \Gamma(\MB,\phi)$ be a regular right invertible interpretation of $\MA$ in $\MB$. If $\MB$ is prime, then $\MA$ is prime too.
\end{theorem}

It was claimed in~\cite[Lemma~4.16]{KhMS} that a similar result holds for bi-interpretations with parameters, though it was proven only for regular bi-interpretations, so whether this is true for bi-interpretations with parameters is, in fact, an open question.

As we have seen, regular bi-interpretations offer several advantages. Next, we will explore the implications of bi-interpretation with parameters.

\subsection{Richness, i-rigidity and elimination of imaginaries}\label{subsec:rich}

Recall the notion of rich algebraic structures and how it is transferred in bi-interpretations (see~\cite{KhMS}).

\begin{definition}
An algebraic structure $\MA=\langle A;L(\MA)\rangle$ is \emph{rich} (correspondingly, \emph{absolutely rich}) if every formula $\phi(x_1, \ldots,x_n)$ of the weak second-order logic in  $L(\MA)$ is equivalent over $\MA$  to some first-order language formula $\phi^\ast(x_1, \ldots,x_n)$ in $L(\MA)\cup A$ (correspondingly, in $L(\MA)$).
\end{definition}

Algebraic structure $\Z$ is absolutely rich~\cite[Lemma~4.14]{KhMS}.

\begin{theorem}[{\cite[Theorem~4.7]{KhMS}}]\label{th11}
Let structures $\MA$ and $\MB$ of finite languages be bi-interpretable with parameters in each other. Then $\MA$ is rich if and only if $\MB$ is rich.
\end{theorem}

\begin{theorem}[{\cite[Lemma 4.15]{KhMS}}]\label{th12}
Let a structure $\MA$ in a finite language $L(\MA)$ be regularly half-absolute bi-interpretable in a structure $\MB$ in a finite language $L(\MB)$. Then if $\MB$ is absolutely rich, then $\MA$ is also absolutely rich.
\end{theorem}

We would like to mention that Theorems~\ref{th11} and~\ref{th12} hold in a more general context; we will discuss this in a future paper of the series.

\begin{theorem}[expected in the next paper]\label{th18}
    Suppose that $\MA$ is finitely generated algebraic structure in a finite language $L(\MA)$ and $\MB$ is a rich algebraic structure. Then any interpretations with parameters $(\Gamma_1,\bar p_1),(\Gamma_2,\bar p_2)\colon \MA\rightsquigarrow \MB$ are homotopic.
\end{theorem}

\begin{corollary}
    Every finitely generated rich algebraic structure in a finite language is i\=/rigid.
\end{corollary}

\begin{theorem}[expected in the next paper]\label{reg_th18}
Suppose that $\MA$ is a finitely generated algebraic structure in a finite language $L(\MA)$ and $\MB$ is an absolutely rich algebraic structure.
    Then any regular interpretations $(\Gamma_1,\phi_1),(\Gamma_2,\phi_2)\colon \MA\rightsquigarrow \MB$ are regularly homotopic. In particular, every regular interpretation $(\Gamma,\phi)\colon \MA\rightsquigarrow\MB$ is self-homotopic.
\end{theorem}

\begin{corollary}\label{cor:reg_i-rigid}
    Every finitely generated absolutely rich algebraic structure in a finite language is regularly i\=/rigid.
\end{corollary}

\begin{theorem}[expected in the next paper]
    Suppose that $\MA$ and $\MB$ are strongly bi-interpretable in each other. Then $\MA$ is i\=/rigid if and only if $\MB$ is i\=/rigid.
\end{theorem}

\begin{corollary}\label{th19}
Suppose that $\MA$ and $\MB$ are strongly bi-interpretable in each other and $\MA$ is i\=/rigid. Then for any interpretations $(\Gamma,\bar p)\colon \MA\rightsquigarrow\MB$ and $(\Delta,\bar q)\colon \MB\rightsquigarrow \MA$  $(\Gamma,\bar p;\Delta,\bar q)$ is a bi-interpretation between $\MA$ and $\MB$ .
\end{corollary}

\begin{theorem}[expected in the next paper]
Suppose that $\MA$ and $\MB$ are regularly bi-interpretable in each other. Then $\MA$ is regularly i\=/rigid if and only if $\MB$ is regularly i\=/rigid.
\end{theorem}

\begin{corollary}
If there exists an absolute regularly right-invertible interpretation of $\MA$ in $\Z$, then $\MA$ and $\Z$ are regularly strongly bi-interpretable in each other and $\MA$ is regularly i\=/rigid.
\end{corollary}

The concept of elimination of imaginaries (without parameters) is well-known in model theory (see~\cite{Hodges, Poizat}). We consider here its natural generalization,~--- elimination of imaginaries with parameters, and show that both of them behave well with respect to bi-interpretations.

\begin{definition}
Algebraic structure $\MA=\langle A;L(\MA)\rangle$ has {\em uniform elimination of imaginaries (without parameters)}, if for any natural number $n$ and any $0$\=/definable equivalence relation $\sim$ on $A^n$ there exists a formula $\psi(\bar x, \bar y)$ of the language $L(\MA)$ with free variables $x_1,\ldots,x_n,y_1,\ldots,y_m$, such that for any tuple $\bar a\in A^n$ there exists a unique tuple $\bar c\in A^m$ with 
$$
\bar a\sim \bar a^\prime \iff \MA\models\psi(\bar a^\prime, \bar c), \quad \bar a^\prime\in A^n,
$$
i.\,e., $\psi(A^n,\bar c)$ is the equivalence class of the element $\bar a$. We say that an algebraic structure $\MA$ has {\em uniform elimination of imaginaries with parameters}, if the algebraic structure $\MA_A$ has elimination of imaginaries without parameters.
\end{definition}

It is easy to see that the algebraic structure $\Z$ has uniform elimination of imaginaries without parameters.

\begin{theorem}[expected in the next paper]
If an algebraic structure $\MA$ has uniform elimination of imaginaries without parameters, then it has uniform elimination of imaginaries with parameters. 
\end{theorem}

\begin{theorem}[expected in the next paper]
Suppose that $\MA$ is strongly half-injectively bi-interpretable with $\MB$ (with parameters), i.\,e., interpretation $\MA\rightsquigarrow\MB$ is injective. If $\MA$ has uniform elimination of imaginaries with parameters, then $\MB$ has uniform elimination of imaginaries with parameters too. 
\end{theorem}

\begin{theorem}[expected in the next paper]
Suppose that $\MA$ is strongly half-injectively absolutely bi-interpretable with $\MB$. If $\MA$ has uniform elimination of imaginaries without parameters, then $\MB$ has uniform elimination of imaginaries without parameters. 
\end{theorem}

\subsection{Logical categories}\label{subsec:cat}

Another consequence of the bi-interpretation with parameters is the equivalence of logical categories that come from B.\,Plotkin's works~\cite{Plotkin6, Plotkin10, Plotkin5}. We will not give all the necessary definitions here, since this would be too long, but instead we mention one principal result and refer to~\cite{Th_int2} for details. Let us just say here that by $\PLS(\MA)$ we denote the category of all projective logical sets over an algebraic structure $\MA$ and by $\PDS(\MA)$ we denote the full subcategory of projective definable sets. {\em Logical sets} are actually the same as type-definable sets, and {\em projective sets} are their quotient sets by definable equivalence relations.

\begin{theorem}[expected in the next paper]
Let $\MA$ and $\MB$ be strongly bi-interpretable with parameters. Then the categories $\PLS(\MA)$ and $\PLS(\MB)$ of projective logical sets over $\MA$ and $\MB$ are equivalent, and their subcategories $\PDS(\MA)$ and $\PDS(\MB)$ of projective definable sets are equivalent too.
\end{theorem}

\subsection{Bi-interpretation with \texorpdfstring{$\Z$}{} and QFA}\label{subsec:Z}

The notion of quasi-finitely axiomatizability was introduced in 1986 by G.\,Ahlbrandt and M.\,Ziegler~\cite{AhZ}. For groups this notion was studied by F.\,Oger, G.\,Sabbagh~\cite{Oger, OS} and A.\,Nies~\cite{Nies1, Nies2}.

\begin{definition}
A finitely generated algebraic structure $\MA$ in a language $L(\MA)$ is called {\em QFA} ({\em quasi-finitely
axiomatizable}), if there exists a sentence $\psi\in \Th(\MA)$ such that for any finitely generated $L(\MA)$\=/structure $\nsA$ with $\psi\in \Th(\nsA)$ one has $\MA\simeq\nsA$. 
\end{definition}

For example, the ring of integers $\Z$ is QFA; the metabelian Baumslag\,--\,Solitar group $\BS(1,k)$, $k>1$, the Heisenberg group $\UT_3(\Z)$, the  restricted wreath product $\Z_p\!\wr\Z$ for a prime $p$ are QFA~\cite{Nies1}.

\begin{theorem}[\cite{Khelif}, {\cite[Theorem 7.14]{Nies2}}]\label{th17}
Suppose that $\MA$ is a finitely generated algebraic structure in a finite language $L(\MA)$, which is bi-interpretable with $\Z$. Then $\MA$ is QFA.
\end{theorem}

The result of Theorem~\ref{th17} was proved differently in the article~\cite[Theorem 4.9]{KhMS}, and in a more general formulation.

Recall that $\Z$ is rigid and i\=/rigid, so Lemma~\ref{rigid+i-rigid} holds for $\Z$. 

The following theorem summarizes a whole bunch of results related to bi-interpretation with $\Z$. 

\begin{theorem}
Let $\MA$ be a finitely generated algebraic structure in a finite language $L(\MA)$ and there exists a right invertible interpretation $(\Gamma,\emptyset)$ of $\MA$ in $\Z$ with right inverse $(\Delta,\bar q)$. Then 
\begin{enumerate}[label=\arabic*)]\setlength\itemsep{0em}
\item $(\Gamma,\emptyset;\Delta,\bar q)$ is a strong bi-interpretation between $\MA$ and $\Z$ for some coordinate maps $\mu_\Gamma$ and $\mu_\Delta$;
\item the elementary theory $\Th(\MA)$ is hereditarily undecidable;
\item $\MA$ is rich;
\item $\MA$ is QFA;
\item $\MA$ is i\=/rigid;
\item if $\Delta$ is injective then $\MA$ has elimination of imaginaries with parameters.
\end{enumerate}
\end{theorem}

Now let's take a look at regular bi-interpretation with $\Z$. Since $\Z$ is absolutely rich, then Theorem~\ref{reg_th18} holds for $\Z$, and by Corollary~\ref{cor:reg_i-rigid}, $\Z$ is regularly i\=/rigid, so Lemma~\ref{reg_rigid+i-rigid} holds for $\Z$ too. 

The simplest way to establish a regular bi-interpretation with $\Z$ is given by the following theorem.

\begin{theorem}[expected in the next paper]\label{th23}
Suppose that $\MA$ is a finitely generated algebraic structure in a finite language $L(\MA)$, and there exists a right invertible absolute interpretation $\Gamma\colon \MA\rightsquigarrow\Z$ with right inverse $(\Delta,\bar q)\colon\Z\rightsquigarrow\MA$. If there is a regular interpretation $(\Delta,\psi)\colon \MA\rightsquigarrow\Z$, such that $\bar q\in \psi(\MA)$, then there exists a parameter extension $(\Gamma_\param,\psi\wedge\delta)$ of $(\Delta,\psi)$, such that $(\Gamma,\emptyset;\Delta_\param,\psi\wedge\delta)$ is a strong regular bi-interpretation between $\MA$ and $\Z$ (without descriptors).
\end{theorem}

The following theorem brings together a collection of results concerning regular bi-interpretation with $\mathbb{Z}$.

\begin{theorem}
Let $\MA$ be a finitely generated algebraic structure in a finite language $L(\MA)$, and there exists a regularly right invertible interpretation $(\Gamma,\emptyset)\colon\MA\rightsquigarrow \Z$ with a regular right inverse $(\Delta,\psi)$. Then 
\begin{enumerate}[label=\arabic*)]\setlength\itemsep{0em}
\item $(\Gamma,\emptyset;\Delta,\psi)$ is a strong regular bi-interpretation between $\MA$ and $\Z$;
\item $\MA$ is absolutely rich;
\item $\MA$ is definable by types;
\item $\MA$ is prime, atomic, and homogeneous;
\item $\MA$ is  regularly i\=/rigid;
\item an $L(\MA)$\=/structure $\nsA$ is elementarily equivalent to $\MA$ if and only if $\nsA\simeq\Gamma(\nsZ)$ for some $\nsZ\equiv\Z$.
\end{enumerate}
\end{theorem}

The following fact is known~\cite{KhMS, Khelif, Nies2}, but we explain it differently.

\begin{lemma}\label{lem:UT}
The Heisenberg group $\UT_3(\Z)$ is not bi-interpretable with $\Z$.
\end{lemma}

\begin{proof}
Let's assume the opposite, that there exists a bi-interpretation between $\UT_3(\Z)$ and $\Z$, then, by Corollary~\ref{cor:rigid+i-rigid}, $\UT_3(\Z)$ and $\Z$ are strongly bi-interpretable. Denote by $\Gamma$ the standard interpretation of $\UT_3(\Z)$ in $\Z$ (see Example~\ref{ex:UT1}) and by $(\Delta,\bar q)$ the interpretation  of $\Z$ in $\UT_3(\Z)$ from Example~\ref{ex:UT2}. Due to Lemma~\ref{ex:UT3} there exists a regular interpretation $(\Delta,\psi)\colon \Z\rightsquigarrow \UT_3(\Z)$, such that $\bar q\in \psi(\UT_3(\Z))$. Therefore, by Theorem~\ref{th23}, $\UT_3(\Z)$ and $\Z$ are regularly bi-interpretable in each other. The last statement gives us that all models of the elementary theory $\Th(\UT_3(\Z))$ have a form $\Gamma(\nsZ)$, $\nsZ\equiv \Z$, i.\,e., they are all non-standard (see Theorem~\ref{th:equiv1}). But in~\cite{MS2} it was shown that there exist models of $\Th(\UT_3(\Z))$ that are not non-standard. So we obtain a contradiction with the initial assumption that $\Z$ and $\UT_3(\Z)$ are bi-interpretable in each other. 
\end{proof}

\end{document}